\documentclass[reqno]{amsart}
\usepackage[utf8]{inputenc}
\usepackage{xcolor}
\usepackage{verbatim}
\usepackage{enumitem}
\usepackage{amstext}
\usepackage{amsthm}
\usepackage{amssymb}
\usepackage{graphicx}
\usepackage{esint}
\usepackage[all]{xy}
\PassOptionsToPackage{normalem}{ulem}
\usepackage{ulem}
\usepackage[unicode=true,pdfusetitle,
 bookmarks=true,bookmarksnumbered=false,bookmarksopen=false,
 breaklinks=false,pdfborder={0 0 0},pdfborderstyle={},backref=false,colorlinks=true]
 {hyperref}
\hypersetup{
 pdfborderstyle=,pdfborderstyle={}}

\makeatletter

\providecolor{lyxadded}{rgb}{0,0,1}
\providecolor{lyxdeleted}{rgb}{1,0,0}

\DeclareRobustCommand{\lyxsout}[1]{\ifx\\#1\else\sout{#1}\fi}

\numberwithin{equation}{section}
\numberwithin{figure}{section}
\theoremstyle{plain}
\newtheorem{thm}{\protect\theoremname}
\theoremstyle{definition}
\newtheorem{defn}[thm]{\protect\definitionname}
\theoremstyle{plain}
\newtheorem{lem}[thm]{\protect\lemmaname}
\theoremstyle{remark}
\newtheorem{rem}[thm]{\protect\remarkname}
\theoremstyle{plain}
\newtheorem{cor}[thm]{\protect\corollaryname}
\theoremstyle{definition}
\newtheorem{example}[thm]{\protect\examplename}

\usepackage{pdfsync}
\usepackage{graphics}
\usepackage{epstopdf}
\usepackage[all]{xy}
\usepackage{amscd}
\usepackage{enumitem}
\usepackage{mathtools}
\usepackage{bbold}
\usepackage[utf8]{inputenc}
\setlist[enumerate]{leftmargin=*,label=(\roman*),align=left}

\makeatletter
\@namedef{subjclassname@2020}{%
  \textup{2020} Mathematics Subject Classification}
\makeatother

\newcommand{\xyR}[1]{ \makeatletter
\xydef@\xymatrixrowsep@{#1} \makeatother} 
\newcommand{\xyC}[1]{ \makeatletter
\xydef@\xymatrixcolsep@{#1} \makeatother} 
\entrymodifiers={++[ ][F]} 

\newcommand{\ra}{\longrightarrow}
\newcommand{\xra}[1]{\xrightarrow{\ \ #1\ \ }} 

\newcommand{\field}[1]{\mathbb{#1}}
\newcommand{\R}{\field{R}} 
\newcommand{\N}{\field{N}} 
\newcommand{\Z}{\ensuremath{\mathbb{Z}}} 

\newcommand{\D}{\mathcal{D}}

\newcommand{\eps}{\varepsilon} 
\renewcommand{\phi}{\varphi}
\newcommand{\diff}[1]{\ifmmode\mathchoice{\hbox{\rm d}#1}  
 {\hbox{\rm d}#1}  
 {\scalebox{0.75}{$\hbox{\rm d}#1$}}  
 {\scalebox{0.35}{$\hbox{\rm d}#1$}}  
 \fi} 

\newcommand{\abs}[2][\empty]{\ifx#1\empty\left|#2\right|%
\else#1\vert #2 #1\vert\fi}

\newcommand{\Coo}{\mbox{\ensuremath{\mathcal{C}}}^{\infty}} 
\DeclareMathOperator*{\Man}{{\bf Man}} 
\DeclareMathOperator{\Set}{{\bf Set}} 
\newcommand{\Gcinf}{\mathcal{G}\cinfty} 

\newcommand{\st}[1]{{#1^\circ}} 

\newcommand{\Rtil}{\widetilde \R} 
\newcommand{\gs}{\mathcal{G}^s} 
\newcommand{\ns}{\mathcal{N}^s} 

\newcommand{\thick}{\text{\rm th}} 
\newcommand{\sint}[1]{\langle#1\rangle} 
\newcommand{\Eball}{B^{{\scriptscriptstyle \text{\rm E}}}} 
\newcommand{\Fball}{B^{{\scriptscriptstyle \text{\rm F}}}} 
\newcommand{\sintF}[1]{\sint{#1}_{\scriptscriptstyle\text{\rm F}}}
\newcommand{\nrst}[1]{{#1}^\bullet} 
\newcommand{\csp}[1]{{\text{\rm c}}({#1})}
\newcommand{\fcmp}{\Subset_{\text{f}}}
\newcommand{\frontRise}[2]{\ifmmode\mathchoice{{\vphantom{#1}}^{\scalebox{0.6}{$#2$}}}  
 {{\vphantom{#1}}^{\scalebox{0.56}{$#2$}}}  
 {{\vphantom{#1}}^{\scalebox{0.47}{$#2$}}}  
 {{\vphantom{#1}}^{\scalebox{0.35}{$#2$}}}\fi} 
\newcommand{\RC}[1]{\frontRise{\R}{#1}\Rtil}
\newcommand{\frontRiseDown}[3]{\ifmmode\mathchoice{{\vphantom{#1}}^{\scalebox{0.6}{$#2$}}_{\scalebox{0.6}{$#3$}}}  
 {{\vphantom{#1}}^{\scalebox{0.56}{$#2$}}_{\scalebox{0.56}{$#3$}}}  
 {{\vphantom{#1}}^{\scalebox{0.47}{$#2$}}_{\scalebox{0.47}{$#3$}}}  
 {{\vphantom{#1}}^{\scalebox{0.35}{$#2$}}_{\scalebox{0.35}{$#3$}}}\fi} 
\newcommand{\RCud}[2]{\frontRiseDown{\R}{#1}{#2}\Rtil}
\newcommand{\rcrho}{\RC{\rho}}
\newcommand{\rti}{\RC{\rho}}
\newcommand{\gsf}{\frontRise{\mathcal{G}}{\rho}\mathcal{GC}^{\infty}}
\newcommand{\Ogsf}{\frontRise{\mathcal{O}}{\rho}\mathcal{OGC}^{\infty}}
\newcommand{\Sgsf}{\frontRise{\mathcal{S}}{\rho}\mathcal{SGC}^{\infty}}
\newcommand{\gluable}{\frontRise{\mathcal{G}}{\rho}\mathcal{G}\ell^{\infty}}
\newcommand{\Tgsf}{\frontRise{\text{T}}{\rho}\text{T}\mathcal{GC}^{\infty}}
\newcommand{\hyperN}[1]{\frontRise{\N}{#1}\widetilde{\N}}

\newcommand{\hyperNrho}{\hyperN{\rho}}
\newcommand{\nint}{\text{ni}}
\newcommand{\hyperlimarg}[3]{\mathchoice{\frontRise{\lim}{\raisebox{-0.05em}{$#1\hspace{-0.67em}$}}\lim_{#3\in \hyperN{#2}\,}}
{\frontRise{\lim}{#1\hspace{-0.25em}}\lim_{#3\in \hyperN{#2}\,}}
{\frontRise{\lim}{#1\hspace{-0.25em}}\lim_{#3\in \hyperN{#2}\,}}
{\frontRise{\lim}{#1\hspace{-0.25em}}\lim_{#3\in \hyperN{#2}\,}}}
\newcommand{\hyperlim}[2]{\hyperlimarg{#1}{#2}{n}}
\newcommand{\hyperlimrho}{\lim_{n\in \hyperN{\rho}\,}}

\newcommand{\subzero}{\subseteq_{0}}
\newcommand{\dom}[1]{\text{dom}(#1)}

\newcommand{\GI}{\frontRise{\mathcal{G}}{\rho}\mathcal{GI}} 

\newcommand{\ptind}{\displaystyle \mathop {\ldots\ldots\,}} 
\newcommand{\then}{\quad \Longrightarrow \quad} 

\newcommand{\CC}{\mathbb C}

\newcommand{\supp}{\mbox{stsupp}}
\newcommand{\cinfty}{{\mathcal C}^\infty}
\newcommand{\vphi}{\varphi}
\newcommand{\comp}{\Subset}
\newcommand{\esm}{{\mathcal E}_M}

\newcommand{\Om}{\Omega}

\newcommand{\restr}[2]{{{#1}|_{#2}}}

\newcommand{\sse}{\subseteq}

\newcommand{\co}[1]{{#1}^{c}}

\newcommand{\inv}[1]{{#1}^{-1}}

\newcommand{\norm}[2][\empty]{\ifx#1\empty\left\Vert#2\right\Vert%
\else#1\Vert #2 #1\Vert\fi}

\makeatother

\providecommand{\corollaryname}{Corollary}
\providecommand{\definitionname}{Definition}
\providecommand{\examplename}{Example}
\providecommand{\lemmaname}{Lemma}
\providecommand{\remarkname}{Remark}
\providecommand{\theoremname}{Theorem}

\begin{document}

\title[The Grothendieck topos of generalized smooth functions I]{A Grothendieck topos of generalized functions I: basic theory}
\author{Paolo Giordano \and Michael Kunzinger \and Hans Vernaeve}
\thanks{P.\ Giordano has been supported by grants P33538-N, P30407, P25311,
P25116 and P26859 of the Austrian Science Fund FWF}
\address{\textsc{Wolfgang Pauli Institute, Vienna, Austria}}
\email{paolo.giordano@univie.ac.at}
\thanks{M.\ Kunzinger has been supported by grant P30233 of the Austrian
Science Fund FWF}
\address{\textsc{University of Vienna, Austria}}
\email{michael.kunzinger@univie.ac.at}
\thanks{H.\ Vernaeve has been supported by grant 1.5.129.18N of the Research
Foundation Flanders FWO}
\address{\textsc{Ghent University, Belgium}}
\email{Hans.Vernaeve@UGent.be}
\begin{abstract}
The main aim of the present work is to arrive at a mathematical theory
close to the historically original conception of generalized functions,
i.e.~set theoretical functions defined on, and with values in, a
suitable ring of scalars and sharing a number of fundamental properties
with smooth functions, in particular with respect to composition and
nonlinear operations. This is how they are still used in informal
calculations in Physics. We introduce a category of generalized functions
as smooth set-theoretical maps on (multidimensional) points of a ring
of scalars containing infinitesimals and infinities. This category
extends Schwartz distributions. The calculus of these generalized
functions is closely related to classical analysis, with point values,
composition, non-linear operations and the generalization of several
classical theorems of calculus. Finally, we extend this category of
generalized functions into a Grothendieck topos of sheaves over a
concrete site. This topos hence provides a suitable framework for
the study of spaces and functions with singularities. In this first
paper, we present the basic theory; subsequent ones will be devoted
to the resulting theory of ODE and PDE. 
\end{abstract}

\maketitle
\tableofcontents{}

\section{\label{sec:Intro}Introduction: foundations of generalized functions
as set-theoretical maps}

The aim of the present work is to lay the foundations for a new approach
to the theory of generalized functions, so-called \emph{generalized
smooth functions} (GSF). In developing such a theory, various objectives
can be pursued, and our motivations mainly come from applications
in mathematical physics and nonlinear singular differential equations,
where the need for such a nonlinear theory is well-known (see, e.g.,
\cite{GMS09,SV06,Gsp09,Col92,Ba-Zo90,Raj82,Mik66,Bo-Sh59,Gon-Sca56}
for applications in mathematical physics, \cite{Sch54,Sch,Hor83,Eva10,GKOS,Col92}
for differential equations, and references therein).

In particular, our aim is to arrive at a mathematical theory close
to the historically original conception of generalized function, \cite{Dir26,Lau89,Kat-Tal12}:
in essence, the idea of authors such as Dirac, Cauchy, Poisson, Kirchhoff,
Helmholtz, Kelvin and Heaviside (who informally worked with ``numbers''
which also comprise infinitesimals and infinite scalars) was to view
generalized functions as certain types of smooth set-theoretical maps
obtained from ordinary smooth maps by introducing a dependence on
suitable infinitesimal or infinite parameters. We call this idea the
Cauchy-Dirac approach to generalized functions. For example, the density
of a Cauchy-Lorentz distribution with an infinitesimal scale parameter
was used by Cauchy to obtain classical properties which nowadays are
attributed to the Dirac delta, \cite{Kat-Tal12}. More generally,
in the GSF approach, generalized functions are seen as set-theoretical
functions defined on, and attaining values in, a suitable non-Archimedean
ring of scalars containing infinitesimals and infinities, as well
as sharing essential properties of ordinary smooth functions. In the
present work, we will develop this point of view, and prove that it
generalizes the mentioned Cauchy-Dirac approach. In our view, the
main benefits of this theory lie in a clarification of a number of
foundational issues in the theory of generalized functions, namely: 
\begin{enumerate}
\item GSF include all Schwartz distributions, see Thm.~\ref{thm:embeddingD'},
and Colombeau generalized functions, see \cite{GKV}. 
\item They allow nonlinear operations on generalized functions, Sec.~\ref{sec:GSF},
and to compose them unrestrictedly, Thm.~\ref{thm:GSFcategory}. 
\item GSF are simpler than standard approaches as they allow us to treat
generalized functions more closely to classical smooth functions.
In particular, they allow us to prove a number of analogues of fundamental
theorems of classical analysis: e.g., mean value theorem, intermediate
value theorem, extreme value theorem, Taylor's theorem, see Sec.~\ref{sec:Some-classical-theorems},
local and global inverse function theorem, \cite{GK17}, integrals
via primitives, Sec.~\ref{sec:Integral-calculus}, multidimensional
integrals, Sec.~\ref{sec:n-dimIntegral}, theory of compactly supported
functions, \cite{GK15}. Therefore, this approach to generalized functions
results in a flexible and rich framework which allows both the formalization
of calculations appearing in physics and the development of new applications
in mathematics and mathematical physics. 
\item Several results of the classical theory of calculus of variations
can be developed for GSF: the fundamental lemma, second variation
and minimizers, higher order Euler-Lagrange equations, D’Alembert
principle in differential form, a weak form of the Pontryagin maximum
principle, necessary Legendre condition, Jacobi fields, conjugate
points and Jacobi’s theorem, Noether’s theorem, see \cite{LeLuGi17,FGBL}. 
\item The closure with respect to composition leads to a solution concept
of differential equations close to the classical one. In the second
and third article of this series, we will introduce a non-Archimedean
version of the Banach fixed point theorem that is well suited for
spaces of GSF, a Picard-Lindelöf theorem for both ODE and PDE, results
about the maximal set of existence, Gronwall theorem, flux properties,
continuous dependence on initial conditions, full compatibility with
classical smooth solutions, etc., see \cite{Lu-Gi16,GiLu}.%
\end{enumerate}
Moreover, we think that a satisfactory theory of generalized functions
as used in mathematical physics should also provide an extension to
function spaces, possibly in a Cartesian closed category or, better,
in a topos. The use of a Cartesian closed category as a useful framework
for mathematics and mathematical physics can be motivated in several
ways: 
\begin{enumerate}
\item It is well known that a non trivial problem of the category $\Man$
of smooth manifolds is the absence of closure properties with respect
to interesting categorical operations such as the construction of
function spaces $\Man(M,N)$, subspaces, equalizers, etc., see \cite{Ba-Ho11,Be98,Che82,Fr-Kr88,Gio10c,Gio10e,Igl,KM,Koc,Lav,Mo-Re,Sou70,Sou81,Sou84}.
The search for a Cartesian closed category embedding $\Man$ is the
most widespread approach to solving this problem. 
\item In physics, the necessity to use infinite-dimensional spaces frequently
appears. A classical example is the space $\Man(M,N)$ of all smooth
mappings between two smooth manifolds $M$ and $N$, or some of its
subspaces, e.g.~the space of all the diffeomorphisms of a smooth
manifold. Typically, we are interested in infinite dimensional Lie
groups, because they appear, e.g., in the study of both compressible
and incompressible fluids, in magnetohydrodynamics, in plasma-dynamics
or in electrodynamics (see e.g.~\cite{Ab-Ma-Ra} and references therein).
It is also well established (see e.g.~\cite{Fr-Kr88,Gio10c,Gio09})
that Cartesian closedness is a desirable condition in the calculus
of variations. 
\item The convenient setting, \cite{KM,Fr-Kr88}, is the most advanced theory
of smooth spaces extending the theory of Banach manifolds. Some applications
of this notion to classical field theory can be found in \cite{Ab-Ma}.
In addition, several other approaches to a new notion of smooth space
have been motivated by problems of physics. For example, the notion
of diffeological space has been used in \cite{Sou70,Sou81,Sou84},
starting also from a variant of \cite{Che82}, to study quantization
of coadjoint orbits in infinite dimensional groups of diffeomorphisms.
Diffeological spaces form a Cartesian closed, complete, co-complete
quasi-topos, \cite{Igl,Ba-Ho11,Ko-Re04}. 
\end{enumerate}
For these reasons, we close this work by embedding our category of
generalized functions into a Grothendieck topos, see Sec.~\ref{sec:Gtopos}.

Finally, a theory of generalized functions for mathematical physics
frequently appears coupled with a theory of actual infinitesimals
and infinities (see e.g.~\cite{Hos06,Ko-Re04,Ko-Re03,GKOS}). This
is natural, since informal descriptions of these functions used in
many calculations in physics employ a language including infinitesimals
or infinities. Historically, it has turned out that approaches requiring
a substantial background knowledge in mathematical logic are only
reluctantly accepted by some physicists and mathematicians. Therefore,
even if sometimes they appear less powerful, theories that do not
need such knowledge (\cite{Gio10a,Sha99,C1}) are more easily accepted.
In the following section, we introduce the non-Archimedean ring of
scalars in a very natural way, without requiring any notions from
mathematical logic or ultrafilter set theory.

The structure of the paper is as follows: we first introduce the new
ring of scalars and its natural topology in Sec.~\ref{sec:The-ring-of-scalars};
in Sec.~\ref{sec:GSF} we define the notion of GSF and prove that
GSF are always continuous; we present the embedding of Schwartz distributions
and prove the closure of GSF with respect to composition (e.g.~we
study and graphically represent $\delta\circ\delta$); in Sec.~\ref{sec:Differential-calculus},
\ref{sec:Integral-calculus}, \ref{sec:Some-classical-theorems} we
study the differential calculus, the (1-dimensional) integral calculus,
and several related classical theorems. In Sec.~\ref{sec:n-dimIntegral},
we introduce multidimensional integration, with related convergence
theorems. Sheaf properties for GSF defined on different types of domains
are proved in Sec.~\ref{sec:Sheaf-properties}: they e.g.~allow
one to glue GSF defined on infinitesimal domains to get a global GSF
defined on a finite or even on an unbounded domain. Finally, in Sec.~\ref{sec:Gtopos}
we construct the Grothendieck topos of generalized functions, including
a full introduction of all the necessary preliminaries. Throughout
the paper, several theorems will treat the connections of notions
related to GSF to the corresponding classical notions, in case the
latter can at least be formulated. Even if other papers about GSF
already appeared in the literature (see \cite{GKV,GK17,LeLuGi17}),
this is the first one where all these basic results (and several others)
are presented with the related proofs.

The paper is completely self-contained: only a basic knowledge of
Schwartz distribution theory and the concepts of category, functor
and natural transformation are needed.

\section{\label{sec:The-ring-of-scalars}The ring of scalars and its topologies}

Exactly as real numbers can be seen as equivalence classes of sequences
$(q_{n})_{n\in\N}$ of rationals\footnote{In the naturals $\N=\{0,1,2,\ldots\}$ we include zero.},
it is very natural to consider a non-Archimedean extension of $\R$
defined by a quotient ring $\mathcal{R}/\sim$, where $\mathcal{R}\subseteq\R^{I}$.
Here $\mathcal{R}$ is a subalgebra of nets $(x_{\eps})_{\eps\in I}\in\R^{I}$
defined on a directed set $(I,\le)$, and with pointwise algebraic
operations. For simplicity and for historical reasons, instead of
$I=\N$, we consider $I:=(0,1]$, corresponding to $\eps\to0^{+}$,
$\eps\in I$, but any other directed set can be used instead of $I$.
In this work, we will denote $\eps$-dependent nets simply by $(x_{\eps}):=(x_{\eps})_{\eps\in I}$,
and the corresponding equivalence class simply by $[x_{\eps}]:=[(x_{\eps})]_{\sim}\in\mathcal{R}/\sim$.
We aim at constructing the quotient ring $\Rtil:=\mathcal{R}/\sim$
so that it contains infinitesimals and infinities. The following observation
points to a natural way of achieving this goal. Let us assume that
$[z_{\eps}]=0\in\Rtil$ and $[J_{\eps}]\in\Rtil$ is generated by
an infinite net $(J_{\eps})$, i.e.~such that $\lim_{\eps\to0^{+}}\left|J_{\eps}\right|=+\infty$.
Then we would have 
\begin{align}
[z_{\eps}]\cdot[J_{\eps}] & =0\cdot[J_{\eps}]=0\nonumber \\
 & =[z_{\eps}\cdot J_{\eps}].\label{eq:zeroInf}
\end{align}
Finally, let us assume that 
\begin{equation}
\forall[w_{\eps}]\in\Rtil:\ [w_{\eps}]=0\ \Rightarrow\ \lim_{\eps\to0^{+}}w_{\eps}=0.\label{eq:infinitesimalsAreAccessible}
\end{equation}
Under these assumptions, \eqref{eq:zeroInf} yields $\lim_{\eps\to0^{+}}z_{\eps}\cdot J_{\eps}=0$,
and hence 
\begin{equation}
\exists\eps_{0}\in I\,\forall\eps\in(0,\eps_{0}]:\ |z_{\eps}|\le\left|J_{\eps}^{-1}\right|.\label{eq:ColNec}
\end{equation}
Consequently, \emph{the nets $(z_{\eps})$ representing $0$, i.e.~such
that $(z_{\eps})\sim0$, must be dominated by the reciprocals of every
infinite number $[J_{\eps}]\in\Rtil$}. It is not hard to prove that
if every infinite net $(J_{\eps})$ is in the subalgebra $\mathcal{R}$,
then \eqref{eq:ColNec} implies that the equivalence relation $\sim$
must be trivial: 
\begin{equation}
\exists\eps_{0}\in I\,\forall\eps\in(0,\eps_{0}]:\ z_{\eps}=0.\label{eq:Schmieden-Laugwitz}
\end{equation}
This situation corresponds to the Schmieden-Laugwitz model, \cite{Sc-La}.

If we do not want to have the trivial model \eqref{eq:Schmieden-Laugwitz},
we can hence either negate the natural property \eqref{eq:infinitesimalsAreAccessible}
(this is the case of nonstandard analysis; see \cite{CKKR} for more
details) or to restrict the class of all the nets $(J_{\eps})$ generating
infinite numbers in $\Rtil$. Since we want to start from a subalgebra
$\mathcal{R}\subseteq\R^{I}$, a \emph{first} natural idea is to consider
the following class of infinite nets 
\begin{equation}
\mathcal{I}:=\left\{ (\eps^{-a})\mid a\in\R_{>0}\right\} .\label{eq:polynomialInfinities}
\end{equation}
and hence to consider the subalgebra $\mathcal{R}\subseteq\R^{I}$
containing nets $(b_{\eps})\in\R^{I}$ bounded by some $(J_{\eps})\in\mathcal{I}$.
This idea is generalized in the following definition, where we take
exactly \eqref{eq:ColNec} as the widest possible definition of $(z_{\eps})\sim0$: 
\begin{defn}
\label{def:RCGN}Let $\rho=(\rho_{\eps})\in(0,1]^{I}$ be a net such
that $(\rho_{\eps})\to0$ as $\eps\to0^{+}$ (in the following, such
a net will be called a \emph{gauge}), then 
\begin{enumerate}
\item $\mathcal{I}(\rho):=\left\{ (\rho_{\eps}^{-a})\mid a\in\R_{>0}\right\} $
is called the \emph{asymptotic gauge} generated by $\rho$. 
\item If $\mathcal{P}(\eps)$ is a property of $\eps\in I$, we use the
notation $\forall^{0}\eps:\,\mathcal{P}(\eps)$ to denote $\exists\eps_{0}\in I\,\forall\eps\in(0,\eps_{0}]:\,\mathcal{P}(\eps)$.
We can read $\forall^{0}\eps$ as \emph{``for $\eps$ small''}. 
\item We say that a net $(x_{\eps})\in\R^{I}$ \emph{is $\rho$-moderate},
and we write $(x_{\eps})\in\R_{\rho}$ if 
\[
\exists(J_{\eps})\in\mathcal{I}(\rho):\ x_{\eps}=O(J_{\eps})\text{ as }\eps\to0^{+},
\]
i.e., if 
\[
\exists N\in\N\,\forall^{0}\eps:\ |x_{\eps}|\le\rho_{\eps}^{-N}.
\]
\item Let $(x_{\eps})$, $(y_{\eps})\in\R^{I}$, then we say that $(x_{\eps})\sim_{\rho}(y_{\eps})$
if 
\[
\forall(J_{\eps})\in\mathcal{I}(\rho):\ x_{\eps}=y_{\eps}+O(J_{\eps}^{-1})\text{ as }\eps\to0^{+},
\]
that is if 
\[
\forall n\in\N\,\forall^{0}\eps:\ |x_{\eps}-y_{\eps}|\le\rho_{\eps}^{n}.
\]
This is a congruence relation on the ring $\R_{\rho}$ of moderate
nets with respect to pointwise operations, and we can hence define
\[
\RC{\rho}:=\R_{\rho}/\sim_{\rho},
\]
which we call \emph{Robinson-Colombeau ring of generalized numbers},
\cite{Rob73,C1}. 
\item In particular, if the gauge $\rho=(\rho_{\eps})$ is non-decreasing,
then we say that $\rho$ is a \emph{monotonic gauge}. Clearly, considering
a monotonic gauge narrows the class of moderate nets: e.g.~if $\lim_{\eps\to\frac{1}{k}}x_{\eps}=+\infty$
for all $k\in\N_{>0}$, then $(x_{\eps})\notin\R_{\rho}$ for any
monotonic gauge $\rho$.
\end{enumerate}
\end{defn}

In the following, $\rho$ will always denote a net as in Def.~\ref{def:RCGN},
even if we will sometimes omit the dependence on the infinitesimal
$\rho$, when this is clear from the context. We will see below that
we can choose $\rho$ e.g.~depending on the class of differential
equations we need to solve for the generalized functions we are going
to introduce.

We can also define an order relation on $\RC{\rho}$ by writing $[x_{\eps}]\le[y_{\eps}]$
if there exists $(z_{\eps})\in\R^{I}$ such that $(z_{\eps})\sim_{\rho}0$
(we then say that $(z_{\eps})$ is \emph{$\rho$-negligible}) and
$x_{\eps}\le y_{\eps}+z_{\eps}$ for $\eps$ small. Equivalently,
we have that $x\le y$ if and only if there exist representatives
$[x_{\eps}]=x$ and $[y_{\eps}]=y$ such that $x_{\eps}\le y_{\eps}$
for all $\eps$. The following result follows directly from the previous
definitions: 
\begin{thm}
\label{thm:CGN-Ring}$\RC{\rho}$ is a partially ordered ring. The
real numbers $r\in\R$ are embedded in $\RC{\rho}$ by viewing them
as constant nets $[r]\in\RC{\rho}$. 
\end{thm}

Although the order $\le$ is not total, we still have the possibility
to define the infimum $[x_{\eps}]\wedge[y_{\eps}]:=[\min(x_{\eps},y_{\eps})]$,
and analogously the supremum function $[x_{\eps}]\vee[y_{\eps}]:=\left[\max(x_{\eps},y_{\eps})\right]$
and the absolute value $|[x_{\eps}]|:=[|x_{\eps}|]\in\RC{\rho}$.
Henceforth, we will also use the customary notation $\RC{\rho}^{*}$
for the set of invertible generalized numbers.

As in every non-Archimedean ring, we have the following 
\begin{defn}
\label{def:nonArchNumbs}Let $x\in\RC{\rho}$ be a generalized number,
then 
\begin{enumerate}
\item $x$ is \emph{infinitesimal} if $|x|\le r$ for all $r\in\R_{>0}$.
If $x=[x_{\eps}]$, this is equivalent to $\lim_{\eps\to0^{+}}x_{\eps}=0$.
We write $x\approx y$ if $x-y$ is infinitesimal, and $D_{\infty}:=\left\{ h\in\rti\mid h\approx0\right\} $
for the set of all infinitesimals. 
\item $x$ is \emph{infinite} if $|x|\ge r$ for all $r\in\R_{>0}$. If
$x=[x_{\eps}]$, this is equivalent to $\lim_{\eps\to0^{+}}\left|x_{\eps}\right|=+\infty$. 
\item $x$ is \emph{finite} if $|x|\le r$ for some $r\in\R_{>0}$. 
\end{enumerate}
\end{defn}

\noindent For example, setting $\diff{\rho}:=[\rho_{\eps}]\in\RC{\rho}$,
we have that $\diff{\rho}^{n}\in\RC{\rho}$, $n\in\N_{>0}$, is an
invertible infinitesimal, whose reciprocal is $\diff{\rho}^{-n}=[\rho_{\eps}^{-n}]$,
which is necessarily a positive infinite number. Of course, in the
ring $\RC{\rho}$ there exist generalized numbers which are not in
any of the three classes of Def.~\ref{def:nonArchNumbs}, like e.g.~$x_{\eps}=\frac{1}{\eps}\sin\left(\frac{1}{\eps}\right)$. 
\begin{defn}
~ 
\begin{enumerate}
\item If $\mathcal{P}\left\{ (x_{\eps})\right\} $ is a property of $(x_{\eps})\in\R_{\rho}^{n}$,
then we also use the abbreviations: 
\begin{align*}
\forall[x_{\eps}]\in X:\ \mathcal{P}\left\{ (x_{\eps})\right\} \quad & :\iff\quad\forall(x_{\eps})\in\R_{\rho}^{n}:\ [x_{\eps}]\in X\ \Rightarrow\ \mathcal{P}\left\{ (x_{\eps})\right\} \\
\exists[x_{\eps}]\in X:\ \mathcal{P}\left\{ (x_{\eps})\right\} \quad & :\iff\quad\exists(x_{\eps})\in\R_{\rho}^{n}:\ [x_{\eps}]\in X,\ \mathcal{P}\left\{ (x_{\eps})\right\} .
\end{align*}
For example, if $X=\{x\}\subseteq\rcrho^{n}$, then $\forall[x_{\eps}]=x:\ \mathcal{P}\left\{ (x_{\eps})\right\} $
means that the property holds for all representatives of $x$, and
$\exists[x_{\eps}]=x:\ \mathcal{P}\left\{ (x_{\eps})\right\} $ means
that the same property holds for some representative of $x$. 
\item Our notations for intervals are: $[a,b]:=\{x\in\RC{\rho}\mid a\le x\le b\}$,
$[a,b]_{\R}:=[a,b]\cap\R$, and analogously for segments $[x,y]:=\left\{ x+r\cdot(y-x)\mid r\in[0,1]\right\} \subseteq\RC{\rho}^{n}$
and $[x,y]_{\R^{n}}=[x,y]\cap\R^{n}$. 
\item \label{enu:subPoint} For subsets $J,K\subseteq I$ we write $K\subzero J$
if $0$ is an accumulation point of $K$ and $K\sse J$ (we read it
as: $K$ \emph{is co-final in $J$}). For any $J\subzero I$, the
constructions introduced so far can be repeated with nets $(x_{\eps})_{\eps\in J}$.
We indicate this by using the symbol $\rcrho^{n}|_{J}$. If $K\subzero J$,
$x\in\rcrho^{n}|_{J}$ and $x'\in\rcrho^{n}|_{K}$, then $x'$ is
called a \emph{subpoint} of $x$, denoted as $x'\subseteq x$, if
there exist representatives $(x_{\eps})_{\eps\in J}$, $(x'_{\eps})_{\eps\in K}$
of $x$, $x'$ such that $x'_{\eps}=x_{\eps}$ for all $\eps\in K$.
In this case we write $x'=x|_{K}$, $\dom{x'}:=K$, and the restriction
$(-)|_{K}:\rcrho^{n}\ra\rcrho^{n}|_{K}$ is a well defined operation.
In general, for $X\sse\rcrho^{n}$ we set $X|_{J}:=\{x|_{J}\in\rcrho^{n}|_{J}\mid x\in X\}$.
Finally note that 
\[
\left(\neg\forall^{0}\eps:\ \mathcal{P}\left\{ x_{\eps}\right\} \right)\ \iff\ \exists J\subzero I\,\forall\eps\in J:\ \neg\mathcal{P}\left\{ x_{\eps}\right\} .
\]
\end{enumerate}
\end{defn}

\subsection{Topologies on $\RC{\rho}^{n}$}

On the $\RC{\rho}$-module $\RC{\rho}^{n}$ we can consider the natural
extension of the Euclidean norm, i.e.~$|[x_{\eps}]|:=[|x_{\eps}|]\in\RC{\rho}$,
where $[x_{\eps}]\in\RC{\rho}^{n}$. Even if this generalized norm
takes values in $\RC{\rho}$, it shares some essential properties
with classical norms: 
\begin{align*}
 & |x|=x\vee(-x)\\
 & |x|\ge0\\
 & |x|=0\Rightarrow x=0\\
 & |y\cdot x|=|y|\cdot|x|\\
 & |x+y|\le|x|+|y|\\
 & ||x|-|y||\le|x-y|.
\end{align*}

\noindent It is therefore natural to consider on $\RC{\rho}^{n}$
topologies generated by balls defined by this generalized norm and
a set of radii: 
\begin{defn}
\label{def:setOfRadii}We say that $\mathfrak{R}$ is a \emph{set
of radii} if 
\begin{enumerate}
\item $\mathfrak{R}\subseteq\RC{\rho}_{\ge0}^{*}$ is a non-empty subset
of positive invertible generalized numbers. 
\item For all $r$, $s\in\mathfrak{R}$ the infimum $r\wedge s\in\mathfrak{R}$. 
\item $k\cdot r\in\mathfrak{R}$ for all $r\in\mathfrak{R}$ and all $k\in\R_{>0}$. 
\end{enumerate}
\noindent Moreover, if $\mathfrak{R}$ is a set of radii and $x$,
$y\in\RC{\rho}$, then: 
\begin{enumerate}[resume]
\item We write $x<_{\mathfrak{R}}y$ if $\exists r\in\mathfrak{R}:\ r\le y-x$. 
\item $B_{r}^{\mathfrak{R}}(x):=\left\{ y\in\RC{\rho}^{n}\mid\left|y-x\right|<_{\mathfrak{R}}r\right\} $
for any $r\in\mathfrak{R}$. 
\item $\Eball_{\rho}(x):=\{y\in\R^{n}\mid|y-x|<\rho\}$, for any $\rho\in\R_{>0}$,
denotes an ordinary Euclidean ball in $\R^{n}$. 
\end{enumerate}
\end{defn}

\noindent For example, $\RC{\rho}_{\ge0}^{*}$ and $\R_{>0}$ are
sets of radii. 
\begin{lem}
\label{lem:<R}Let $\mathfrak{R}$ be a set of radii and $x$, $y$,
$z\in\RC{\rho}$, then 
\begin{enumerate}
\item \label{enu:NoReflex}$\neg\left(x<_{\mathfrak{R}}x\right)$. 
\item \label{enu:transitive}$x<_{\mathfrak{R}}y$ and $y<_{\mathfrak{R}}z$
imply $x<_{\mathfrak{R}}z$. 
\item \label{enu:radiiArePositive}$\forall r\in\mathfrak{R}:\ 0<_{\mathfrak{R}}r$. 
\end{enumerate}
\end{lem}

\begin{proof}
\ref{enu:NoReflex}: $x<_{\mathfrak{R}}x$ would imply $r\le0$ for
some $r\in\mathfrak{R}\subseteq\RC{\rho}_{\ge0}^{*}$. But then $r^{-1}r=1\le0$.

\noindent \ref{enu:transitive}: If $r\le y-x$ and $s\le z-y$ for
$r$, $s\in\mathfrak{R}$, then $2(r\wedge s)\le r+s\le z-x$.

\noindent \ref{enu:radiiArePositive}: In fact, we have $0<_{\mathfrak{R}}r$
if and only if $s\le r$ for some $s\in\mathfrak{R}$. 
\end{proof}
\noindent The relation $<_{\mathfrak{R}}$ has better topological
properties as compared to the usual strict order relation $x\le y$
and $x\ne y$ (a relation that we will therefore never use) because
of the following result: 
\begin{thm}
\label{thm:intersectionBalls}The set of balls $\left\{ B_{r}^{\mathfrak{R}}(x)\mid r\in\mathfrak{R},\ x\in\RC{\rho}^{n}\right\} $
generated by a set of radii $\mathfrak{R}$ is a base for a topology
on $\RC{\rho}^{n}$. 
\end{thm}

\begin{proof}
It suffices to consider $z\in B_{r}^{\mathfrak{R}}(x)\cap B_{s}^{\mathfrak{R}}(y)$
and to prove that $B_{\nu}^{\mathfrak{R}}(z)\subseteq B_{r}^{\mathfrak{R}}(x)\cap B_{s}^{\mathfrak{R}}(y)$
for some $\nu\in\mathfrak{R}$. The proof is essentially a reformulation
of the classical proof in metric spaces. In fact, we have $\bar{r}\le r-|x-z|$
and $\bar{s}\le s-|y-z|$ for some $\bar{r}$, $\bar{s}\in\mathfrak{R}$.
Set $\nu:=\bar{r}\wedge\bar{s}\in\mathfrak{R}$. The inequality $|w-z|<_{\mathfrak{R}}\nu$
implies $\sigma\le\nu-|w-z|$ for some $\sigma\in\mathfrak{R}$. Therefore,
$|w-x|\le|w-z|+|z-x|\le\nu-\sigma+r-\bar{r}$ and thereby $\sigma\le\bar{r}+\sigma-\nu\le r-|w-x|$,
i.e.~$|w-x|<_{\mathfrak{R}}r$. This proves that $B_{\nu}^{\mathfrak{R}}(z)\subseteq B_{r}^{\mathfrak{R}}(x)$,
and the other inclusion follows analogously. 
\end{proof}
Henceforth, we will only consider the sets of radii $\RC{\rho}_{\ge0}^{*}$
and $\R_{>0}$. The topology generated in the former case is called
\emph{sharp topology}, whereas the latter is called \emph{Fermat topology}.
We will call \emph{sharply open set} any open set in the sharp topology,
and \emph{large open set} any open set in the Fermat topology; clearly,
the latter is coarser than the former. Let us note explicitly that
taking an infinitesimal radius $r\in\RC{\rho}_{\ge0}^{*}$ we can
consider infinitesimal neighborhoods of $x\in\RC{\rho}^{n}$ in the
sharp topology. Of course, this is not possible in the Fermat topology.
The existence of infinitesimal neighborhoods implies that the sharp
topology induces the discrete topology on $\R$, see \cite{GK13b}.
The necessity to consider infinitesimal neighborhoods occurs in any
theory containing continuous generalized functions which have infinite
derivatives. Indeed, from the mean value theorem Thm.~\ref{thm:classicalThms}.\ref{enu:meanValue}
below, we have $f(x)-f(x_{0})=f'(c)\cdot(x-x_{0})$ for some $c\in[x,x_{0}]$.
Therefore, we have $f(x)\in B_{r}(f(x_{0}))$, for a given $r\in\RC{\rho}_{>0}$,
if and only if $|x-x_{0}|\cdot|f'(c)|<r$, which yields an infinitesimal
neighborhood of $x_{0}$ in case $f'(c)$ is infinite; see \cite{GK13b,GK15}
for precise statements and proofs corresponding to this intuition.
By an innocuous abuse of language, we write $x<y$ instead of $x<_{\RC{\rho}_{\ge0}^{*}}y$
and $x<_{\R}y$ instead of $x<_{\R_{>0}}y$. For example, $\RC{\rho}_{\ge0}^{*}=\RC{\rho}_{>0}$.
We will simply write $B_{r}(x)$ to denote an open ball in the sharp
topology and $\Fball_{r}(x)$ for an open ball in the Fermat topology.
Proceeding by contradiction, it is not difficult to prove that the
sharp topology on $\rcrho^{n}$ is Hausdorff and that the set of all
the infinitesimals $D_{\infty}$ is a clopen set; moreover, as will
be proved more generally in \cite{Lu-Gi16}, this topology is also
Cauchy complete.

The following result is useful in dealing with positive and invertible
generalized numbers. 
\begin{lem}
\label{lem:mayer} Let $x\in\RC{\rho}$. Then the following are equivalent: 
\begin{enumerate}
\item \label{enu:positiveInvertible}$x$ is invertible and $x\ge0$, i.e.~$x>0$. 
\item \label{enu:strictlyPositive}For each representative $(x_{\eps})\in\R_{\rho}$
of $x$ we have $\forall^{0}\eps:\ x_{\eps}>0$. 
\item \label{enu:greater-i_epsTom}For each representative $(x_{\eps})\in\R_{\rho}$
of $x$ we have $\exists m\in\N\,\forall^{0}\eps:\ x_{\eps}>\rho_{\eps}^{m}$. 
\item \label{enu:There-exists-a}There exists a representative $(x_{\eps})\in\R_{\rho}$
of $x$ such that $\exists m\in\N\,\forall^{0}\eps:\ x_{\eps}>\rho_{\eps}^{m}$. 
\end{enumerate}
\end{lem}

\begin{proof}
\ref{enu:positiveInvertible} $\Rightarrow$ \ref{enu:strictlyPositive}:
Since $x$ is positive, we can find a representative $[x_{\eps}]=x$
such that $x_{\eps}\ge0$ for all $\eps$. But $x$ is also invertible,
so for all $\eps$ we can also write $x_{\eps}y_{\eps}=1+z_{\eps}$,
where $(z_{\eps})\sim_{\rho}0$ is a negligible net. By contradiction,
assume that $x_{\eps_{k}}\le0$ for each $k\in\N$, where $(\eps_{k})_{k\in\N}\to0^{+}$.
Then $x_{\eps_{k}}=0$ and hence $x_{\eps_{k}}y_{\eps_{k}}=0=1+z_{\eps_{k}}\to1$
for $k\to+\infty$, which is a contradiction.

\noindent \ref{enu:strictlyPositive} $\Rightarrow$ \ref{enu:greater-i_epsTom}:
Assume that there exists a representative $[x_{\eps}]=x$ such that
$x_{\eps_{k}}\le\rho_{\eps_{k}}^{k}$ for each $k\in\N$, where $(\eps_{k})_{k\in\N}\to0^{+}$
monotonically. We then define a $\rho$-moderate net by $\hat{x}_{\eps}:=0$
if $\eps=\eps_{k}$ and $\hat{x}_{\eps}:=x_{\eps}$ otherwise. For
each $n\in\N$, if $k$ is sufficiently big, we have $\left|x_{\eps_{k}}-\hat{x}_{\eps_{k}}\right|\le\rho_{\eps_{k}}^{k}\le\rho_{\eps_{k}}^{n}$.
This implies that $(x_{\eps})\sim_{\rho}(\hat{x}_{\eps})$. Therefore
$(\hat{x}_{\eps})$ is another representative of $x$, which contradicts
\ref{enu:strictlyPositive} by construction.

\noindent \ref{enu:greater-i_epsTom} $\Rightarrow$ \ref{enu:positiveInvertible}:
By assumption, $\lim_{\eps\to0^{+}}\rho_{\eps}=0^{+}$. This and \ref{enu:greater-i_epsTom}
yield that $x_{\eps}>\rho_{\eps}^{m}>0$ for $\eps$ small, say for
$\eps\le\eps_{0}$. Therefore, $0<y_{\eps}:=x_{\eps}^{-1}\le\rho_{\eps}^{-m}$
for $\eps\le\eps_{0}$ (and $y_{\eps}$ arbitrarily defined elsewhere)
is $\rho$-moderate and hence it is a representative of the inverse
of $x$.

Finally, \ref{enu:greater-i_epsTom} implies \ref{enu:There-exists-a}
for logical reasons, and \ref{enu:There-exists-a} implies \ref{enu:positiveInvertible}
because $\rho_{\eps}>0$. 
\end{proof}

\subsection{Open, closed and bounded sets generated by nets}

A natural way to obtain sharply open, closed and bounded sets in $\RC{\rho}^{n}$
is by using a net $(A_{\eps})$ of subsets $A_{\eps}\subseteq\R^{n}$.
We have two ways of extending the membership relation $x_{\eps}\in A_{\eps}$
to generalized points $[x_{\eps}]\in\RC{\rho}^{n}$ (cf.\ \cite{Ob-Ve,GKV}): 
\begin{defn}
\label{def:internalStronglyInternal}Let $(A_{\eps})$ be a net of
subsets of $\R^{n}$, then 
\begin{enumerate}
\item $[A_{\eps}]:=\left\{ [x_{\eps}]\in\RC{\rho}^{n}\mid\forall^{0}\eps:\,x_{\eps}\in A_{\eps}\right\} $
is called the \emph{internal set} generated by the net $(A_{\eps})$. 
\item Let $(x_{\eps})$ be a net of points of $\R^{n}$, then we say that
$x_{\eps}\in_{\eps}A_{\eps}$, and we read it as $(x_{\eps})$ \emph{strongly
belongs to $(A_{\eps})$}, if 
\begin{enumerate}[label=(\alph*)]
\item $\forall^{0}\eps:\ x_{\eps}\in A_{\eps}$. 
\item If $(x'_{\eps})\sim_{\rho}(x_{\eps})$, then also $x'_{\eps}\in A_{\eps}$
for $\eps$ small. 
\end{enumerate}
\noindent Moreover, we set $\sint{A_{\eps}}:=\left\{ [x_{\eps}]\in\RC{\rho}^{n}\mid x_{\eps}\in_{\eps}A_{\eps}\right\} $,
and we call it the \emph{strongly internal set} generated by the net
$(A_{\eps})$. 
\item We say that the internal set $K=[A_{\eps}]$ is \emph{sharply bounded}
if there exists $M\in\RC{\rho}_{>0}$ such that $K\subseteq B_{M}(0)$. 
\item Finally, we say that the $(A_{\eps})$ is a \emph{sharply bounded
net} if there exists $N\in\R_{>0}$ such that $\forall^{0}\eps\,\forall x\in A_{\eps}:\ |x|\le\rho_{\eps}^{-N}$. 
\end{enumerate}
\end{defn}

\noindent Therefore, $x\in[A_{\eps}]$ if there exists a representative
$[x_{\eps}]=x$ such that $x_{\eps}\in A_{\eps}$ for $\eps$ small,
whereas this membership is independent from the chosen representative
in case of strongly internal sets. An internal set generated by a
constant net $A_{\eps}=A\subseteq\R^{n}$ will simply be denoted by
$[A]$.

The following theorem (cf.\ \cite{Ob-Ve,GKV} for the case $\rho_{\eps}=\eps$)
shows that internal and strongly internal sets have dual topological
properties: 
\begin{thm}
\noindent \label{thm:strongMembershipAndDistanceComplement}For $\eps\in I$,
let $A_{\eps}\subseteq\R^{n}$ and let $x_{\eps}\in\R^{n}$. Then
we have 
\begin{enumerate}
\item \label{enu:internalSetsDistance}$[x_{\eps}]\in[A_{\eps}]$ if and
only if $\forall q\in\R_{>0}\,\forall^{0}\eps:\ d(x_{\eps},A_{\eps})\le\rho_{\eps}^{q}$.
Therefore $[x_{\eps}]\in[A_{\eps}]$ if and only if $[d(x_{\eps},A_{\eps})]=0\in\RC{\rho}$. 
\item \label{enu:stronglyIntSetsDistance}$[x_{\eps}]\in\sint{A_{\eps}}$
if and only if $\exists q\in\R_{>0}\,\forall^{0}\eps:\ d(x_{\eps},A_{\eps}^{\text{c}})>\rho_{\eps}^{q}$,
where $A_{\eps}^{\text{c}}:=\R^{n}\setminus A_{\eps}$. Therefore,
if $(d(x_{\eps},A_{\eps}^{\text{c}}))\in\R_{\rho}$, then $[x_{\eps}]\in\sint{A_{\eps}}$
if and only if $[d(x_{\eps},A_{\eps}^{\text{c}})]>0$. 
\item \label{enu:internalAreClosed}$[A_{\eps}]$ is sharply closed. 
\item \label{enu:stronglyIntAreOpen}$\sint{A_{\eps}}$ is sharply open. 
\item \label{enu:internalGeneratedByClosed}$[A_{\eps}]=\left[\text{\emph{cl}}\left(A_{\eps}\right)\right]$,
where $\text{\emph{cl}}\left(S\right)$ is the closure of $S\subseteq\R^{n}$. 
\item \label{enu:stronglyIntGeneratedByOpen}$\sint{A_{\eps}}=\sint{\text{\emph{int}\ensuremath{\left(A_{\eps}\right)}}}$,
where $\emph{int}\left(S\right)$ is the interior of $S\subseteq\R^{n}$. 
\end{enumerate}
\end{thm}

\begin{proof}
\ref{enu:internalSetsDistance} $\mbox{\ensuremath{\Rightarrow}}$:
We have $x'_{\eps}\in A_{\eps}$ for some representative $[x'_{\eps}]=[x_{\eps}]\in[A_{\eps}]$.
But $d(x_{\eps},A_{\eps})\le|x_{\eps}-x'_{\eps}|+d(x'_{\eps},A_{\eps})$,
from which the conclusion follows.

\noindent \ref{enu:internalSetsDistance} $\Leftarrow$: Since the
net $\left(\inf_{a\in A_{\eps}}|x_{\eps}-a|\right)$ is $\rho$-negligible,
we can find a decreasing sequence $(\eps_{k})_{k\in\N}\downarrow0$
such that $\inf_{a\in A_{\eps}}|x_{\eps}-a|<\rho_{\eps}^{k}$ for
$\eps\le\eps_{k}$. Hence, for each $\eps\in(\eps_{k+1},\eps_{k}]_{\R}$
we can find $x'_{\eps}\in A_{\eps}$ such that $|x_{\eps}-x'_{\eps}|\le\rho_{\eps}^{k}$.
Therefore, $(x'_{\eps})$ is another representative of $[x_{\eps}]$
and $x'_{\eps}\in A_{\eps}$ for $\eps\le\eps_{0}$.

\noindent \ref{enu:stronglyIntSetsDistance}: Let $\left[x_{\eps}\right]\in\sint{A_{\eps}}$
and suppose to the contrary that there exists a sequence $\eps_{k}\downarrow0$
such that $d(x_{\eps_{k}},A_{\eps_{k}}^{\text{c}})\le\rho_{\eps_{k}}^{k}$
for all $k\in\N$. For each $k$, pick some $x_{k}'\in A_{\eps_{k}}^{\text{c}}$
with $|x_{k}'-x_{\eps_{k}}|<2\rho_{\eps_{k}}^{k}$ and choose $(x'_{\eps})\sim_{\rho}(x_{\eps})$
such that $x'_{\eps_{k}}=x'_{k}$ for all $k$. Then $x_{\eps_{k}}'\not\in A_{\eps_{k}}$
for all $k$, contradicting $x_{\eps}\in_{\eps}A_{\eps}$. Conversely,
let $d(x_{\eps},A_{\eps}^{\text{c}})>\rho_{\eps}^{q}$ for $\eps$
small. Then in particular, $x_{\eps}\in A_{\eps}$. Also, if $(x_{\eps}')\sim_{\rho}(x_{\eps})$
then $d(x_{\eps}',A_{\eps}^{\text{c}})>(1/2)\rho_{\eps}^{q}$ for
$\eps$ small, so $x_{\eps}'\in A_{\eps}$. Thus, $\left[x_{\eps}\right]\in\sint{A_{\eps}}$.

\noindent \ref{enu:internalAreClosed}: Let $x=[x_{\eps}]\in\RC{\rho}^{n}\setminus[A_{\eps}]$.
Then \ref{enu:internalSetsDistance} yields that $d(x_{\eps_{k}},A_{\eps_{k}})>\rho_{\eps_{k}}^{q}$
for some $q\in\R_{>0}$ and some sequence $(\eps_{k})_{k\in\N}\downarrow0$.
Set $r:=\frac{1}{2}\diff{\rho}^{q}$, then $y\in B_{r}(x)$ implies
that for some representative $[y_{\eps}]=y$ we have $d(y_{\eps_{k}},A_{\eps_{k}})\ge d(x_{\eps_{k}},A_{\eps_{k}})-\left|x_{\eps_{k}}-y_{\eps_{k}}\right|>\rho_{\eps_{k}}^{q}-\frac{1}{2}\rho_{\eps_{k}}^{q}$.
Thereby \ref{enu:internalSetsDistance} gives $y\notin[A_{\eps}]$.
This proves that $\RC{\rho}^{n}\setminus[A_{\eps}]$ is sharply open.

\noindent \ref{enu:stronglyIntAreOpen}: \ref{enu:stronglyIntSetsDistance}
yields that $\left[x_{\eps}\right]\in\sint{A_{\eps}}$ if and only
if $[x_{\eps}]$ is in the interior of $\sint{A_{\eps}}$ with respect
to the sharp topology.

\noindent \ref{enu:internalGeneratedByClosed}, \ref{enu:stronglyIntGeneratedByOpen}:
Directly from \ref{enu:internalSetsDistance} and \ref{enu:stronglyIntSetsDistance}. 
\end{proof}
For example, it is not hard to show that the closure in the sharp
topology of a ball of center $c=[c_{\eps}]$ and radius $r=[r_{\eps}]>0$
is 
\begin{equation}
\overline{B_{r}(c)}=\left\{ x\in\rti^{d}\mid\left|x-c\right|\le r\right\} =\left[\overline{\Eball_{r_{\eps}}(c_{\eps})}\right].\label{eq:closureBall}
\end{equation}
In fact, it suffices to prove these equalities for $c=0$, because
the translation $x\mapsto x-c$ is sharply continuous. If $(x_{n})$
is a sequence in $\{x\mid\left|x\right|\le r\}$ that converges to
$x_{0}$, then $\left|x_{0}\right|\le|x_{0}-x_{n}|+|x_{n}|\le|x_{0}-x_{n}|+r$.
Letting $n\to+\infty$, this shows that $\{x\mid\left|x\right|\le r\}$
is closed. Conversely, if $|x|\le r$, to prove that $x$ is an adherent
point of $B_{r}(0)$, we need to show that 
\[
\forall s\in\rti_{>0}\,\exists\bar{x}\in B_{r}(0)\cap B_{s}(x).
\]

\noindent Take $k\in\N$ such that $2\diff{\rho}^{k}\le\min(r,s)$,
and representatives $[x_{\eps}]=x$ and $[r_{\eps}]=r$ such that
$|x_{\eps}|\le r_{\eps}$ for small $\eps$. The point $\bar{x}_{\eps}:=x_{\eps}$
if $|x_{\eps}|<r_{\eps}-\rho_{\eps}^{k}$ and $\bar{x}_{\eps}:=x_{\eps}-\frac{x_{\eps}}{|x_{\eps}|}\rho_{\eps}^{k}$
otherwise satisfies the desired conditions. This proves the first
equality in \eqref{eq:closureBall}. The proof that $\overline{B_{r}(0)}\supseteq\left[\overline{\Eball_{r_{\eps}}(0)}\right]$
is easy. Vice versa, if $|\bar{x}_{\eps}|\le r_{\eps}+z_{\eps}$ for
some representatives $[\bar{x}_{\eps}]=x$ and $[z_{\eps}]=0$, then
setting $x_{\eps}:=\bar{x}_{\eps}$ if $|\bar{x}_{\eps}|\le r_{\eps}$
and $x_{\eps}:=\frac{\bar{x}_{\eps}}{\left|\bar{x}_{\eps}\right|}r_{\eps}$
otherwise gives another representative of $x$ that shows that $x\in\left[\overline{\Eball_{r_{\eps}}(0)}\right]$.

\noindent From \eqref{eq:closureBall} and Thm.~\eqref{thm:strongMembershipAndDistanceComplement},
it hence also follows that 
\begin{equation}
B_{r}(c)=\sint{\Eball_{r_{\eps}}(c_{\eps})}.\label{eq:ballStrgInt}
\end{equation}

\noindent In a similar way, it can be shown that for every $x$, $y\in\RC{\rho}$
\begin{equation}
y\geq x\Leftrightarrow y\in\overline{{\left\{ z\in\RC{\rho}\mid z>x\right\} }}.\label{eq: geq as a closure}
\end{equation}
Some relations between internal and strongly internal sets that we
will use below are listed in the following 
\begin{lem}
\noindent \label{Lem:OmegaEpsAroundBoundedAndCptSupp}Let $(\Omega_{\eps})$
be a net of subsets in $\R^{n}$ for all $\eps$, and $(B_{\eps})$
a sharply bounded net such that $[B_{\eps}]\subseteq\sint{\Omega_{\eps}}$,
then 
\begin{enumerate}
\item \label{enu:1.10}$\forall^{0}\eps:\ B_{\eps}\subseteq\Omega_{\eps}$. 
\item \label{enu:2.10}If each $B_{\eps}$ is closed, then $\exists S\in\N\,\forall[x_{\eps}]\in[B_{\eps}]\,\forall^{0}\eps:\ d(x_{\eps},\Omega_{\eps}^{c})\ge\rho_{\eps}^{S}$. 
\item \label{enu:ballInt}If $r=[r_{\eps}]\in\rcrho_{>0}$, $b=[b_{\eps}]\in\rcrho^{n}$
and $B_{r}(b)\subseteq\text{\emph{int}}\left([B_{\eps}]\right)$,
then $\forall^{0}\eps:\ \Eball_{r_{\eps}}(b_{\eps})\subseteq B_{\eps}$. 
\item \label{enu:reprStrInt}If $x\in\sint{A_{\eps}}\subseteq[B_{\eps}]$,
then $\forall[x_{\eps}]=x\,\forall^{0}\eps:\ x_{\eps}\in B_{\eps}$. 
\item \label{enu:int}If $(C_{\eps})$ is also sharply bounded and $[B_{\eps}]\subseteq\text{\emph{int}}\left([C_{\eps}]\right)$,
then there exists $S\in\N$ such that: 
\begin{enumerate}[label=(\alph*)]
\item \label{enu:a}$\forall^{0}\eps:\ \Eball_{\rho_{\eps}^{S}}(B_{\eps})\subseteq C{}_{\eps}$ 
\item \label{enu:b}$B_{\diff{\rho}^{S}}(B)\subseteq C$, 
\end{enumerate}
where, in general 
\[
B_{r}^{d}(B)=\left\{ x\mid d(x,B)<r\right\} =\bigcup_{b\in B}B_{r}^{d}(b)
\]
(in $\rcrho^{n}$, we also set $d(x,B):=\left[d(x_{\eps},B_{\eps})\right]\in\rcrho$). 
\item \label{enu:neigh}If each $B_{\eps}$ is closed, then there exists
a sharply bounded net $\left(B_{\eps}^{+}\right)$ of closed sets
such that $[B_{\eps}^{+}]\subseteq\sint{\Omega_{\eps}}$ is a sharp
neighborhood of $[B_{\eps}]$. 
\end{enumerate}
\end{lem}

\begin{proof}
\noindent To prove \ref{enu:1.10}, let us assume, by contradiction,
that we can find sequences $(\eps_{k})_{k}$ and $(x_{k})_{k}$ such
that $\eps_{k}\downarrow0$ and $x_{k}\in B_{\eps_{k}}\setminus\Omega_{\eps_{k}}$.
Defining $x_{\eps}:=x_{k}$ for $\eps=\eps_{k}$, and $x_{\eps}\in B_{\eps}$
otherwise, we have that $x:=[x_{\eps}]$ is moderate since $(B_{\eps})$
is sharply bounded. Hence $x\in[B_{\eps}]$, but $x_{\eps_{k}}\notin\Omega_{\eps_{k}}$
by construction, hence $x\notin\sint{\Omega_{\eps}}$ by Def.~\ref{def:internalStronglyInternal},
which is impossible because $[B_{\eps}]\subseteq\sint{\Omega_{\eps}}$.

\noindent \ref{enu:2.10}: Assume that \ref{enu:1.10} holds for all
$\eps\le\eps_{0}$. If $B_{\eps}$ is closed, then $B_{\eps}\Subset\R^{n}$
because $(B_{\eps})$ is sharply bounded. We can therefore find a
point $\bar{x}_{\eps}\in B_{\eps}$ such that $d(\bar{x}_{\eps},\Omega_{\eps}^{c})=d(B_{\eps},\Omega_{\eps}^{c})$.
At $\bar{x}:=[\bar{x}_{\eps}]\in[B_{\eps}]$ property \ref{enu:stronglyIntSetsDistance}
of Thm.~\ref{thm:strongMembershipAndDistanceComplement} yields the
existence of some $S\in\N$ such that $d(\bar{x}_{\eps},\Omega_{\eps}^{c})\ge\rho_{\eps}^{S}$
for $\eps$ small. From this the conclusion follows because $d(x_{\eps},\Omega_{\eps}^{c})\ge d(B_{\eps},\Omega_{\eps}^{c})=d(\bar{x}_{\eps},\Omega_{\eps}^{c})$
if $x_{\eps}\in B_{\eps}$ for $\eps$ small. If $[x'_{\eps}]=[x_{\eps}]$
is any other representative, then claim \ref{enu:2.10} still holds
because $d(x_{\eps},x'_{\eps})$ is negligible.

\noindent \ref{enu:ballInt}: By contradiction, assume that for some
$J\subzero I$ we can find $x_{\eps}\in\Eball_{r_{\eps}}(b_{\eps})\setminus B_{\eps}$
for all $\eps\in J$. Therefore, $x:=\left[\left(x_{\eps}\right)_{\eps\in J}\right]\in\overline{B_{r}(B)}|_{J}$.
But the assumption $B_{r}(b)\subseteq\text{int}\left([B_{\eps}]\right)$
yields $\overline{B_{r}(b)}\subseteq[B_{\eps}]=:B$ and hence $x\in B|_{J}$,
which is impossible.

\noindent \ref{enu:reprStrInt}: Directly from the previous result
and Thm.~\ref{thm:strongMembershipAndDistanceComplement}.\ref{enu:stronglyIntAreOpen}.

\noindent \ref{enu:int}: We prove by contradiction that there exists
$S\in\N$ satisfying \ref{enu:a}; we will then show that this $S$
also works for \ref{enu:b}. So, assume that for all $s\in\N$ there
exists $J_{s}\subzero I$ and $x_{s\eps}\in\Eball_{\rho_{\eps}^{s}}(B_{\eps})\setminus C_{\eps}$
for all $\eps\in J_{s}$. We can hence find $\eps_{s}\in J_{s}$ such
that $\eps_{s}<\frac{1}{s}$ and $x_{s\eps_{s}}\in\Eball_{\rho_{\eps_{s}}^{s}}(B_{\eps_{s}})\setminus C_{\eps_{s}}$.
Choosing recursively these $\eps_{s}$, we can assume that $\eps_{s+1}<\eps_{s}$,
so that $(\eps_{s})_{s}\downarrow0$. Set $J:=\left\{ \eps_{s}\in J_{s}\mid s\in\N_{>0}\right\} \subzero I$.
For each $\eps\in J$, we can set $x_{\eps}:=x_{s\eps_{s}}$ for the
unique $s\in\N_{>0}$ such that $\eps=\eps_{s}$, so that $x\in\rcrho|_{J}$.
For all $\eps=\eps_{p}\in J$, if $\eps<\eps_{s}$, then $p>s$ because
$(\eps_{s})_{s}$ is strictly decreasing. Thereby 
\[
x_{\eps}=x_{\eps_{p}}\in\Eball_{\rho_{\eps_{p}}^{p}}(B_{\eps_{p}})\setminus C_{\eps_{p}}\subseteq\Eball_{\rho_{\eps}^{s}}(B_{\eps})\setminus C_{\eps}
\]
because $p>s$ and $\eps=\eps_{p}$. This proves that $\left(d(x_{\eps},B_{\eps})\right)_{\eps\in J}\sim_{\rho}0$
and hence that $x:=\left[\left(x_{\eps}\right)_{\eps\in J}\right]\in B|_{J}\subseteq\text{int}\left(C\right)|_{J}$.
Therefore, $B_{r}(x)\subseteq\text{int}\left(C|_{J}\right)$ for some
$r\in\rcrho_{>0}|_{J}$. Using \ref{enu:ballInt}, we get $\Eball_{r_{\eps}}(x_{\eps})\subseteq C_{\eps}$
for $\eps\in J$ sufficiently small, and hence $x_{\eps}\in C_{\eps},$
a contradiction. Now assume that $S\in\N$ satisfies \ref{enu:a}.
Then for all $x\in B_{\diff{\rho}^{S}}(B)$, we have $x\in B_{\diff{\rho}^{S}}(b)$
for some $b=[b_{\eps}]\in B$. Therefore, for all $[x_{\eps}]=x$
and $\eps$ small, we have $x_{\eps}\in\Eball_{r_{\eps}}(b_{\eps})\subseteq C_{\eps}$
using \ref{enu:a}.

\noindent \ref{enu:neigh}: To prove this property, it suffices to
consider an $M\in\rcrho_{>0}$ such that $[B_{\eps}]\in B_{M}(0)$
and to define 
\[
B_{\eps}^{+}:=\overline{\left\{ x\in\Eball_{M_{\eps}}(0)\mid d(x,\Omega_{\eps}^{c})\ge\rho_{\eps}^{S+1}\right\} }\Subset\R^{n},
\]
where $S\in\N$ comes from \ref{enu:2.10}. 
\end{proof}
Let $X=\sint{A_{\eps}}$ be a strongly internal set, $x$, $y\in X$
and both $K$, $K^{c}\subzero I$. Set $e_{K}:=[1_{K\eps}]\in\rcrho$,
where $1_{K\eps}:=1$ if $\eps\in K$ and $1_{K\eps}:=0$ otherwise.
Then $z:=x\cdot e_{K}+y\cdot e_{K^{c}}\in X$ and $z|_{K}\subseteq x$,
$z|_{K^{c}}\subseteq y$ (we then say that $X$ is closed with respect
to \emph{interleaving}; this property holds also for internal sets,
see \cite{Ob-Ve}). The same property does not hold if $x\in B_{r}(c)\setminus B_{s}(d)$
and $y\in B_{s}(d)\setminus B_{r}(c)$, so that $B_{r}(c)\cup B_{s}(d)$
is sharply open but is not strongly internal. The same kind of example
can be repeated e.g.~considering arbitrary unions of pairwise disjoint
balls.

To obtain large open sets starting from a net of subsets $A_{\eps}\subseteq\R^{n}$,
we can consider the analogue of $\sint{A_{\eps}}$ but using the radii
of the Fermat topology: 
\begin{defn}
\label{def:largeStrongInt}Let $(A_{\eps})$ be a net of subsets of
$\R^{n}$ and let $(x_{\eps})$, $(x'_{\eps})$ be nets of points
of $\R^{n}$. Then 
\begin{enumerate}
\item We write $(x_{\eps})\sim_{\text{F}}(x'_{\eps})$ to denote the property
$|x-x'|<r$ for all $r\in\R_{>0}$, i.e., $\lim_{\eps\to0^{+}}|x_{\eps}-x'_{\eps}|=0$. 
\item We say that $x_{\eps}\in_{\text{F}}A_{\eps}$, and we read it as $(x_{\eps})$
\emph{strongly belongs to $(A_{\eps})$ in the Fermat topology}, if 
\begin{enumerate}[label=(\alph*)]
\item $\forall^{0}\eps:\ x_{\eps}\in A_{\eps}$. 
\item If $(x'_{\eps})\sim_{\text{F}}(x_{\eps})$, then also $x'_{\eps}\in A_{\eps}$
for $\eps$ small. 
\end{enumerate}
\noindent Moreover, we set $\sintF{A_{\eps}}:=\left\{ [x_{\eps}]\in\RC{\rho}^{n}\mid x_{\eps}\in_{\text{F}}A_{\eps}\right\} $,
and we call it the \emph{strongly internal set} generated by the net
$(A_{\eps})$\emph{ in the Fermat topology}. 
\end{enumerate}
\end{defn}

The following result can be proved simply by generalizing the proof
of Thm\@.~\ref{thm:strongMembershipAndDistanceComplement}. 
\begin{thm}
\noindent \label{thm:strongMembershipAndDistanceComplementFermat}For
$\eps\in I$, let $A_{\eps}\subseteq\R^{n}$ and let $x_{\eps}\in\R^{n}$.
Then we have 
\begin{enumerate}
\item \label{enu:stronglyIntSetsDistanceFermat}$[x_{\eps}]\in\sintF{A_{\eps}}$
if and only if $\exists r\in\R_{>0}\,\forall^{0}\eps:\ d(x_{\eps},A_{\eps}^{\text{c}})>r$. 
\item \label{enu:stronglyIntAreOpenFermat}$\sintF{A_{\eps}}$ is a Fermat
open set. 
\end{enumerate}
\end{thm}

Sharply bounded internal sets (which are always sharply closed by
Thm.~\ref{thm:strongMembershipAndDistanceComplement}~\ref{enu:internalAreClosed})
serve as compact sets for our generalized functions. We will show
this by proving for them an extreme value theorem (see Thm.~\ref{cor:extremeValues});
for a deeper study of this type of sets in the case $\rho=(\eps)$
see \cite{GK15}; in the same particular setting, the notion of sharp
topology was introduced in \cite{Biag,Sca98}; see also \cite{May08,GKOS}
for an analogue of Lem.~\ref{lem:mayer}; see \cite{Ob-Ve} for the
study of internal sets, and see \cite{GKV} for strongly internal
sets.

\section{\label{sec:GSF}Generalized functions as smooth set-theoretical maps}

\subsubsection{\label{subsec:DefSharpCont}Definition and sharp continuity}

Using the ring $\rti$, it is easy to consider a Gaussian with an
infinitesimal standard deviation. If we denote this probability density
by $f(x,\sigma)$, and if we set $\sigma=[\sigma_{\eps}]\in\RC{\rho}_{>0}$,
where $\sigma\approx0$, we obtain the net of smooth functions $(f(-,\sigma_{\eps}))_{\eps\in I}$.
This is the basic idea we are going to develop in the following definitions.
We will first introduce the notion of a net $(f_{\eps})$ defining
a generalized smooth function of the type $X\longrightarrow Y$, where
$X\subseteq\RC{\rho}^{n}$ and $Y\subseteq\RC{\rho}^{d}$. This is
a net of smooth functions $f_{\eps}\in\cinfty(\Omega_{\eps},\R^{d})$
that induces well-defined maps of the form $[\partial^{\alpha}f_{\eps}(-)]:\sint{\Omega_{\eps}}\ra\RC{\rho}^{d}$,
for every multi-index $\alpha\in\N^{n}$. 
\begin{defn}
\label{def:netDefMap}Let $X\subseteq\RC{\rho}^{n}$ and $Y\subseteq\RC{\rho}^{d}$
be arbitrary subsets of generalized points. Let $(\Omega_{\eps})$
be a net of open subsets of $\R^{n}$, and $(f_{\eps})$ be a net
of smooth functions, with $f_{\eps}\in\cinfty(\Omega_{\eps},\R^{d})$.
Then we say that 
\[
(f_{\eps})\textit{ defines a generalized smooth function}:X\longrightarrow Y
\]
if: 
\begin{enumerate}
\item \label{enu:dom-cod}$X\subseteq\sint{\Omega_{\eps}}$ and $[f_{\eps}(x_{\eps})]\in Y$
for all $[x_{\eps}]\in X$. 
\item \label{enu:partial-u-moderate}$\forall[x_{\eps}]\in X\,\forall\alpha\in\N^{n}:\ (\partial^{\alpha}f_{\eps}(x_{\eps}))\in\R_{\rho}^{d}$. 
\end{enumerate}
\noindent We recall that the notation 
\[
\forall[x_{\eps}]\in X:\ \mathcal{P}\{(x_{\eps})\}
\]
means 
\[
\forall(x_{\eps})\in\R_{\rho}^{n}:\ [x_{\eps}]\in X\ \Rightarrow\ \mathcal{P}\{(x_{\eps})\},
\]
i.e.~for all representatives $(x_{\eps})$ generating a point $[x_{\eps}]\in X$,
the property $\mathcal{P}\{(x_{\eps})\}$ holds. 
\end{defn}

\noindent A generalized smooth function (or map, in this paper these
terms are used as synonymous) is simply a function of the form $f=[f_{\eps}(-)]|_{X}$: 
\begin{defn}
\label{def:generalizedSmoothMap} Let $X\subseteq\RC{\rho}^{n}$ and
$Y\subseteq\RC{\rho}^{d}$ be arbitrary subsets of generalized points,
then we say that 
\[
f:X\longrightarrow Y\text{ is a \emph{generalized smooth function}}
\]
if $f\in\Set(X,Y)$ and there exists a net $f_{\eps}\in\cinfty(\Omega_{\eps},\R^{d})$
defining a generalized smooth map of type $X\longrightarrow Y$, in
the sense of Def.\ \ref{def:netDefMap}, such that 
\begin{equation}
\forall[x_{\eps}]\in X:\ f\left([x_{\eps}]\right)=\left[f_{\eps}(x_{\eps})\right].\label{eq:f-u-relations}
\end{equation}

\noindent We will also say that $f$ \emph{is defined by the net of
smooth functions} $(f_{\eps})$ or that the net $(f_{\eps})$ \emph{represents}
$f$. The set of all these generalized smooth functions (GSF) will
be denoted by $^{\rho}\Gcinf(X,Y)$ or simply by $\Gcinf(X,Y)$. 
\end{defn}

Let us note explicitly that definitions \ref{def:netDefMap} and \ref{def:generalizedSmoothMap}
state minimal logical conditions to obtain a set-theoretical map from
$X$ into $Y$ and defined by a net of smooth functions. In particular,
the following Thm.~\ref{thm:indepRepr} states that in equality \eqref{eq:f-u-relations}
we have independence from the representatives for all derivatives
$[x_{\eps}]\in X\mapsto[\partial^{\alpha}f_{\eps}(x_{\eps})]\in\RC{\rho}^{d}$,
$\alpha\in\N^{n}$. 
\begin{thm}
\label{thm:indepRepr}Let $X\subseteq\RC{\rho}^{n}$ and $Y\subseteq\RC{\rho}^{d}$
be arbitrary subsets of generalized points. Let $(\Omega_{\eps})$
be a net of open subsets of $\R^{n}$, and $(f_{\eps})$ be a net
of smooth functions, with $f_{\eps}\in\cinfty(\Omega_{\eps},\R^{d})$.
Assume that $(f_{\eps})$ defines a generalized smooth map of the
type $X\longrightarrow Y$, then 
\[
\forall\alpha\in\N^{n}\,\forall(x_{\eps}),(x'_{\eps})\in\R_{\rho}^{n}:\ [x_{\eps}]=[x'_{\eps}]\in X\ \Rightarrow\ (\partial^{\alpha}f_{\eps}(x_{\eps}))\sim_{\rho}(\partial^{\alpha}f_{\eps}(x'_{\eps}))
\]
\end{thm}

\begin{proof}
Let $\alpha\in\N^{n}$ and $(x_{\eps})$, $(x'_{\eps})$ be two representatives
of the same point $x=[x_{\eps}]=[x'_{\eps}]\in X\subseteq\sint{\Omega_{\eps}}$.
Thm.\ \ref{thm:strongMembershipAndDistanceComplement}~\ref{enu:stronglyIntSetsDistance}
yields 
\begin{equation}
d(x_{\eps},\Omega_{\eps}^{c})>\rho_{\eps}^{q}\label{eq:dist}
\end{equation}
for some $q\in\R_{>0}$ and $\eps$ small. Thus $\Eball_{\rho_{\eps}^{q}}(x_{\eps})\subseteq\Omega_{\eps}$
for these values of $\eps$. Choose $r\in\R_{>0}$ sufficiently big
so that 
\begin{equation}
2\rho_{\eps}^{r}<\rho_{\eps}^{q}\label{eq:eps_r_q}
\end{equation}
for $\eps$ small. Since $(x_{\eps})\sim_{\rho}(x'_{\eps})$ we have
that 
\begin{equation}
x'_{\eps}\in\Eball_{\rho_{\eps}^{r}}(x_{\eps})\label{eq:x'EuclBall}
\end{equation}
for $\eps$ small, and the entire segment $[x_{\eps},x'_{\eps}]$
connecting $x_{\eps}$ and $x'_{\eps}$ lies in $\Eball_{\rho_{\eps}^{r}}(x_{\eps})$.
Suppose that \eqref{eq:dist}, \eqref{eq:eps_r_q} and \eqref{eq:x'EuclBall}
hold for $\eps\in(0,\eps_{0}]$. Fix $i\in\{1,\dots,d\}$ and set
$\mu_{\eps}(t):=\partial^{\alpha}f_{\eps}^{i}(x_{\eps}+t(x'_{\eps}-x_{\eps}))$
for $t\in[0,1]_{\R}$ and $\eps\in(0,\eps_{0}]$. By the classical
mean value theorem $\partial^{\alpha}f_{\eps}^{i}(x'_{\eps})-\partial^{\alpha}f_{\eps}^{i}(x_{\eps})=\mu_{\eps}(1)-\mu_{\eps}(0)=\mu_{\eps}'(\theta_{\eps})$
for some $\theta_{\eps}\in(0,1)$, and hence for all $\eps\in(0,\eps_{0}]$
we get 
\begin{equation}
\partial^{\alpha}f_{\eps}^{i}(x'_{\eps})-\partial^{\alpha}f_{\eps}^{i}(x_{\eps})=\sum_{k=1}^{n}\partial^{\alpha+e_{k}}f_{\eps}^{i}(\xi_{\eps})\cdot(x_{\eps}^{\prime k}-x_{\eps}^{k}),\label{eq:meanValueIndepRepr}
\end{equation}
where $\xi_{\eps}:=x_{\eps}+\theta_{\eps}(x'_{\eps}-x_{\eps})$ and
$e_{k}:=(0,\ptind^{k-1},0,1,0,\ldots,0)\in\N^{n}$. The generalized
point $[\xi_{\eps}]=[x_{\eps}]\in X$ since $(x'_{\eps})\sim_{\rho}(x_{\eps})$.
Therefore by Def.~\ref{def:netDefMap}~\ref{enu:partial-u-moderate}
we get that every derivative $\left(\partial^{\alpha+e_{k}}f_{\eps}^{i}(\xi_{\eps})\right)$
is $\rho$-moderate. From this and the equivalence $(x'_{\eps})\sim_{\rho}(x_{\eps})$,
equation \eqref{eq:meanValueIndepRepr} yields the conclusion $(\partial^{\alpha}f_{\eps}(x'_{\eps}))\sim_{\rho}(\partial^{\alpha}f_{\eps}(x_{\eps}))$. 
\end{proof}
\noindent Note that taking arbitrary subsets $X\subseteq\RC{\rho}^{n}$
in Def.\ \ref{def:netDefMap}, we can also consider GSF defined on
closed sets, like the set of all infinitesimals, or like a closed
interval $[a,b]\subseteq\RC{\rho}$. We can also consider GSF defined
at infinite generalized points. A simple case is the exponential map
\begin{equation}
e^{(-)}:[x_{\eps}]\in\left\{ x\in\RC{\rho}\mid\exists z\in\RC{\rho}_{>0}:\ x\le\log z\right\} \mapsto\left[e^{x_{\eps}}\right]\in\RC{\rho}.\label{eq:exp}
\end{equation}
The domain of this map depends on the infinitesimal net $\rho$. For
instance, if $\rho=(\eps)$ then all its points are bounded by generalized
numbers of the form $[-N\log\eps]$, $N\in\N$; whereas if $\rho=\left(e^{-\frac{1}{\eps}}\right)$,
all points are bounded by $[N\eps^{-1}]$, $N\in\N$. Another possibility
for the exponential function is to consider two gauges $\rho\ge\sigma$
and the subring of $\RC{\sigma}$ defined by 
\[
\RCud{\sigma}{\rho}:=\{x\in\RC{\sigma}\mid\exists N\in\N:\ |x|\le\diff{\rho}^{-N}\},
\]
where here we have set $\diff{\rho}:=[\rho_{\eps}]_{\sim_{\sigma}}\in\RC{\sigma}$.
If we have 
\begin{equation}
\forall N\in\N\,\exists M\in\N:\ \diff{\rho}^{-N}\le-M\log\diff{\sigma},\label{eq:expTwoGauges}
\end{equation}
then $e^{(-)}:[x_{\eps}]\in\RCud{\sigma}{\rho}\mapsto\left[e^{x_{\eps}}\right]\in\RC{\sigma}$
is well defined. For example, if $\sigma_{\eps}:=\exp\left(-\rho_{\eps}^{1/\eps}\right)$,
then $\sigma\le\rho$ and \eqref{eq:expTwoGauges} holds for $M=1$.
Note that the natural ring morphism $[x_{\eps}]_{\sim_{\sigma}}\in\RCud{\sigma}{\rho}\mapsto[x_{\eps}]_{\sim_{\rho}}\in\rcrho$
is surjective but generally not injective.

A first regularity property of GSF is the continuity with respect
to the sharp topology, as proved in the following 
\begin{thm}
\label{thm:GSF-continuity}Let $X\subseteq\RC{\rho}^{n}$, $Y\subseteq\RC{\rho}^{d}$
and $f_{\eps}\in\cinfty(\Omega_{\eps},\R^{d})$ be a net of smooth
functions that defines a GSF of the type $X\longrightarrow Y$. Then 
\begin{enumerate}
\item \label{enu:modOnEpsDepBall}$\forall[x_{\eps}]\in X\,\forall\alpha\in\N^{n}\,\exists q\in\R_{>0}\,\forall^{0}\eps:\ \sup_{y\in\Eball_{\rho_{\eps}^{q}}(x_{\eps})}\left|\partial^{\alpha}f_{\eps}(y)\right|\le\rho_{\eps}^{-q}.$ 
\item \label{enu:locLipSharp}For all $\alpha\in\N^{n}$, the GSF $g:[x_{\eps}]\in X\mapsto[\partial^{\alpha}f_{\eps}(x_{\eps})]\in\Rtil^{d}$
is locally Lipschitz in the sharp topology, i.e.~each $x\in X$ possesses
a sharp neighborhood $U$ such that $|g(x)-g(y)|\le L|x-y|$ for all
$x$, $y\in U$ and some $L\in\RC{\rho}$. 
\item \label{enu:GSF-cont}Each $f\in\gsf(X,Y)$ is continuous with respect
to the sharp topologies induced on $X$, $Y$. 
\item \label{enu:suffCondFermatCont}Assume that the GSF $f$ is locally
Lipschitz in the Fermat topology and that its Lipschitz constants
are always finite: $L\in\R$. Then $f$ is continuous in the Fermat
topology. 
\end{enumerate}
\end{thm}

\begin{proof}
We first prove \ref{enu:modOnEpsDepBall} by contradiction, assuming
that for some $[x_{\eps}]\in X$ and some $\alpha$ there exists $(\eps_{k})_{k}\downarrow0$
and a sequence $(y_{k})_{k}$ of points in $\R^{n}$ such that $|x_{\eps_{k}}-y_{k}|<\rho_{\eps_{k}}^{k}$
but $|\partial^{\alpha}f_{\eps}(y_{k})|>\rho_{\eps_{k}}^{-k}$. Define
$x'_{\eps}:=y_{k}$ for $\eps=\eps_{k}$ and $x'_{\eps}:=x_{\eps}$
otherwise. Then $(x'_{\eps})\sim_{\rho}(x_{\eps})$ but $(\partial^{\alpha}f_{\eps}(x'_{\eps}))$
is not $\rho$-moderate, which contradicts Def.~\ref{def:netDefMap}
\ref{enu:partial-u-moderate}.

To prove \ref{enu:locLipSharp}, we apply \ref{enu:modOnEpsDepBall}
to each derivative $\partial^{\alpha+e_{k}}f_{\eps}$ to obtain 
\begin{equation}
\forall k=1,\ldots,n\,\exists q_{k}\in\R_{>0}\exists\eps_{k}\in I\,\forall\eps\in(0,\eps_{k}]:\ \sup_{y\in\Eball_{\rho_{\eps}^{q_{k}}}(x_{\eps})}|\partial^{\alpha+e_{k}}f_{\eps}(y)|\le\rho_{\eps}^{-q_{k}}.\label{eq:EballsDer}
\end{equation}
Set $q:=\max_{k=1,\ldots,n}q_{k}$, so that for $y$, $z\in B_{\text{d}\rho^{q}}(x)$
we get 
\begin{equation}
\exists\eps_{0}\,\forall\eps\in(0,\eps_{0}]:\ [y_{\eps},z_{\eps}]\subseteq\Eball_{\rho_{\eps}^{q}}(x_{\eps}).\label{eq:y-z-inEBall-r}
\end{equation}
For any $i\in\{1,\dots,d\}$ and $\eps$ small we can write 
\[
|\partial^{\alpha}f_{\eps}^{i}(y_{\eps})-\partial^{\alpha}f_{\eps}^{i}(z_{\eps})|=\left|\sum_{k=1}^{n}\partial^{\alpha+e_{k}}f_{\eps}^{i}(\zeta_{\eps})\cdot(y_{\eps}^{k}-z_{\eps}^{k})\right|
\]
where $\zeta_{\eps}:=y_{\eps}+\sigma_{\eps}(z_{\eps}-y_{\eps})$ for
some $\sigma_{\eps}\in(0,1)$. Therefore $\zeta_{\eps}\in\Eball_{\rho_{\eps}^{q}}(x_{\eps})\subseteq\Eball_{\rho_{\eps}^{q_{k}}}(x_{\eps})$
and \eqref{eq:EballsDer} implies 
\[
|\partial^{\alpha}f_{\eps}(y_{\eps})-\partial^{\alpha}f_{\eps}(z_{\eps})|\le d\sqrt{n}\rho_{\eps}^{-q}|y_{\eps}-z_{\eps}|.
\]

Property \ref{enu:GSF-cont} follows upon setting $\alpha=0$ in \ref{enu:locLipSharp}.
Property \ref{enu:suffCondFermatCont} follows directly from the definition
of locally Lipschitz function in the Fermat topology. In fact, we
have that $L|x-y|<r\in\R$ if $y\in\Fball_{r/L}(x)$, which is an
open ball in the Fermat topology because $L\in\R$. 
\end{proof}
In the following result, we show that the dependence of the domains
$\Omega_{\eps}$ on $\eps$ can be avoided since it does not lead
to a larger class of generalized functions. In spite of this possibility,
we preferred to formulate Def.~\ref{def:netDefMap} using $\eps$-dependent
domains because the extension of $f_{\eps}\in\cinfty(\Omega_{\eps},\R^{d})$
to the whole of $\R^{n}$ is not unique and hence introduces extrinsic
elements. 
\begin{lem}
\label{lem:fromOmega_epsToRn}Let $X\subseteq\RC{\rho}^{n}$ and $Y\subseteq\RC{\rho}^{d}$
be arbitrary subsets of generalized points, then $f:X\longrightarrow Y$
is a GSF if and only if there exists a net $v_{\eps}\in\cinfty(\R^{n},\R^{d})$
defining a generalized smooth map of type $X\longrightarrow Y$ such
that $f=[v_{\eps}(-)]|_{X}$. 
\end{lem}

\begin{proof}
The stated condition is clearly sufficient. Conversely, assume that
$f:X\longrightarrow Y$ is defined by the net $f_{\eps}\in\cinfty(\Omega_{\eps},\R^{d})$.
For every $\eps\in I$ let $\Omega'_{\eps}:=\left\{ x\in\Omega_{\eps}\mid d(x,\Omega_{\eps}^{\text{c}})>\rho_{\eps}^{\frac{1}{\eps}}\right\} $,
$\Omega''_{\eps}:=\left\{ x\in\Omega_{\eps}\mid d(x,\Omega_{\eps}^{\text{c}})>\rho_{\eps/2}^{\frac{2}{\eps}}\right\} $
and choose $\chi_{\eps}\in\Coo(\Omega_{\eps})$ with $\text{supp}(\chi_{\eps})\subseteq\Omega''_{\eps}$
and $\chi_{\eps}=1$ in a neighborhood of $\Omega'_{\eps}$. Set $f_{\eps}:=0$
on $\R^{n}\setminus\Omega{}_{\eps}$ and $v_{\eps}:=\chi_{\eps}\cdot f_{\eps}$,
so that $v_{\eps}\in\cinfty(\R^{n},\R^{d})$. If $x=[x_{\eps}]\in X\subseteq\sint{\Omega_{\eps}}$,
then $x_{\eps}\in\Omega'_{\eps}\subseteq\Omega_{\eps}$ for $\eps$
small by Thm.~\ref{thm:strongMembershipAndDistanceComplement}, so
for all $\alpha\in\N^{n}$ we get $\partial^{\alpha}v_{\eps}(x_{\eps})=\partial^{\alpha}f_{\eps}(x_{\eps})$.
Therefore, $(v_{\eps})_{\eps}$ defines a GSF of the type $X\longrightarrow Y$
and clearly $f=[f_{\eps}(-)]|_{X}=[v_{\eps}(-)]|_{X}$. 
\end{proof}
Consider a GSF $f:X\ra Y$. We want to show that for a large class
of domains $X$, the function $f$ is uniquely determined by its values
on particularly well behaved points $x\in X$. These domains and these
points are introduced in the following 
\begin{defn}
\label{def:nearStd-Subpoint}~ 
\begin{enumerate}
\item \label{def:nearStdInf}Let $x\in\RC{\rho}^{n}$, then we say that
the point $x$ is \emph{near-standard} if there exists a representative
$(x_{\eps})$ of $x$ such that $\exists\lim_{\eps\to0^{+}}x_{\eps}=:\st{x}\in\R^{n}$
($\st{x}$ is called the standard part of $x$). Clearly, this limit
does not depend on the representative of $x$. 
\item If $\Omega\subseteq\R^{n}$, then $\nrst{\Omega}:=\left\{ x\in\rti^{n}\mid\exists\st{x}\in\Omega\right\} $. 
\item \label{enu:closedWrtSubpoints}We say that $X\subseteq\rcrho^{n}$
\emph{contains its converging subpoints} if for all $J\subzero I$
and all $x'\in X|_{J}$ which is near standard or infinite, there
exists some $x\in X$ with $x'\subseteq x$ and such that $\lim_{\eps\to0,\eps\in J}x'_{\eps}=\lim_{\eps\to0}x_{\eps}$. 
\end{enumerate}
\end{defn}

\begin{thm}
\label{thm:nearStdInfEquality}Let $X\subseteq\rcrho^{n}$, $Y\subseteq\RC{\rho}^{d}$,
and let $f:X\longrightarrow Y$ be a GSF. If $X$ contains its converging
subpoints and if $f(x)=0$ for all near-standard and for all infinite
points $x\in X$, then $f=0$. 
\end{thm}

\begin{proof}
In fact, suppose that $f$ vanishes on every near-standard and every
infinite point belonging to $X$, but that $f(x)\ne0$ for some $x\in X$.
Let $(x_{\eps})$ be a representative of $x$. Then there exist $m\in\N$
and $(\eps_{k})_{k}\downarrow0$ such that $|f_{\eps_{k}}(x_{\eps_{k}})|>\rho_{\eps_{k}}^{m}$,
where $(f_{\eps})$ is a net that defines $f$. If $\left(x_{\eps_{k}}\right)_{k}$
is a bounded sequence, we can extract from it a convergent subsequence
$(x_{\eps_{k_{l}}})_{l}$. Setting $J:=\{k_{l}\mid l\in\N\}$, $x'=x|_{J}$
is a subpoint of $x$ and by assumption there exists some $y\in X$
that satisfies $y|_{J}=x'$ and additionally is near-standard, with
the same limit as $x'$. By construction, $f(y)\ne0$, a contradiction.
If, on the other hand, the sequence $(x_{\eps_{k}})_{k}$ is unbounded,
then we can extract a subsequence with $\lim_{l\to+\infty}|x_{\eps_{k_{l}}}|=+\infty$,
and can then proceed as above to construct an infinite point $y\in X$
at which $f(y)\ne0$. 
\end{proof}
Analogously, we can prove the following: 
\begin{thm}
\label{thm:nearStdInfModerate}Let $X\subseteq\rcrho^{n}$ and $Y\subseteq\RC{\rho}^{d}$.
Let $(\Omega_{\eps})$ be a net of open subsets of $\R^{n}$, and
$(f_{\eps})$ be a net of smooth functions, with $f_{\eps}\in\cinfty(\Omega_{\eps},\R^{d})$.
Assume that $X$ contains its subpoints. Then $(f_{\eps})$ defines
a GSF of the type $X\ra Y$ if and only if 
\begin{enumerate}
\item $X\subseteq\sint{\Omega_{\eps}}$ and $[f_{\eps}(x_{\eps})]\in Y$
for all $[x_{\eps}]\in X$. 
\item $\forall\alpha\in\N^{n}:\ (\partial^{\alpha}f_{\eps}(x_{\eps}))\in\R_{\rho}^{d}$
for all near-standard and for all infinite points $[x_{\eps}]\in X$. 
\end{enumerate}
\end{thm}

\noindent For example, if $\Omega$ is an open subset of $\R^{n}$,
and we define the set of compactly supported generalized points by
\[
\csp{\Omega}:=\{[x_{\eps}]\in\RC{\rho}^{n}\mid\exists K\comp\Omega\ \forall^{0}\eps:\ x_{\eps}\in K\}\sse\langle\Omega\rangle,
\]
then $\csp{\Omega}$ contains its subpoints. Internal and strongly
internal sets generated by a constant sequence $A\subseteq\R^{n}$,
i.e.~$[A]$ and $\sint{A}$, provide further examples of a subset
containing its subpoints. Moreover, an arbitrary union $\bigcup_{j\in J}X_{j}$
of sets, with each $X_{j}$ containing its subpoints, still contains
its subpoints.

The subset $\csp{\Omega}$ is the natural domain for embedded distributions,
as shown in the following section.

\subsubsection{\label{subsec:Embeddings}Embedding of Schwartz distributions}

\paragraph*{Introduction}

Among the re-occurring themes of this work are the choices which the
solution of a given problem within our framework may depend upon.
For instance, \eqref{eq:exp} shows that the domain of a GSF depends
on the infinitesimal net $\rho$. It is also easy to show that the
trivial Cauchy problem 
\[
\begin{cases}
x'(t)-[\eps^{-1}]\cdot x(t)=0\\
x(0)=1
\end{cases}
\]
has no solution in $^{\rho}\Gcinf(\R,\R)$ if $\rho_{\eps}=\eps$.
Nevertheless, it has the unique solution $x(t)=\left[e^{\frac{1}{\eps}t}\right]\in\gsf(\R,\R)$
if $\rho_{\eps}=e^{-\frac{1}{\eps}}$. Therefore, the choice of the
infinitesimal net $\rho$ is closely tied to the possibility of solving
a given class of differential equations. This illustrates the dependence
of the theory on the infinitesimal net $\rho$.

Further choices concern the embedding of Schwartz distributions. Since
we need to associate a net of smooth functions $(f_{\eps})$ to a
given distribution $T\in\D'(\Omega)$, this embedding is naturally
built upon a regularization process. In our approach, this regularization
will depend on an infinite number $b\in\rcrho$, and the choice of
$b$ depends on what properties we need from the embedding. For example,
if $\delta$ is the (embedding of the) one-dimensional Dirac delta,
then we have the property 
\begin{equation}
\delta(0)=b.\label{eq:deltaAt0}
\end{equation}
We can also choose the embedding so as to get the property 
\begin{equation}
H(0)=\frac{1}{2},\label{eq:H-at-0}
\end{equation}
where $H$ is the (embedding of the) Heaviside step function. Equalities
like these are used in diverse applications (see, e.g., \cite{Col92,MObook}
and references therein). In fact, we are going to construct a family
of structures of the type $(\mathcal{G},\partial,\iota)$, where $(\mathcal{G},\partial)$
is a a sheaf of differential algebra and $\iota:\D'\ra\mathcal{G}$
is an embedding. The particular structure we need to consider depends
on the problem we have to solve. Of course, one may be more interested
in having an intrinsic embedding of distributions. This can be done
by following the ideas of the full Colombeau algebra (see e.g.~\cite{GKOS,GiNi15,GiLu15,GN}).
Nevertheless, this choice decreases the simplicity of the present
approach and is incompatible with properties like \eqref{eq:deltaAt0}
and \eqref{eq:H-at-0}.

\paragraph*{The embedding}

If $\phi\in\mathcal{D}(\R^{n})$, $r\in\R_{>0}$ and $x\in\R^{n}$,
we use the notations $r\odot\phi$ for the function $x\in\R^{n}\mapsto\frac{1}{r^{n}}\cdot\phi\left(\frac{x}{r}\right)\in\R$
and $x\oplus\phi$ for the function $y\in\R^{n}\mapsto\phi(y-x)\in\R$.
These notations highlight that $\odot$ is a free action of the multiplicative
group $(\R_{>0},\cdot,1)$ on $\D(\R^{n})$ and $\oplus$ is a free
action of the additive group $(\R_{>0},+,0)$ on $\D(\R^{n})$. We
also have the distributive property $r\odot(x\oplus\phi)=rx\oplus r\odot\phi$.
Our embedding procedure will ultimately rely on convolution with suitable
mollifiers. To construct these, we need some technical preparations. 
\begin{lem}
\label{lem:ColombeauMollifier} For any $n\in\N_{>0}$ there exists
some $\mu_{n}\in\mathcal{S}(\R)$ with the following properties: 
\begin{enumerate}
\item \label{enu:iColMoll}$\int\mu_{n}(x)\,\diff{x}=1$. 
\item \label{enu:nullMoments}$\int_{0}^{\infty}x^{\frac{j}{n}}\mu_{n}(x)\,\diff{x}=0$
for all $j\in\N_{>0}$. 
\item $\mu_{n}(0)=1$. 
\item $\mu_{n}$ is even. 
\item \label{enu:vColMoll}$\mu_{n}(k)=0$ for all $k\in\mathbb{Z}\setminus\{0\}$. 
\end{enumerate}
\end{lem}

\begin{proof}
Consider the Fréchet space 
\[
F:=\{\mu\in\mathcal{S}(\R)\mid\mu\text{ even},\ \forall k\in\Z\setminus\{0\}:\mu(k)=0\}
\]
and define the continuous linear functionals $f_{m}:F\to\R$, where
$f_{0}(\mu):=\mu(0)$, $f_{1}(\mu):=\int\mu(x)\,dx$, and $f_{m}(\mu):=\int_{0}^{\infty}x^{\frac{m-1}{n}}\mu(x)\,dx$
($m\ge2$). Our objective then is to implement conditions (i)–(iii)
by showing the solvability of the system 
\begin{equation}
f_{0}(\mu)=1,\ f_{1}(\mu)=1,\ f_{m}(\mu)=0\ (m\ge2)\label{eideq}
\end{equation}
in $F$. To this end, we employ a classical result of M.\ Eidelheit
(\cite[Satz 2]{Eid}). First, the family $(f_{m})_{m\in\N}$ is linearly
independent. Next, the topology of $F\subseteq\mathcal{S}(\R)$ is
generated by the family of norms $p_{k}(\mu)=\sup_{l+m\le k}\sup_{x\in\R}(1+|x|)^{l}|\mu^{(m)}(x)|$,
$k\in\N$. Suppose now that $\lambda_{1},\dots,\lambda_{i}$ are nonzero
numbers and that the \emph{order} of the linear functional $\sum_{m=0}^{i}\lambda_{m}f_{m}$
is less or equal $l$. Here, the order of an element $f$ of $\mathcal{S}'(\R)$
is defined to be the smallest $k$ such that $|f(\mu)|\le Cp_{k}(\mu)$
for some $C>0$ and all $\mu\in\mathcal{S}(\R)$. Let $i_{l}:=nl+1$,
then certainly $i\le i_{l}$. Hence both conditions of \cite[Satz 2]{Eid}
are satisfied and we may conclude that \eqref{eideq} has a solution
$\mu_{n}$ in $F$. 
\end{proof}
\begin{rem}
In addition to conditions \ref{enu:iColMoll}-\ref{enu:vColMoll}
from Lemma \ref{lem:ColombeauMollifier} we may require that $\mu_{n}$
satisfy finitely many additional properties expressible by linearly
independent functionals as in the above proof (again by \cite[Satz 2]{Eid}).
In particular, we may prescribe the values for $\mu_{n}$ or its derivatives
at finitely many further points.

Finally, we note that any element of $\mathcal{S}(\R)$ satisfying
condition \ref{enu:nullMoments} from Lemma \ref{lem:ColombeauMollifier}
must change sign infinitely often. 
\end{rem}

\noindent We call \emph{Colombeau mollifier} (for a fixed dimension
$n$) any function $\mu$ that satisfies the properties of the previous
lemma. Concerning embeddings of Schwartz distributions, the idea is
classically to regularize distributions using a mollifier. The use
of a Colombeau mollifier allows us, on the one hand, to identify the
distribution $\phi\in\D(\Omega)\mapsto\int f\phi$ with the GSF $f\in\Coo(\Omega)\subseteq\gsf(\Omega,\R)$
(thanks to property \ref{enu:nullMoments}); on the other hand, it
allows us to explicitly calculate compositions such as $\delta\circ\delta$,
$H\circ\delta$, $\delta\circ H$ (see below).

\noindent It is worth noting that the condition \ref{enu:nullMoments}
of null moments is well known in the study of convergence of numerical
solutions of singular differential equations, see e.g.~\cite{Ho-Ni-St16,En-To-Ts05,To-En04}
and references therein. 
\noindent \begin{center}
\begin{figure}
\label{fig: Col_mol} 
\begin{centering}
\includegraphics[scale=0.2]{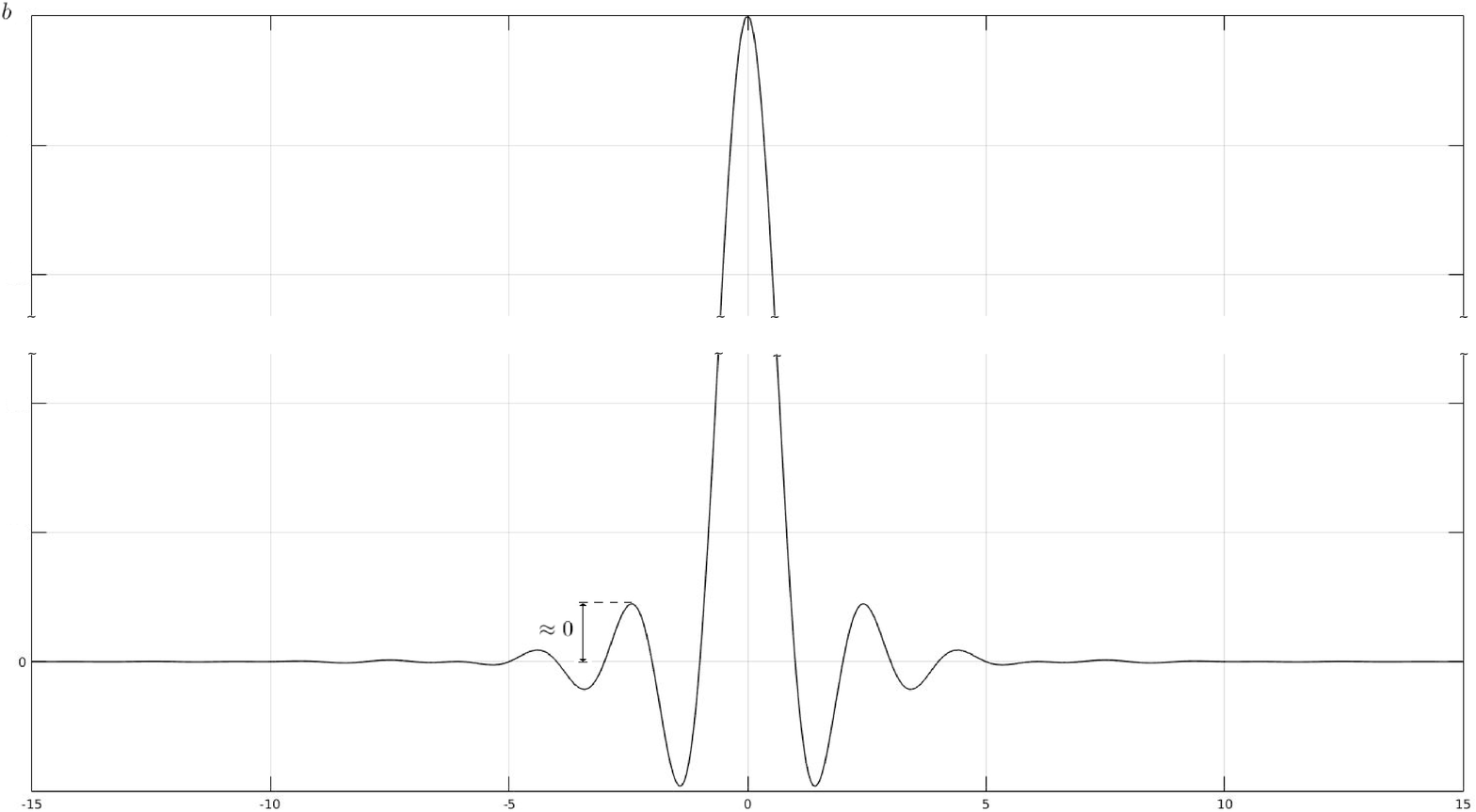}
\par\end{centering}
\begin{centering}
\includegraphics[scale=0.2]{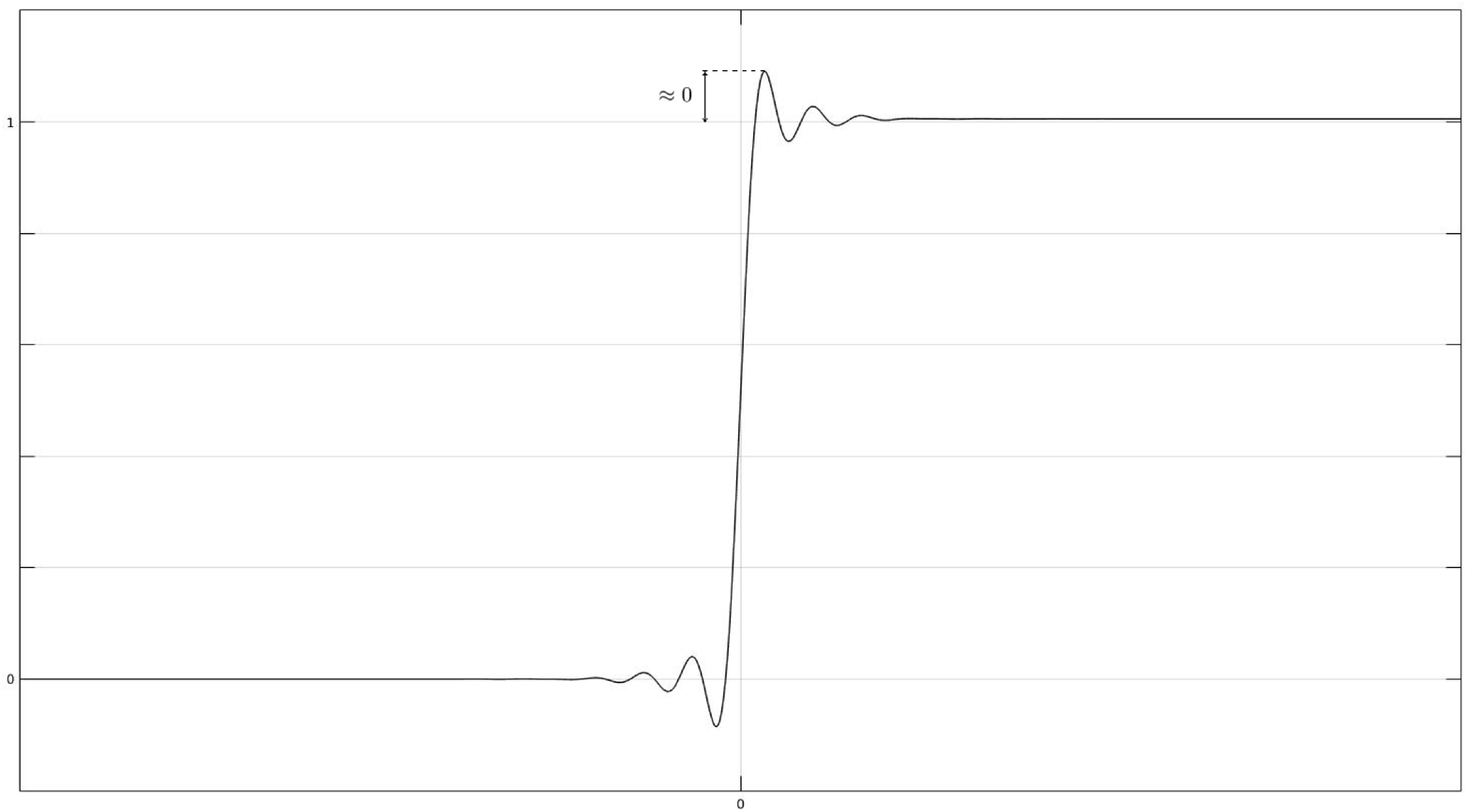} 
\par\end{centering}
\caption{\label{fig:MollifierHeaviside}A representation of Dirac delta and
Heaviside function. A Colombeau mollifier has a representation similar
to Dirac delta (but with finite values).}
\end{figure}
\par\end{center}

\noindent Next we show that the assignment $U\mapsto{}^{\rho}\Gcinf(\csp{U},\RC{\rho})$
($U\sse\Omega$ open) is a fine sheaf on $\Omega$. In fact, for $V\sse U$,
the natural restriction map $^{\rho}\Gcinf(\csp{U},\RC{\rho})\to\Gcinf(\csp{V},\RC{\rho})$
can also be written, in terms of defining nets, as $f=[f_{\eps}]\mapsto[f_{\eps}|_{V}]$.
Also, $\csp{U}\cap\csp{V}=\csp{U\cap V}$.

Suppose that $\Omega_{j}$ ($j\in J$) is an open covering of $\Omega$
and that for each $j\in J$ we are given $f^{j}=[f_{\eps}^{j}]\in{}^{\rho}\Gcinf(\csp{\Omega_{j}},\RC{\rho})$
such that $f^{j}|_{\csp{\Omega_{j}\cap\Omega_{k}}}=f^{k}|_{\csp{\Omega_{j}\cap\Omega_{k}}}$
for all $j$, $k\in J$. Then letting $\chi_{j}$ ($j\in J$) be a
partition of unity subordinate to $\Omega_{j}$ ($j\in J$), the GSF
defined by the net 
\[
f_{\eps}:=\sum_{j\in J}\chi_{j}\cdot f_{\eps}^{j}\in\cinfty(\Omega)
\]
is the unique element of $^{\rho}\Gcinf(\csp{\Omega},\RC{\rho})$
with $f|_{\csp{\Omega_{j}}}=f^{j}$ for all $j\in J$. In particular,
we may define a corresponding notion of standard support for each
$f\in\gsf(\csp{\Omega},\rcrho)$ by 
\[
\supp(f):=\left(\bigcup\left\{ \Omega'\subseteq\Omega\mid\Omega'\text{ open},\ f|_{\Omega'}=0\right\} \right)^{\text{c}}.
\]
The adjective \emph{standard} underscores that this set is made only
of standard points; a better notion of support for GSF is defined
as $\text{supp}(f)=\overline{\left\{ x\in X\mid|f(x)|>0\right\} }$
and studied in \cite{GK15}.

As a final preparation for the embedding of $\D'(\Omega)$ into $^{\rho}\Gcinf(\csp{\Omega},\RC{\rho})$
we need to construct suitable $n$-dimensional mollifiers from a Colombeau
mollifier $\mu$ as given by Lemma \ref{lem:ColombeauMollifier}.
To this end, let $\omega_{n}$ denote the surface area of $S^{n-1}$
and set 
\[
c_{n}:=\left\{ \begin{array}{lr}
\frac{2n}{\omega_{n}} & \text{for }n>1\\
1 & \text{for }n=1.
\end{array}\right.
\]
Then let $\tilde{\mu}:\R^{n}\to\R$, $\tilde{\mu}(x):=c_{n}\mu(|x|^{n})$.
Since $\mu$ is even, $\tilde{\mu}$ is smooth. Moreover, by Lemma
\ref{lem:ColombeauMollifier}, it has unit integral and all its higher
moments $\int x^{\alpha}\tilde{\mu}(x)\,dx$ vanish ($|\alpha|\ge1$).
With this notation we have: 
\begin{lem}
\label{lem:del} Let $\chi\in\D(\R^{n})$, $\chi=1$ on $\overline{\Eball_{1}(0)}$,
and $\chi=0$ on $\R^{n}\setminus\Eball_{2}(0)$. Also, let $b=[b_{\eps}]\in\rcrho$
be an infinite positive number, i.e.~$\lim_{\eps\to0^{+}}b_{\eps}=+\infty$.
Now set 
\begin{equation}
\mu_{\eps}^{b}(x):=(b_{\eps}^{-1}\odot\tilde{\mu})(x)\chi(x|\log(b_{\eps})|)=b_{\eps}^{n}\tilde{\mu}(b_{\eps}x)\chi(x|\log(b_{\eps})|).\label{col_mol}
\end{equation}
\begin{enumerate}
\item \label{deli} $\forall\eps:\ \mu_{\eps}^{b}\in\cinfty(\R^{n}),\ \supp(\mu_{\eps}^{b})\sse\Eball_{2|\log(b_{\eps})|^{-1}}(0)$. 
\item \label{delii} $\forall\alpha\in\N^{n}\,\exists N\in\N:\ \sup_{x\in\R^{n}}|\partial^{\alpha}\mu_{\eps}^{b}(x)|=O(b_{\eps}^{N})\ (\eps\to0)$. 
\item \label{deliii} $\forall\alpha\in\N^{n}\,\forall q\in\N:\ \sup_{x\in\R^{n}}|\partial^{\alpha}(\mu_{\eps}^{b}-b_{\eps}^{n}\tilde{\mu}(b_{\eps}\,.\,))(x)|=O(b_{\eps}^{-q})\ (\eps\to0)$. 
\item \label{deliv} $\forall q\in\N:\ \int\mu_{\eps}^{b}(x)\,\diff{x}=1+O(b_{\eps}^{-q})\ (\eps\to0)$. 
\item \label{delv} $\forall q\in\N\,\forall\alpha\in\N^{n}:\ |\alpha|>0\ \Rightarrow\ \int x^{\alpha}\mu_{\eps}^{b}(x)\,\diff{x}=O(b_{\eps}^{-q})\ (\eps\to0)$
. 
\end{enumerate}
\end{lem}

\begin{proof}
All the claimed properties have been proved for the special case $b_{\eps}=\eps^{-1}$
in \cite[Sec. 3]{Del}, and the arguments employed there carry over
in a straightforward way to the present setting. 
\end{proof}
\begin{thm}
\label{thm:embeddingD'} Let $(\emptyset\not=)\Omega\subseteq\R^{n}$
be an open set and let $\mu_{\eps}^{b}$ as in Lemma \ref{lem:del}.
Set $\Omega{}_{\eps}:=\left\{ x\in\Omega\mid d(x,\Omega^{c})\ge\eps,\ |x|\le\frac{1}{\eps}\right\} $
and fix some $\chi\in\D(\R^{n})$, $\chi=1$ on $\overline{\Eball_{1}(0)}$,
$0\le\chi\le1$ and $\chi=0$ on $\R^{n}\setminus\Eball_{2}(0)$.
Also, take $\kappa_{\eps}\in\D(\Omega)$ such that $\kappa_{\eps}=1$
on a neighborhood $L_{\eps}$ of $\Omega_{\eps}$. Then the map 
\begin{equation}
\iota_{\Omega}^{b}:T\in\D'(\Omega)\mapsto\left[\left(\left(\kappa_{\eps}\cdot T\right)*\mu_{\eps}^{b}\right)(-)\right]\in\gsf(\csp{\Omega},\RC{\rho}).\label{eq:ColEmb}
\end{equation}
satisfies: 
\begin{enumerate}
\item \label{enu:restriction}$\iota^{b}:\mathcal{D}'\ra\gsf(\csp{-},\rti)$
is a sheaf-morphism of real vector spaces: If $\Omega'\subseteq\Omega$
is another open set and $T\in\mathcal{D}'(\Omega)$, then $\iota_{\Omega}^{b}(T)|_{\csp{\Omega'}}=\iota_{\Omega'}^{b}(T|_{\Omega'})$. 
\item \label{enu:embInjective} $\iota^{b}$ preserves supports, hence is
in fact a sheaf-monomorphism. 
\item \label{enu:embeddingSmoothUpToInfinitesimals} Any $f\in\cinfty(\Omega)$
can naturally be considered an element of $\gsf(\csp{\Omega},\RC{\rho})$
via $[x_{\eps}]\mapsto[f(x_{\eps})]$. Moreover, $\forall q\in\N_{>0}\ \forall x\in\csp{\Omega}:\ \left|\iota_{\Omega}^{b}(f)(x)-f(x)\right|\le b^{-q}$. 
\item \label{enu:smoothEmb} If $f\in\cinfty(\Omega)$ and if $b\ge\diff{\rho}^{-a}$
for some $a\in\R_{>0}$, then $\iota_{\Omega}^{b}(f)=f$. In particular,
$\iota^{b}$ then provides a multiplicative sheaf-monomorphism $\Coo(-)$
$\hookrightarrow\gsf(c(-),\R)$. 
\item \label{enu:commute_der} For any $T\in\D'(\Omega)$ and any $\alpha\in\N^{n}$,
$\iota_{\Omega}^{b}(\partial^{\alpha}T)=\partial^{\alpha}\iota_{\Omega}^{b}(T)$. 
\item \label{enu:valuesDistr}Let $b\ge\diff{\rho}^{-a}$ for some $a\in\R_{>0}$.
Then for any $\phi\in\mathcal{D}(\Omega)$ and any $T\in\mathcal{D}'(\Omega)$,
\[
\big[\int_{\Omega}\iota_{\Omega}^{b}(T)_{\eps}(x)\cdot\phi(x)\,\diff{x}\big]=\langle T,\phi\rangle\quad\text{in }\rcrho.
\]
\item \label{enu:deltaH}$\iota_{\R^{n}}^{b}(\delta)(0)=c_{n}b^{n}$ and
if $b\ge\diff{\rho}^{-a}$ for some $a\in\R_{>0}$, then $\iota_{\R}^{b}(H)(0)=\frac{1}{2}$. 
\item \label{enu:depXi-e}The embedding $\iota^{b}$ does not depend on
the particular choice of $(\kappa_{\eps})$ and (if $b\ge\diff{\rho}^{-a}$
for some $a\in\R_{>0}$) $\chi$ as above. 
\item \label{enu:iota_indep_of_repr} $\iota^{b}$ does not depend on the
representative $(b_{\eps})$ of $b$ employed in \eqref{col_mol}. 
\end{enumerate}
\end{thm}

\begin{proof}
We follow ideas from \cite[Sec. 1.2]{GKOS} and \cite{Del}. Let $T\in\mathcal{D}'(\Omega)$
and let $[x_{\eps}]\in\csp{\Omega}$. Then there exists some $K\comp\Omega$
such that $x_{\eps}\in K$ for $\eps$ small. Also, we may assume
that $K+\Eball_{2|\log(b_{\eps})|^{-1}}(0)\sse L\sse\Omega_{\eps}$
for these $\eps$, where $L\comp\Omega$. Then by \ref{deli} of Lemma
\ref{lem:del}, for $\eps$ small we have 
\begin{equation}
\iota_{\Omega}^{b}(T)_{\eps}(x_{\eps})=\left(\kappa_{\eps}\cdot T\right)*\mu_{\eps}^{b}(x_{\eps})=T*\mu_{\eps}^{b}(x_{\eps})=\langle T,\mu_{\eps}^{b}(x_{\eps}-\,.\,)\rangle.\label{iotom}
\end{equation}
Since $T\in\D'(\Omega)$, we have a seminorm estimate of the form
\[
\forall\vphi\in\D_{L}(\Omega):\ |\langle T,\vphi\rangle|\le C\max_{|\beta|\le m}\sup_{x\in L}|\partial^{\beta}\vphi(x)|.
\]
Together with Lemma \ref{lem:del} \ref{deli} and \ref{delii} this
implies that $(\partial^{\alpha}\iota_{\Omega}^{b}(T)_{\eps}(x_{\eps}))\in\R_{\rho}^{n}$
for each $\alpha$. Hence $\iota_{\Omega}^{b}$ indeed maps $\D'(\Omega)$
into $^{\rho}\Gcinf(\csp{\Omega},\RC{\rho})$.

To show \ref{enu:restriction}, let $\Omega'\sse\Omega$ be open and
let $[x_{\eps}]\in c(\Omega')$. Then using the notations introduced
before \eqref{iotom}, we may suppose that $L\sse\Omega'_{\eps}$,
and so for $\eps$ small we have $\mu_{\eps}^{b}\in\D(\Omega')$.
Therefore, \eqref{iotom} implies for such $\eps$: 
\[
\iota_{\Omega}^{b}(T)_{\eps}(x_{\eps})=\langle T,\mu_{\eps}^{b}(x_{\eps}-\,.\,)\rangle=\langle T|_{\Omega'},\mu_{\eps}^{b}(x_{\eps}-\,.\,)\rangle=\iota_{\Omega'}^{b}(T|_{\Omega'})_{\eps}(x_{\eps}).
\]

Next we show \ref{enu:embInjective}. Suppose first that $T|_{\Omega'}=0$
for some open subset $\Omega'$ of $\Omega$. Let $[x_{\eps}]\in\csp{\Omega'}$
and pick $K\comp\Omega'$ such that $x_{\eps}\in K$ for $\eps$ small.
As above, for $\eps$ small we have $\supp(\mu_{\eps}^{b}(x_{\eps}-\,.\,))\sse\Omega'$,
as well as $\iota_{\Omega}^{b}(T)_{\eps}(x_{\eps})=\langle T,\mu_{\eps}^{b}(x_{\eps}-\,.\,)\rangle$,
which therefore vanishes. Hence $\iota_{\Omega}^{b}(T)|_{\Omega'}=0$,
and thereby $\supp(\iota_{\Omega}^{b}(T))\sse\supp(T)$.

Conversely, let $\Omega'\sse\Omega$ such that $\iota_{\Omega}^{b}(T)|_{\Omega'}=0$
and let $\vphi\in\D(\Omega').$ Since $(\kappa_{\eps}T)*\mu_{\eps}^{b}\to T$
in $\D'(\Omega)$, in order to show $\langle T,\vphi\rangle=0$ it
suffices to demonstrate that $(\kappa_{\eps}T)*\mu_{\eps}^{b}\to0$
as $\eps\to0$, uniformly on compact subsets of $\Omega'$. Suppose
this were not the case, then we could find some $L\comp\Omega'$,
some $c>0$ and sequences $\eps_{k}\downarrow0$ and $x_{k}\in L$
such that $|(\kappa_{\eps}T)*\mu_{\eps_{k}}^{b}(x_{k})|\ge c$ for
all $k$. Fixing any $z\in\Omega'$ and setting $x_{\eps}:=x_{k}$
for $\eps=\eps_{k}$ and $x_{\eps}=z$ otherwise then defines an element
$[x_{\eps}]\in\csp{\Omega'}$ with $\iota_{\Omega}^{b}(T)([x_{\eps}])\not=0$,
a contradiction.

Consequently, $\iota_{\Omega}^{b}$ induces an injective sheaf morphism
(again denoted by) $\iota^{b}:\D'(-)\ra{}^{\rho}\Gcinf(\csp{-},\RC{\rho})$.

\ref{enu:embeddingSmoothUpToInfinitesimals}: If $f\in\cinfty(\Omega)$
then any derivative of $f$ is uniformly (in $\eps$) bounded on any
$(x_{\eps})$ (for $[x_{\eps}]\in c(\Omega)$). Thus $f\in\Gcinf(\csp{\Omega},\RC{\rho})$.
Now let $[x_{\eps}]\in\csp{\Omega}$ and suppose first that $f$ has
compact support. By Lemma \ref{lem:del} \ref{deliv}, for any $x\in\Omega$,
$f(x)=\int f(x)\mu_{\eps}^{b}(y)\,dy+n_{\eps}$, where $n_{\eps}=O(b_{\eps}^{-q})$
for every $q>0$. Thus for $\eps$ small and any $q\in\N$ we have
by Taylor expansion 
\begin{equation}
\begin{split} & (\iota_{\Omega}^{b}(f)_{\eps}-f)(x_{\eps})=\int(f(x_{\eps}-y)-f(x_{\eps}))\mu_{\eps}^{b}(y)\,\diff{y}+n_{\eps}\\
 & =\sum_{k=1}^{q-1}\int\frac{1}{k!}((-y\cdot D)^{k}f)(x_{\eps})\mu_{\eps}^{b}(y)\,\diff{y}\\
 & +\frac{b_{\eps}^{-q}}{q!}\int((-y\cdot D)^{q}f)(x_{\eps}-\theta_{\eps}b_{\eps}^{-1}y)\tilde{\mu}(y)\chi(b_{\eps}^{-1}y|\log(b_{\eps})|)\,\diff{y}+n_{\eps},
\end{split}
\label{taylorarg}
\end{equation}
where $\theta_{\eps}\in(0,1)$. Here, the first sum is $O(b_{\eps}^{-q})$
by Lemma \ref{lem:del} \ref{delv}, as is the second term since $f$
is compactly supported, $\chi$ is globally bounded, and $\tilde{\mu}\in{\mathcal{S}}(\R^{n})$.
If $f$ is not compactly supported, pick $L\comp\Omega$ such that
$x_{\eps}\in L$ for $\eps$ small and let $\vphi\in\D(\Omega)$ equal
$1$ in a neighborhood of $L$. Then $f(x)=(\vphi f)(x)$ and \ref{enu:embInjective}
implies that $\iota_{\Omega}^{b}(f)(x)=\iota_{\Omega}^{b}(\vphi f)(x)$,
so the general case follows as well.

\ref{enu:smoothEmb}: It suffices to observe that, by our assumption
on $b$, \ref{enu:embeddingSmoothUpToInfinitesimals} implies that
$\iota_{\Omega}^{b}(f)([x_{\eps}])=[f(x_{\eps})]=f(x)$ for any $f\in\cinfty(\Omega)$
and any $x=[x_{\eps}]\in\csp{\Omega}$.

\ref{enu:commute_der}: As in the proof of \ref{enu:smoothEmb} we
may assume that $T$ has compact support. Then for $\eps$ small we
have 
\[
\iota_{\Omega}^{b}(\partial^{\alpha}T)_{\eps}=\partial^{\alpha}T*\mu_{\eps}^{b}=\partial^{\alpha}(T*\mu_{\eps}^{b})=\partial^{\alpha}\iota_{\Omega}^{b}(T)_{\eps}.
\]

\ref{enu:valuesDistr}: Pick $\zeta\in\D(\Omega)$ such that $\zeta\equiv1$
on a neighborhood of $\supp(\vphi)$. Then 
\[
\supp(\iota_{\Omega}^{b}(T)-\iota_{\Omega}^{b}(\zeta T))\cap\supp\vphi=\supp(T-\zeta T)\cap\supp\vphi=\emptyset,
\]
so we may replace $T$ by $\zeta T$, i.e., we may assume without
loss of generality that $T\in{\mathcal{E}}'(\Omega)$. By the representation
theorem of distribution theory $T$ then is a finite sum of terms
of the form $\partial^{\alpha}f$ with $f\in\mathcal{C}^{0}(\Omega)$
compactly supported in $\Omega$, so it will suffice to treat the
case $T=\partial^{\alpha}f$. For any $\varphi\in\mathcal{D}(\Om)$
we have 
\begin{align*}
\int(\iota_{\Omega}^{b}(\partial^{\alpha}f)_{\eps}- & \partial^{\alpha}f)(x)\varphi(x)\,\diff{x}=\int\int(\partial^{\alpha}f(x-y)-\partial^{\alpha}f(x))\mu_{\eps}^{b}(y)\vphi(x)\,\,\diff{y}\diff{x}+n_{\eps}\\
 & =\int\partial^{\alpha}f(x)\int\mu_{\eps}^{b}(y)(\vphi(x+y)-\vphi(x))\,\,\diff{y}\diff{x}+n_{\eps}.
\end{align*}
with $n_{\eps}=O(b_{\eps}^{-q})$ for any $q\in\N$ by Lemma \ref{lem:del}
\ref{delv}. As in the proof of \ref{enu:embeddingSmoothUpToInfinitesimals}
it follows that also the integral term in the above equality is of
order $O(b_{\eps}^{-q})$, giving the claim due to our assumption
on $b$.

\ref{enu:deltaH}: The first claim is immediate from $\iota_{\R}^{b}(\delta)_{\eps}(0)=\mu_{\eps}^{b}(0)$.
To show the second, note first that 
\[
\iota_{\R}^{b}(H)_{\eps}(0)-\int H(y)\mu_{\eps}^{b}(-y)\,\diff{y}=\int H(y)(\kappa_{\eps}(y)-1)\mu_{\eps}^{b}(-y)\,\diff{y}=0
\]
for $\eps$ small by the support properties of $\kappa_{\eps}$ and
$\chi$. Furthermore, since $\int_{0}^{\infty}\mu(y)\,dy=1/2$, we
obtain 
\[
\begin{split}\Big|\int H(y)\mu_{\eps}^{b}(-y)\,\diff{y}-\frac{1}{2}\Big| & =\Big|\int_{0}^{\infty}\mu(y)(\chi(b_{\eps}^{-1}|\log(b_{\eps})|)y)-1)\,\diff{y}\Big|\\
 & \le\int_{b_{\eps}|\log(b_{\eps})|^{-1}}^{\infty}|\mu(y)|\,\diff{y}=O(b_{\eps}^{-q})
\end{split}
\]
for any $q\in\N$, so the claim follows.

\ref{enu:depXi-e}: We first note that any two choices for either
$(\kappa_{\eps})$ or $\chi$ provide sheaf morphisms as in \ref{enu:restriction},
\ref{enu:embInjective}, hence it suffices to check that the resulting
embeddings coincide on compactly supported distributions. For any
such $T$ we have $\kappa_{\eps}T=T$ for $\eps$ small, so independence
from the choice of $(\kappa_{\eps})$ follows.

Now suppose that two different $\chi$'s have been chosen and denote
the corresponding functions from \eqref{col_mol} by $\mu_{\eps}^{b}$
and $\bar{\mu}_{\eps}^{b}$, and the resulting embeddings by $\iota^{b}$
and $\bar{\iota}^{b}$, respectively. Since $T\in{\mathcal{E}}'(\Omega)$,
it satisfies a seminorm estimate of the form 
\begin{equation}
\forall\vphi\in\cinfty(\Omega):\ |\langle T,\vphi\rangle|\le C\max_{|\beta|\le m}\sup_{x\in L}|\partial^{\beta}\vphi(x)|.\label{eq:comp_supp_dist}
\end{equation}
for some $L\comp\Omega$. Together with Lemma \ref{lem:del} \ref{deliii}
this implies that, for any $[x_{\eps}]\in c(\Omega)$ and $\eps$
small, we have 
\[
|(\iota_{\Omega}^{b}(T)_{\eps}-\bar{\iota}_{\Omega}^{b}(T)_{\eps})(x_{\eps})|=|\langle T,(\mu_{\eps}^{b}-\bar{\mu}_{\eps}^{b})(x_{\eps}-\,.\,)\rangle|=O(b_{\eps}^{-q})
\]
for any $q\in\N$.

\ref{enu:iota_indep_of_repr}: Let $(c_{\eps})$ be another representative
of $b$, so $(c_{\eps})\sim_{\rho}(b_{\eps})$. As in the proof of
\ref{enu:depXi-e} it then suffices to show that $\iota^{b}(T)=\iota^{c}(T)$
for any $T\in\mathcal{E}'(\Omega)$. Given $x=[x_{\eps}]\in c(\Omega)$,
let $K\comp\Omega$ be such that $x_{\eps}\in K$ for $\eps$ small.
Then by \eqref{eq:comp_supp_dist} and \eqref{eq:ColEmb}, 
\[
|(\iota^{b}(T)-\iota^{c}(T))(x_{\eps})|\le C\max_{|\beta|\le m}\sup_{x\in K-L}|\partial^{\beta}(\mu_{\eps}^{b}-\mu_{\eps}^{c})(x)|.
\]
Inserting from \eqref{col_mol} it follows by a straightforward estimate
that the right hand side here is of order $O(\rho_{\eps}^{q})$ for
any $q\in\N$, proving the claim. 
\end{proof}
Whenever we use the notation $\iota^{b}$ for an embedding, we assume
that $b\in\rcrho$ satisfies the overall assumptions of Thm.~\ref{thm:embeddingD'}
and of \ref{enu:smoothEmb} in that Theorem, and that $\iota^{b}$
has been defined as in \eqref{eq:ColEmb} using a Colombeau mollifier
$\mu$ for the given dimension. Note in particular that by Theorem
\ref{thm:embeddingD'} \ref{enu:iota_indep_of_repr} we are justified
in using the short hand notation $\iota^{b}$ for the embedding defined
via any representative $(b_{\eps})$ of $b$. 
\begin{rem}
\label{rem:embedding_properties}~ 
\begin{enumerate}
\item In Def.~\ref{def:RCGN}, we introduced the asymptotic gauge $\mathcal{I}(\rho)$,
and the entire construction depends on the fixed infinitesimal net
$\rho$ only through this set $\mathcal{I}(\rho)$. A more general
definition of asymptotic gauge is possible (see \cite{GiLu15}). Anyhow,
\cite[Sec.~4.3]{GiLu15} shows that an embedding of Schwartz's distribution
having certain minimal properties necessarily requires that the asymptotic
gauge be generated by a single net, as is the case for $\mathcal{I}(\rho)$. 
\item Let $\delta$, $H\in\gsf(\rcrho,\rcrho)$ be the corresponding $\iota^{b}$-embeddings
of the Dirac delta and of the Heaviside function. Then $\delta(x)=b\cdot\mu(b\cdot x)$
and $\delta(x)=0$ if $x$ is near-standard and $\st{x}\ne0$ or if
$x$ is infinite because $\mu\in\mathcal{S}(\R)$. Also, by construction
of $\mu_{\eps}^{b}$, $\delta$ can be represented like in the first
diagram of Fig.~\ref{fig:MollifierHeaviside}. E.g., $\delta(k/b)=0$
for each $k\in\Z\setminus\{0\}$, and each $\frac{k}{b}$ is a nonzero
infinitesimal. Similar properties can be stated e.g.~for $\delta^{2}(x)=b^{2}\cdot\mu(b\cdot x)^{2}$. 
\item Analogously, we have $H(x)=1$ if $x$ is near-standard and $\st{x}>0$
or if $x>0$ is infinite; $H(x)=0$ if $x$ is near-standard and $\st{x}<0$
or if $x<0$ is infinite. 
\item In \cite{Loj}, S.~Łojasiewicz introduced the notion of a point value
for distributions. He defined that $T\in\D'(\Omega)$ has the point
value $c\in\CC$ in $x_{0}\in\Omega$ if 
\begin{equation}
\lim_{\eps\to0}\langle T(x_{0}+\eps x),\vphi(x)\rangle=c\int\vphi(x)\,\diff{x}\qquad\forall\vphi\in\D(\Omega).\label{loj}
\end{equation}
Not every distribution has point values in arbitrary points — in fact,
if it does, it already has to be a function of first Baire class (\cite{Loj}).
Conversely, a continuous function $f$ clearly has point value $f(x)$
in any point $x$ in its domain.

We show that if $T$ has point value $c$ at $x_{0}\in\Omega$ then
$\iota_{\Omega}^{b}(T)_{\eps}(x_{0})\to c$ as $\eps\to0$. In fact,
since $\mathcal{S}'(\R^{n})$ is a normal space of distributions that
is invariant under translations, by \cite[Prop. 7]{Sh-67} this follows
if for any sequence $\eps_{k}\downarrow0$, the functions $g_{k}:=\mu_{\eps_{k}}^{b}$
satisfy the following conditions: 
\begin{itemize}
\item[(a)] $\int g_{k}(x)\,\diff{x}\to1$, and $\forall\eta>0$: $\int_{|x|\ge\eta}g_{k}(x)\,\diff{x}\to0$
as $k\to\infty$. 
\item[(b)] For each $\alpha\in\D(\R^{n})$ that is $1$ on a neighborhood of
$0$, $(1-\alpha)g_{k}\to0$ in $\mathcal{S}'(\R^{n})$ for $k\to\infty$. 
\item[(c)] For each $\alpha\in\N^{n}$ there exists some $M_{\alpha}>0$ such
that, for any $\eta>0$: $\int_{|x|\le\eta}|x|^{|\alpha|}|\partial^{\alpha}g_{k}(x)|\,\diff{x}\le M_{\alpha}$. 
\end{itemize}
Indeed, all these properties follow readily from \eqref{col_mol}. 
\item \label{rem_embedding_col}Colombeau's special (or simplified) algebra
$\mathcal{G}$ (\cite{C1,C2,MObook,GKOS}) is defined, for $\Om\sse\R^{n}$
open, as the quotient $\gs(\Om):=\esm(\Om)/\ns(\Om)$ of \emph{moderate
nets} modulo \emph{negligible nets}, where 
\begin{align*}
\esm(\Om):=\{(u_{\eps})\in\cinfty(\Omega)^{I}\mid\forall K\comp\Om\, & \forall\alpha\in\N^{n}\,\exists N\in\N:\\
 & \sup_{x\in K}|\partial^{\alpha}u_{\eps}(x)|=O(\eps^{-N})\}
\end{align*}
and 
\begin{align*}
\ns(\Om):=\{(u_{\eps})\in\cinfty(\Omega)^{I}\mid\forall K\comp\Om\, & \forall\alpha\in\N^{n}\,\forall m\in\N:\\
 & \sup_{x\in K}|\partial^{\alpha}u_{\eps}(x)|=O(\eps^{m})\}.
\end{align*}
It follows from \cite[Th. 37]{GKV} that $\gs(\Om)$ can be identified
with the algebra $^{\rho}\Gcinf(\csp{\Omega},\RC{\rho})$ in the special
case of $\rho(\eps)=\eps$. In this setting, Theorem \ref{thm:embeddingD'}
gives an alternative proof of the well known facts that the Colombeau
algebra contains $\cinfty(\Om)$ as a faithful subalgebra, $\D'(\Om)$
as a linear subspace and that the embedding is a sheaf morphism that
commutes with partial derivatives. An alternative point of view is
that Colombeau generalized functions correspond to those generalized
smooth functions that are defined on compactly supported generalized
points. As was already mentioned, more general domains are both useful
in applications and are indeed a necessary requirement for obtaining
a construction that is closed with respect to composition of generalized
functions. 
\end{enumerate}
\end{rem}

We close this section by considering the following natural problem:
let us define two embeddings $\iota_{\Omega}^{b}$, $\iota_{\Omega}^{c}$
as in \eqref{eq:ColEmb}, but using two different infinite positive
numbers $b$, $c\in\rcrho$, so that for all $T\in\mathcal{E}'(\Omega)$
we have 
\begin{align*}
\iota_{\Omega}^{b}(T) & :=\left[T\ast\mu_{\eps}^{b}\right],\\
\iota_{\Omega}^{c}(T) & :=\left[T\ast\mu_{\eps}^{c}\right].
\end{align*}
The following result characterizes equality of such embeddings. 
\begin{thm}
Let $b$, $c\in\rcrho$ be infinite positive numbers and let $\mu$
be a Colombeau mollifier for dimension $n$. Let $\Omega\sse\R^{n}$
be open. Then $\iota_{\Omega}^{b}=\iota_{\Omega}^{c}$ if and only
if $b=c$ in $\rcrho$, i.e.~if and only if they are equal as Robinson-Colombeau
generalized number. 
\end{thm}

\begin{proof}
By Theorem \ref{thm:embeddingD'} \ref{enu:iota_indep_of_repr}, $\iota^{b}$
is well-defined, i.e., does not depend on the representative of $b\in\rcrho$.
Conversely, suppose that $\iota_{\Omega}^{b}=\iota_{\Omega}^{c}$
and fix any $x_{0}\in\Omega$. Then in particular $\iota_{\Omega}^{b}(\delta_{x_{0}})=\iota_{\Omega}^{c}(\delta_{x_{0}})$
in $\gsf(\csp{\Omega},\rcrho)$. Due to \eqref{col_mol}, \eqref{eq:ColEmb},
an evaluation of these GSF at $x_{0}$ implies 
\[
\forall m\in\N:\ \left|(b_{\eps}^{n}-c_{\eps}^{n})c_{n}\right|=O(\rho_{\eps}^{m}),
\]
so $b=c$ in $\rcrho$. 
\end{proof}

\subsubsection{\label{subsec:ClosureComposition}Closure with respect to composition}

In contrast to the case of distributions, there is no problem in considering
the composition of two GSF. This property opens new interesting possibilities,
e.g.~in considering differential equations $y'=f(y,t)$, where $y$
and $f$ are GSF. For instance, there is no problem in studying $y'=\delta(y)$
(see \cite{Lu-Gi16}). 
\begin{thm}
\label{thm:GSFcategory} Subsets $S\subseteq\RC{\rho}^{s}$ with the
trace of the sharp topology, and generalized smooth maps as arrows
form a subcategory of the category of topological spaces. We will
call this category $\gsf$, the \emph{category of GSF}. 
\end{thm}

\begin{proof}
From Thm.\ \ref{thm:GSF-continuity} \ref{enu:GSF-cont} we already
know that every GSF is continuous; we have hence to prove that these
arrows are closed with respect to identity and composition in order
to obtain a concrete subcategory of topological spaces and continuous
maps.

\noindent If $T\subseteq\RC{\rho}^{t}$ is an arbitrary object, then
$f_{\eps}(x):=x$ is the net of smooth functions that globally defines
the identity $1_{T}$ on $T$. It is immediate that $1_{T}$ is generalized
smooth.

\noindent To prove that arrows of $^{\rho}\Gcinf$ are closed with
respect to composition, let $S\subseteq\RC{\rho}^{s},T\subseteq\RC{\rho}^{t},R\subseteq\RC{\rho}^{r}$
and $f:S\longrightarrow T$, $g:T\longrightarrow R$ be GSF, then
$f(x)=[f_{\eps}(x_{\eps})]\in T$ and $g(y)=[g_{\eps}(y_{\eps})]\in R$
for every $x\in S$ and $y\in T$, where $f_{\eps}\in\cinfty(\Omega'_{\eps},\R^{t})$
and $g_{\eps}\in\cinfty(\Omega''_{\eps},\R^{r})$ are suitable nets
of smooth functions as in Def.\ \ref{def:netDefMap}, and where $\Omega'_{\eps}$
is open in $\R^{s}$ and $\Omega''_{\eps}$ is open in $\R^{t}$.
Of course, the idea is to consider $g_{\eps}\circ f_{\eps}\in\cinfty(\Omega_{\eps},\R^{r})$,
where $\Omega_{\eps}:=f_{\eps}^{-1}(\Omega''_{\eps})$ (let us note
that, even in the case where $\Omega''_{\eps}$ does not depend by
$\eps$, generally speaking $\Omega_{\eps}$ still depends on $\eps$).\\
 Take $x\in S$, so that $f(x)=[f_{\eps}(x_{\eps})]\in T\subseteq\sint{\Omega''_{\eps}}$
and hence $f_{\eps}(x_{\eps})\in_{\eps}\Omega''_{\eps}$ and $x_{\eps}\in\Omega_{\eps}$
for $\eps$ small. If we take another representative $(x'_{\eps})\sim_{\rho}(x_{\eps})$
we have $f(x')=f(x)$ since $f$ is well-defined and, proceeding as
before, we still have that $x'_{\eps}\in\Omega_{\eps}$ for $\eps$
sufficiently small. This proves that $S\subseteq\sint{\Omega_{\eps}}$.
Moreover, since $[f_{\eps}(x_{\eps})]\in T$, we also have that $[g_{\eps}(f_{\eps}(x_{\eps}))]\in R$
and $g(f(x))=[g_{\eps}(f_{\eps}(x_{\eps}))]$. It remains to show
that the net $(g_{\eps}\circ f_{\eps})$ defines a GSF (Def.\ \ref{def:netDefMap})
of the type $S\longrightarrow R$.\\
 To this end, let us consider any $[x_{\eps}]\in S$ and any $\gamma\in\N^{s}$.
We can write 
\begin{equation}
\partial^{\gamma}(g_{\eps}\circ f_{\eps})(x_{\eps})=p\left[\partial^{\alpha_{1}}f_{\eps}(x_{\eps}),\ldots,\partial^{\alpha_{A}}f_{\eps}(x_{\eps}),\partial^{\beta_{1}}g_{\eps}(f_{\eps}(x_{\eps})),\ldots,\partial^{\beta_{B}}g_{\eps}(f_{\eps}(x_{\eps}))\right],\label{eq:derivativeOfComposition}
\end{equation}
where $p$ is a suitable polynomial (from the Faà di Bruno formula)
not depending on $x_{\eps}$. Every term $\partial^{\alpha_{i}}f_{\eps}(x_{\eps})$
and $\partial^{\beta_{j}}g_{\eps}(f_{\eps}(x_{\eps}))$ is $\rho$-moderate
by \ref{enu:partial-u-moderate} of Def.\ \ref{def:netDefMap}. Since
moderateness is preserved by polynomial operations, it follows that
also $\partial^{\gamma}(g_{\eps}\circ f_{\eps})(x_{\eps})$ is $\rho$-moderate. 
\end{proof}
\noindent For instance, we can think of the Dirac delta as a map of
the form $\delta:\RC{\rho}\longrightarrow\RC{\rho}$, and therefore
the composition $e^{\delta}$ is defined in $\{x\in\RC{\rho}\mid\exists z\in\RC{\rho}_{>0}:\ \delta(x)\le\log z\}$,
which of course does not contain $x=0$ but only suitable non zero
infinitesimals. On the other hand, $\delta\circ\delta:\RC{\rho}\ra\RC{\rho}$.
Moreover, from the inclusion of ordinary smooth functions (Thm.~\ref{thm:embeddingD'})
and the closure with respect to composition, it directly follows that
every $^{\rho}\Gcinf(U,\RC{\rho})$ is an algebra with pointwise operations
for every subset $U\subseteq\RC{\rho}^{n}$. For an open subset $\Omega\subseteq\R^{n}$,
the algebra $^{\rho}\Gcinf(\csp{\Omega},\RC{\rho})$ contains the
space $\mathcal{D}'(\Omega)$ of Schwartz distributions.

A natural way to define a GSF is to follow the original idea of classical
authors (see \cite{Kat-Tal12,Lau89,Dir26}) to fix an infinitesimal
or infinite parameter in a suitable ordinary smooth function. We will
call this type of GSF of \emph{Cauchy-Dirac type}; the next theorem
specifies this notion and states that GSF are of Cauchy-Dirac type
whenever the generating net $(f_{\eps})$ is smooth in $\eps$. 
\begin{cor}
\label{cor:quasi-stdAreGSF}Let $X\subseteq\R^{n}$, $Y\subseteq\R^{d}$,
$P\subseteq\R^{m}$ be open sets and $\phi\in\cinfty(P\times X,Y)$
be an ordinary smooth function. Let $p\in[P]$, and define $f_{\eps}:=\phi(p_{\eps},-)\in\cinfty(X,Y)$,
then $[f_{\eps}(-)]:[X]\longrightarrow[Y]$ is a GSF. In particular,
if $f:[X]\longrightarrow[Y]$ is a GSF defined by $(f_{\eps})$ and
the net $(f_{\eps})$ is smooth in $\eps$, i.e.\ if 
\[
\exists\phi\in\cinfty((0,1)\times X,Y):\ f_{\eps}=\phi(\eps,-)\quad\forall\eps\in(0,1),
\]
and if $[\eps]\in\rcrho$, then the GSF $f$ is of Cauchy-Dirac type
because $f(x)=\phi([\eps],x)$ for all $x\in[X]$. Finally, Cauchy-Dirac
GSF are closed with respect to composition. 
\end{cor}

\begin{proof}
In fact, the map $x\in[X]\mapsto(p,x)\in[P]\times[X]$ is trivially
generalized smooth and hence from the inclusion of smooth functions
(Theorem \ref{thm:embeddingD'}) and the closure with respect to composition
(Theorem \ref{thm:GSFcategory}) the conclusions follow. 
\end{proof}
\begin{example}
\label{enu:deltaCompDelta}The composition $\delta\circ\delta\in\gsf(\rcrho,\rcrho)$
is given by $(\delta\circ\delta)(x)=b\mu\left(b^{2}\mu(bx)\right)$
and is an even function. If $x$ is near-standard and $\st{x}\ne0$,
or $x$ is infinite, then $(\delta\circ\delta)(x)=b$. Since $(\delta\circ\delta)(0)=0$,
by the intermediate value theorem (see Cor\@.~\ref{cor:intermValue}
below), we have that $\delta\circ\delta$ attains any value in the
interval $[0,b]\subseteq\rcrho$. If $0\le x\le\frac{1}{2b}$, then
(for a $\mu$ as in Fig.\ \ref{fig: Col_mol}) $x$ is infinitesimal
and $(\delta\circ\delta)(x)=0$ because $\delta(x)\ge b\mu\left(\frac{1}{k}\right)$
is an infinite number. If $x=\frac{k}{b}$ for some $k\in\N_{>0}$,
then $x$ is still infinitesimal but $(\delta\circ\delta)(x)=b$ because
$\mu(bx)=0$. A representation of $\delta\circ\delta$ is given in
Fig.~\ref{fig:deltaCompDelta}. Analogously, one can deal with $H\circ\delta$
and $\delta\circ H$. 
\end{example}

\begin{figure}
\begin{centering}
\includegraphics[scale=0.2]{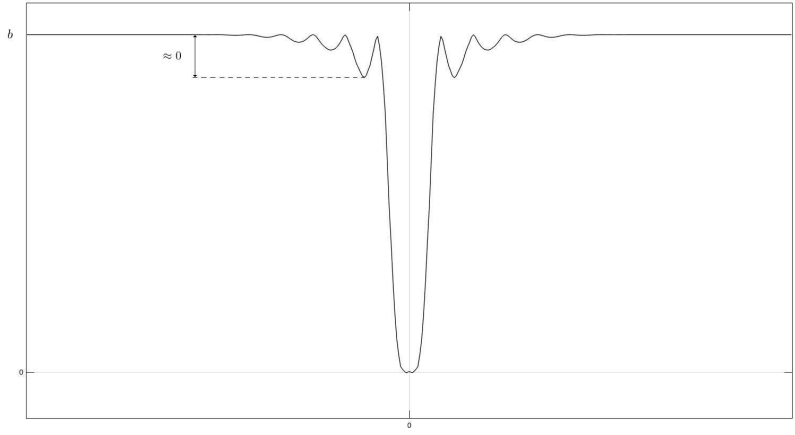} 
\par\end{centering}
\caption{\label{fig:deltaCompDelta}A representation of $\delta\circ\delta$}
\end{figure}

The theory of GSF originates from the theory of Colombeau quotient
algebras. In this well-developed approach, strong analytic tools,
including microlocal analysis, and an elaborate theory of pseudodifferential
and Fourier integral operators have been developed over the past few
years (cf.~\cite{C1,C2,MObook,GKOS,HOP,GHO} and references therein).
In these quotient algebras, each generalized function generates a
unique GSF defined on a subset of $\RC{(\eps)}$. On the other hand,
Colombeau generalized functions are in general not closed with respect
to composition because they cannot be defined on arbitrary domains
$X\subseteq\RC{\rho}^{n}$. We refer to \cite{GKV} for details about
the links between Colombeau algebras and GSF, and to \cite{Tod11,Tod13,ToVe08}
for a treatment of Colombeau algebras in the framework of nonstandard
analysis.

\section{\label{sec:Differential-calculus}Differential calculus and the Fermat-Reyes
theorem}

In this section we show how the derivatives of a GSF can be calculated
using a form of incremental ratio. The idea is to prove the Fermat-Reyes
theorem for GSF (see \cite{Gio10e,GK13b,Koc}). Essentially, this
theorem shows the existence and uniqueness of another GSF serving
as incremental ratio. This is the first of a long list of results
demonstrating the close similarities between ordinary smooth functions
and GSF.

We recall that the \emph{thickening} of an open set $\Omega\subseteq\R^{n}$
\emph{along} $v\in\R^{n}$ is $\thick_{v}(\Omega):=\{(x,h)\in\R^{n+1}\mid[x,x+hv]_{\R^{n}}\subseteq\Omega\}$,
and serves as a natural domain of a partial incremental ratio along
$v$ of any function defined on $\Omega$. In order to prove the Fermat-Reyes
theorem, it is simpler to define what \emph{a }thickening of $U\subseteq\rcrho^{n}$
along $v\in\rcrho^{n}$ is. 
\begin{defn}
\label{def:thickening}Let $U\subseteq\rcrho^{n}$ and let $v\in\rcrho^{n}$,
then we say that $T\subseteq\rcrho^{n+1}$ is \emph{a (sharp) thickening}
\emph{of }$U$\emph{ along $v$} if 
\begin{enumerate}
\item \label{enu:thickNonEmpty}$\forall x\in U:\ (x,0)\in T$ 
\item \label{enu:thickProductOfBalls}For all $(x,h)\in T$ there exist
$a$, $b\in\rcrho_{>0}$, with $b<a$, such that: 
\begin{enumerate}[label=(\alph*)]
\item $|h\cdot v|<b$ 
\item $B_{a}(x)\subseteq U$ 
\item \label{enu:prodBallsInclThick}$B_{a}(x)\times B_{b}(0)\subseteq T$. 
\end{enumerate}
\end{enumerate}
\noindent Finally, we will say that $T$ is \emph{a} \emph{large thickening
of $U$along $v$} if the radii $a$, $b$ in \ref{enu:thickProductOfBalls}
are real: \emph{$a$, $b\in\R_{>0}$.} 
\end{defn}

\begin{rem}
~ 
\begin{enumerate}
\item Conditions \ref{enu:thickNonEmpty} and \ref{enu:thickProductOfBalls}
imply that necessarily $U$ is a sharply open set, whereas $U$ is
a large open set if $T$ is a large thickening. 
\item \label{enu:relWithClassThickening}Let $(x,h)\in T$ and let the radii
$a$, $b$ be as in \ref{enu:thickProductOfBalls}, then for all $s\in[0,1]$
we have $|x+shv-x|\le|hv|<b<a$. Therefore $[x,x+hv]\subseteq B_{a}(x)\subseteq U$.
This gives a connection with the classical definition of thickening
and shows that if $f:U\ra\rcrho$, we can consider the difference
$f(x+hv)-f(x)$. 
\item Condition \ref{enu:thickProductOfBalls} of Def.~\ref{def:thickening}
yields that $T$ is a sharply open subset of $\rcrho{}^{n+1}$; it
is a large open subset in case $T$ is a large thickening. 
\item If $T$ and $\bar{T}$ are two (large) thickenings of $U$ along $v$,
then also $T\cap\bar{T}$ is a (large) thickening of the same type.
Finally, thickenings are also closed with respect to arbitrary non
empty unions. 
\end{enumerate}
\end{rem}

\noindent In the present setting, the Fermat-Reyes theorem is the
following. 
\begin{thm}
\noindent \label{thm:FR-forGSF} Let $U\subseteq\rcrho^{n}$ be a
sharply open set, let $v=[v_{\eps}]\in\rcrho^{n}$, and let $f\in{}^{\rho}\Gcinf(U,\rcrho)$
be a generalized smooth map generated by the net of smooth functions
$f_{\eps}\in\cinfty(\Omega_{\eps},\R)$. Then 
\begin{enumerate}
\item \label{enu:existenceRatio}If $S$ is a thickening of U along $v$
such that $S\subseteq\sint{\thick_{v_{\eps}}(\Omega_{\eps})}$, then
there exists a thickening $T\subseteq S$ of $U$ along $v$ and a
generalized smooth map $r\in{}^{\rho}\Gcinf(T,\rcrho)$, called the
\emph{generalized incremental ratio} of $f$ \emph{along} $v$, such
that 
\[
f(x+hv)=f(x)+h\cdot r(x,h)\qquad\forall(x,h)\in T.
\]
Moreover $r(x,0)=\left[\frac{\partial f_{\eps}}{\partial v_{\eps}}(x_{\eps})\right]$
for every $x\in U$, and we can thus define $\frac{\partial f}{\partial v}(x):=r(x,0)$,
so that $\frac{\partial f}{\partial v}\in{}^{\rho}\Gcinf(U,\rcrho)$. 
\item \label{enu:uniquenessRatio}Any two generalized incremental ratios
of $f$ coincide on the intersection of their domains. 
\end{enumerate}
\noindent If $U$ is a large open set and $S$ is a large thickening
of $U$ along $v$, then an analogous statement holds for a large
thickening $T$ of $U$ along $v$. 
\end{thm}

Note that this result allows us to consider the partial derivative
of $f$ with respect to an arbitrary generalized vector $v\in\rcrho^{n}$
which can be, e.g., near-standard or infinite.

\noindent Before proving the theorem, it is essential to show that
GSF are uniquely determined by invertible elements. 
\begin{lem}
\label{lem:invSufficesForGSF} Let $U\subseteq\rcrho^{n}$ be an open
set in the sharp topology, and let $f\in{}^{\rho}\Gcinf(U,\RC{\rho}^{d})$
be a GSF. Then $f(x)=0$ for every $x\in U$ if and only if $f(x)=0$
for all $x\in U$ such that $|x|$ is invertible. 
\end{lem}

\begin{proof}
Using Lem.~\ref{lem:mayer}, it is straightforward to prove that
the group of invertible elements is dense in $\rcrho$ with respect
to the sharp topology. This implies that the set of points in $U$
whose every coordinate is invertible is dense in $U$. Clearly for
any such point $y$, $|y|$ is invertible. Thus given any point $x\in U$
there exists a sequence $(x_{k})$ in $U$ converging to $x$ in the
sharp topology and such that $|x_{k}|$ is invertible for each $k$.
Since $f$ is continuous with respect to the sharp topology (Thm.~\ref{thm:GSF-continuity}
\ref{enu:GSF-cont}), this yields $0=f(x_{k})\to f(x)$. 
\end{proof}
\noindent To show the existence of thickenings, we also need the following
result 
\begin{lem}
\label{lem:propOfThickenings}Let $(\Omega_{\eps})$ be a net of open
sets of $\R^{n}$, and let $v=[v_{\eps}]\in\rcrho^{n}$. Then 
\begin{enumerate}
\item \label{enu:trivialPair}if $x\in\sint{\Omega_{\eps}}$ then $(x,0)\in\sint{\thick_{v_{\eps}}(\Omega_{\eps})}$. 
\item \label{enu:sintThickSubsetThickSint}If $(x,h)\in\sint{\thick_{v_{\eps}}(\Omega_{\eps})}$
then $x+thv\in\sint{\Omega_{\eps}}$ for all $t\in[0,1]$. 
\item \label{enu:existenceThickening}If $U\subseteq\sint{\Omega_{\eps}}$
is sharply open, there exists a sharp thickening $T$ of $U$ along
$v$ such that $T\subseteq\sint{\thick_{v_{\eps}}(\Omega_{\eps})}$. 
\end{enumerate}
\noindent The same properties hold if we consider the strongly internal
sets $\sintF{\Omega_{\eps}}$ and $\sintF{\thick_{v_{\eps}}(\Omega_{\eps})}$
in the Fermat topology. In this case in \ref{enu:existenceThickening},
$U$ has to be supposed large open and the resulting thickening is
large as well. 
\end{lem}

\begin{proof}
If $x\in\sint{\Omega_{\eps}}$, then $x_{\eps}\in\Omega_{\eps}$ for
$\eps$ small, and we also have $(x_{\eps},0)\in\thick_{v_{\eps}}(\Omega_{\eps})$
for the same $\eps$. Now take $(x'_{\eps},z_{\eps})\sim_{\rho}(x_{\eps},0)$,
so that $x=[x'_{\eps}]\in\sint{\Omega_{\eps}}$ and hence $B_{r}(x)\subseteq\sint{\Omega_{\eps}}$
for some $r\in\rcrho_{>0}$ such that $r_{\eps}<d(x'_{\eps},\Omega_{\eps}^{c})$
for all $\eps\in I$ (see Thm.~\ref{thm:strongMembershipAndDistanceComplement}).
The net $(z_{\eps})\sim_{\rho}0$, so we also have $|z_{\eps}v_{\eps}|<r_{\eps}$
for $\eps$ small. Thus for $\eps$ sufficiently small we obtain both
$x'_{\eps}\in\Omega_{\eps}$ and $|z_{\eps}v_{\eps}|<r_{\eps}$, so
that for all $s\in[0,1]_{\R}$ we have that $|x'_{\eps}+sz_{\eps}v_{\eps}-x'_{\eps}|\le|z_{\eps}v_{\eps}|<r_{\eps}<d(x'_{\eps},\Omega_{\eps}^{c})$.
Hence $x'_{\eps}+sz_{\eps}v_{\eps}\in\Omega_{\eps}$, i.e.~$(x'_{\eps},z_{\eps})\in\thick_{v_{\eps}}(\Omega_{\eps})$
for $\eps$ sufficiently small. This shows that $(x,0)\in\sint{\thick_{v_{\eps}}(\Omega_{\eps})}$,
implying \ref{enu:trivialPair}.\\
 To prove \ref{enu:sintThickSubsetThickSint}, assume that $(x,h)\in\sint{\thick_{v_{\eps}}(\Omega_{\eps})}$
and $t\in[0,1]$. Therefore, $0\le t_{\eps}\le1$ for $\eps$ small
and some representative $(t_{\eps})$ of $t$. Since $(x_{\eps},h_{\eps})\in_{\eps}\thick_{v_{\eps}}(\Omega{}_{\eps})$,
we have that $x_{\eps}+t_{\eps}h_{\eps}v_{\eps}\in\Omega_{\eps}$
for $\eps$ small. If we take another representative $(y_{\eps})\sim_{\rho}(x_{\eps}+t_{\eps}h_{\eps}v_{\eps})$,
then we can define $x'_{\eps}:=y_{\eps}-t_{\eps}h_{\eps}v_{\eps}$
so that $(x_{\eps},h_{\eps})\sim_{\rho}(x'_{\eps},h_{\eps})$. From
$(x_{\eps},h_{\eps})\in_{\eps}\thick_{v_{\eps}}(\Omega_{\eps})$ we
thus get that also $(x'_{\eps},h_{\eps})\in\thick_{v_{\eps}}(\Omega_{\eps})$
for $\eps$ small. Therefore $x'_{\eps}+t_{\eps}h_{\eps}v_{\eps}=y_{\eps}\in\Omega_{\eps}$
for $\eps$ small. This shows that $x+thv\in\sint{\Omega_{\eps}}$.\\
 Finally, in order to prove \ref{enu:existenceThickening}, we assume
that $U\subseteq\sint{\Omega_{\eps}}$ is a sharply open subset. For
all $x\in U\subseteq\sint{\Omega_{\eps}}$, we have $(x,0)\in\sint{\thick_{v_{\eps}}(\Omega_{\eps})}$
from \ref{enu:trivialPair}, and hence Thm.~\ref{thm:strongMembershipAndDistanceComplement}
\ref{enu:stronglyIntAreOpen} yields the existence of $c_{x}\in\rcrho_{>0}$
such that $B_{c_{x}}(x,0)\subseteq\sint{\thick_{v_{\eps}}(\Omega_{\eps})}$.
Since $U$ is a neighborhood of $x$, there exists $a_{x}\in\rcrho_{>0}$,
$a_{x}<c_{x}$, such that $B_{a_{x}}(x)\subseteq U$. Choose $b_{x}\in\rcrho_{>0}$
such that $b_{x}<a_{x}$ and $a_{x}+b_{x}<c_{x}$. Because $v\in\rcrho^{n}$
is $\rho$-moderate, we have $|v|<\diff{\rho}^{-N}$ for some $N\in\N$.
Take $d_{x}\in\rcrho_{>0}$ such that $d_{x}<b_{x}\cdot\diff{\rho}^{N}$
and define 
\[
T:=\bigcup_{x\in U}B_{a_{x}}(x)\times B_{d_{x}}(0).
\]
If $|h|<d_{x}$, then $|h\cdot v|<|v|\cdot d_{x}<b_{x}$ and hence
$T$ is a sharp thickening of $U$ along $v$. We finally note that
$(x',h)\in B_{a_{x}}(x)\times B_{d_{x}}(0)$ implies $|(x',h)-(x,0)|\le|x'-x|+|h|<a_{x}+d_{x}<a_{x}+b_{x}<c_{x}$,
so that $(x',h)\in B_{c_{x}}(x,0)\subseteq\sint{\thick_{v_{\eps}}(\Omega_{\eps})}$.
Therefore $T\subseteq\sint{\thick_{v_{\eps}}(\Omega_{\eps})}$.

Considering $\sim_{\text{F}}$ instead of $\sim_{\rho}$ and radii
in $\R_{>0}$, in the same way we can prove the analogous properties
for strongly internal sets in the Fermat topology. 
\end{proof}
We can now prove the Fermat-Reyes theorem for GSF. 
\begin{proof}[\emph{Proof of Theorem \ref{thm:FR-forGSF}}]
Since $U$ is sharply open, for any point $x\in U$ we can find a
ball $B_{R_{x}}(x)\subseteq U$, $R_{x}\in\rcrho_{>0}$. Define $a_{x}:=\frac{R_{x}}{2}$
and $b_{x}:=\frac{a_{x}}{2}$. Because $v\in\rcrho^{n}$ is $\rho$-moderate,
we have $|v|<\diff{\rho}^{-N}$ for some $N\in\N$. Take $d_{x}\in\rcrho_{>0}$
such that $d_{x}<b_{x}\cdot\diff{\rho}^{N}$ and set $T:=\bigcup_{x\in U}B_{a_{x}}(x)\times B_{d_{x}}(0)$.
Since for all $x\in U$ the pair $(x,0)$ is an interior point of
the given thickening $S$, we can assume to have chosen $a_{x}$ and
$d_{x}$ so that $T\subseteq S\subseteq\sint{\thick_{v_{\eps}}(\Omega_{\eps})}$.
In case $U$ is large open, we can proceed as above to obtain $a_{x}$,
$d_{x}\in\R_{>0}$, so that $T$ would then be a large thickening.\\
 Let us consider the net of smooth function $r_{\eps}\in\cinfty(\thick_{v_{\eps}}(\Omega_{\eps}))$
defined by $r_{\eps}(y,h):=\int_{0}^{1}\frac{\partial f_{\eps}}{\partial v_{\eps}}(y+thv_{\eps})\diff{t}$
for all $\eps\in I$. We calculate the partial derivative $\partial^{\alpha}r_{\eps}(y_{\eps},h_{\eps})$
for $\alpha\in\N^{n+1}$ and an arbitrary point $(y,h)\in T$. For
simplicity, set $\hat{\alpha}:=(\alpha_{1},\ldots,\alpha_{n})$, and
$v_{\eps}=:(v_{1\eps},\ldots,v_{n\eps})\in\R^{n}$. 
\begin{equation}
\partial^{\alpha}r_{\eps}(y_{\eps},h_{\eps})=\int_{0}^{1}\frac{\partial^{|\alpha|}}{\partial h^{\alpha_{n+1}}\partial y^{\hat{\alpha}}}\left[\frac{\partial f{}_{\eps}}{\partial v_{\eps}}(y_{\eps}+th_{\eps}v_{\eps})\right]\diff{t}\label{eq:partialDerOf-r-eps}
\end{equation}
Applying the chain rule and the mean value theorem for integrals,
\eqref{eq:partialDerOf-r-eps} can be written as a sum of terms of
the form $\partial^{\beta}f(y_{\eps}+t_{\eps}h_{\eps}v_{\eps})v_{\eps}^{\gamma}t_{\eps}^{m}$,
for suitable multi-indices $\beta,\gamma$, and $m\in\N$. Here, $t_{\eps}\in[0,1]_{\R}$
for all $\eps\in I$. From $(y,h)\in T$, we get $y\in B_{a_{x}}(x)$
and $h\in B_{d_{x}}(0)$ for some $x\in U$. This gives $|y+thv-x|\le|y-x|+|hv|<a_{x}+b_{x}<R_{x}$,
so that $y+thv\in B_{R_{x}}(x)\subseteq U$. From Def.~\ref{def:netDefMap}
\ref{enu:partial-u-moderate} we hence have that $\left(\partial^{\beta}f_{\eps}(y_{\eps}+t_{\eps}h_{\eps}v_{\eps})\right)$
is $\rho$-moderate. Since moderateness is preserved by polynomials,
and $t_{\eps}\in[0,1]_{\R}$ is moderate, from \eqref{eq:partialDerOf-r-eps}
we obtain that $\left(\partial^{\alpha}r_{\eps}(y_{\eps},h_{\eps})\right)$
is moderate. This proves that $r:=[r_{\eps}(-,-)]|_{T}:T\longrightarrow\rcrho$
is a GSF.\\
 We have 
\begin{align*}
h\cdot r(x,h) & =\left[h_{\eps}\cdot\int_{0}^{1}\frac{\partial f_{\eps}}{\partial v_{\eps}}(x_{\eps}+th_{\eps}v_{\eps})\,\diff{t}\right]\\
 & =\left[\int_{0}^{h_{\eps}}\frac{d}{ds}\left\{ f_{\eps}(x_{\eps}+sv_{\eps})\right\} (s)\,\diff{s}\right]\\
 & =[f_{\eps}(x_{\eps}+h_{\eps}v_{\eps})]-[f_{\eps}(x_{\eps})]=f(x+hv)-f(x).
\end{align*}
Of course $r(x,0)=\left[\frac{\partial f_{\eps}}{\partial v_{\eps}}(x_{\eps})\right]$,
and this concludes the existence part.\\
 To prove uniqueness, consider $(x,h)\in T\cap\bar{T}$, where $T$
and $\bar{T}$ are two thickenings (along $v$) of the incremental
ratios $r$, $\bar{r}$. Define $R(k):=r(x,k)$ and $\bar{R}(k):=\bar{r}(x,k)$
for $k\in B_{b}(0)$, where $(x,h)\in B_{a}(x)\times B_{b}(0)\subseteq T\cap\bar{T}$
by the definition of thickening. Since $k\in B_{b}(0)\mapsto(x,k)\in T\cap\bar{T}$
is a GSF, both $R$ and $\bar{R}$ are still generalized smooth maps
by the closure with respect to composition. Moreover 
\begin{equation}
k\cdot R(k)=k\cdot r(x,k)=f(x+kv)-f(x),\label{eq:FR-propertyByRepres}
\end{equation}
and analogously $k\cdot\bar{R}(k)=f(x+kv)-f(x)=k\cdot R(k)$. Therefore
$R(k)=\bar{R}(k)$ for every $k\in B_{b}(0)$ which is invertible,
and Lemma \ref{lem:invSufficesForGSF} yields $R=\bar{R}$. Since
$h\in B_{b}(0)$ we get $R(h)=r(x,h)=\bar{R}(h)=\bar{r}(x,h)$. 
\end{proof}
We will use the notation $\frac{\partial f}{\partial v}[-,-]_{T}\in{}^{\rho}\Gcinf(T,\rcrho)$
(or simply $\frac{\partial f}{\partial v}[-,-]$ in case the domain
is clear from the context) for the generalized smooth incremental
ratio of a function $f\in{}^{\rho}\Gcinf(U,\rcrho)$ defined on the
thickening $T$, so as to distinguish it from the derivative $\frac{\partial f}{\partial v}\in{}^{\rho}\Gcinf(U,\rcrho)$.
Since any partial derivative of a GSF is still a GSF, higher order
derivatives $\frac{\partial^{\alpha}f}{\partial v^{\alpha}}\in{}^{\rho}\Gcinf(U,\rcrho)$
are simply defined recursively.

As follows from Thm.\ \ref{thm:FR-forGSF} \ref{enu:existenceRatio}
and Thm\@.~\ref{thm:embeddingD'} \ref{enu:commute_der}, the concept
of derivative defined using the Fermat-Reyes theorem is compatible
with the classical derivative of Schwartz distributions via the embeddings
$\iota^{b}$ from Thm.\ \ref{thm:embeddingD'}. The following result
follows from the analogous properties for the nets of smooth functions
defining $f$ and $g$. 
\begin{thm}
\label{thm:linearityLeibnizDer} Let $U\subseteq\rcrho^{n}$ be an
open subset in the sharp topology, let $v\in\rcrho^{n}$ and $f$,
$g:U\longrightarrow\rcrho$ be generalized smooth maps. Then 
\begin{enumerate}
\item $\frac{\partial(f+g)}{\partial v}=\frac{\partial f}{\partial v}+\frac{\partial g}{\partial v}$ 
\item $\frac{\partial(r\cdot f)}{\partial v}=r\cdot\frac{\partial f}{\partial v}\quad\forall r\in\rcrho$ 
\item $\frac{\partial(f\cdot g)}{\partial v}=\frac{\partial f}{\partial v}\cdot g+f\cdot\frac{\partial g}{\partial v}$ 
\item For each $x\in U$, the map $\diff{f}(x).v:=\frac{\partial f}{\partial v}(x)\in\rcrho$
is $\rcrho$-linear in $v\in\rcrho^{n}$. 
\end{enumerate}
\end{thm}

Using the Fermat-Reyes theorem, it is also possible to give intrinsic
proofs (i.e.~without using nets of smooth functions that define a
given GSF), as exemplified in the following 
\begin{thm}
\label{thm:chainRule} Let $U\subseteq\rcrho^{n}$ and $V\subseteq\rcrho^{d}$
be open subsets in the sharp topology and $g\in{}^{\rho}\Gcinf(V,U)$,
$f\in{}^{\rho}\Gcinf(U,\rcrho)$ be generalized smooth maps. Then
for all $x\in V$ and all $v\in\rcrho^{d}$ 
\begin{align*}
\frac{\partial\left(f\circ g\right)}{\partial v}(x) & =\diff{f}\left(g(x)\right).\frac{\partial g}{\partial v}(x)\\
\,\diff{\!\left(f\circ g\right)}(x) & =\diff{f}\left(g(x)\right)\circ\diff{g}(x).
\end{align*}
\end{thm}

\begin{proof}
For $h$ small (in the sharp topology), we can write 
\begin{equation}
f\left[g(x+hv)\right]=f\left[g(x)+h\frac{\partial g}{\partial v}[x,h]\right].\label{eq:chain1}
\end{equation}
Set $u(x,h):=\frac{\partial g}{\partial v}[x,h]\in\rcrho^{n}$. Then
\eqref{eq:chain1} yields 
\[
f\left[g(x+hv)\right]=f(g(x))+h\cdot\frac{\partial f}{\partial u(x,h)}\left[g(x),h\right].
\]
Therefore, the uniqueness of the smooth incremental ratio of $f\circ g$
in the direction $v$ implies 
\[
\frac{\partial\left(f\circ g\right)}{\partial v}[x,h]=\frac{\partial f}{\partial u(x,h)}\left[g(x),h\right].
\]
For $h=0$, we get 
\[
\frac{\partial\left(f\circ g\right)}{\partial v}(x)=\frac{\partial f}{\partial u(x,0)}\left(g(x)\right)=\diff{f}(g(x)).u(x,0)=\diff{f}(g(x)).\frac{\partial g}{\partial v}(x),
\]
which is our conclusion. 
\end{proof}

\section{\label{sec:Integral-calculus}Integral calculus using primitives}

In this section, we inquire existence and uniqueness of primitives
$F$ of a GSF $f\in\gsf([a,b],\rcrho)$ (see also \cite{YeLi16} for
an analogous approach). To this end, we shall have to introduce the
derivative $F'(x)$ at boundary points $x\in[a,b]$, i.e.~such that
$x-a$ or $b-x$ is not invertible. Let us note explicitly, in fact,
that the Fermat-Reyes Theorem \ref{thm:FR-forGSF} is stated only
for sharply open domains. We shal therefore require the following
result. 
\begin{lem}
\label{lem:approxOfBoundaryPointsWithInterior}Let $a$, $b\in\rcrho$
be such that $a<b$, then the interior $\text{int}\left([a,b]\right)$
in the sharp topology is dense in $[a,b]$. 
\end{lem}

\begin{proof}
Take representatives of $a$, $b$ and $x\in[a,b]$ such that $a_{\eps}<b_{\eps}$
and $a_{\eps}\le x_{\eps}\le b_{\eps}$ for $\eps$ small. Thm.~\ref{thm:strongMembershipAndDistanceComplement}
\ref{enu:stronglyIntSetsDistance} yields $\text{int}\left([a,b]\right)=\sint{(a_{\eps},b_{\eps})}$.
To prove the conclusion, it suffices to define 
\[
y_{k\eps}:=\begin{cases}
x_{\eps} & \text{if }a_{\eps}+\rho_{\eps}^{k}\le x_{\eps}\le b_{\eps}-\rho_{\eps}^{k}\\
a_{\eps}+\rho_{\eps}^{k} & \text{if }x_{\eps}<a_{\eps}+\rho_{\eps}^{k}\\
b_{\eps}-\rho_{\eps}^{k} & \text{if }x_{\eps}>b_{\eps}-\rho_{\eps}^{k}
\end{cases}
\]
for any $k\in\N$ and $\eps\in I$. We have $d(y_{k\eps},(a_{\eps},b_{\eps})^{c})\ge\rho_{\eps}^{k}$,
so that $y_{k}\in\sint{(a_{\eps},b_{\eps})}$. Moreover, $|y_{k\eps}-x_{\eps}|<\rho_{\eps}^{k}$
for all $\eps$, and from this the desired limit condition follows. 
\end{proof}
The following result shows that every GSF can have at most one primitive
GSF up to an additive constant. 
\begin{thm}
\label{thm:uniquenessOfPrimitives}Let $X\subseteq\RC{\rho}$ and
let $f\in{}^{\rho}\Gcinf(X,\rcrho)$ be a generalized smooth function.
Let $a$, $b\in\rcrho$, with $a<b$, such that $(a,b)\subseteq X$.
If $f'(x)=0$ for all $x\in\text{\emph{int}}(a,b)$, then $f$ is
constant on $(a,b)$. An analogous statement holds if we take any
other type of interval (closed or half closed) instead of $(a,b)$. 
\end{thm}

\begin{proof}
By Lemma \ref{lem:fromOmega_epsToRn}, we can assume that $f$ is
defined by a net of smooth functions $f_{\eps}\in\cinfty(\R,\R)$.
From the Fermat-Reyes Theorem \ref{thm:FR-forGSF}, we know that $f'(x)=[f'_{\eps}(x_{\eps})]$
for every interior point $x=[x_{\eps}]\in X$. For all $x$, $y\in\text{int}(a,b)\subseteq U$,
we can write 
\begin{align}
f(x)-f(y) & =[f_{\eps}(x_{\eps})-f_{\eps}(y_{\eps})]=\left[(y_{\eps}-x_{\eps})\cdot\int_{0}^{1}f'_{\eps}(x_{\eps}+s(y_{\eps}-x_{\eps}))\,\diff{s}\right]\nonumber \\
 & =(y-x)\cdot[f'_{\eps}(x_{\eps}+s_{\eps}(y_{\eps}-x_{\eps}))]=(y-x)\cdot f'(x+s(y-x)),\label{eq:fx-fy}
\end{align}
where $s_{\eps}\in[0,1]_{\R}$ is provided by the integral mean value
theorem and $s:=[s_{\eps}]\in[0,1]$. Since $x$, $y\in\text{int}(a,b)$,
we have $x+s(y-x)\in\text{int}(a,b)$ and hence $f'(x+s(y-x))=0$.
Thereby, \eqref{eq:fx-fy} yields $f(x)=f(y)$ as claimed. For a different
type of interval, it suffices to consider Lemma \ref{lem:approxOfBoundaryPointsWithInterior}
and sharp continuity of GSF (Thm.~\ref{thm:GSF-continuity}). 
\end{proof}
\begin{rem}
\label{rem:I-function}From the Fermat-Reyes Thm.~\ref{thm:FR-forGSF}
and from Thm.~\ref{thm:uniquenessOfPrimitives}, it follows that
the function $i(x):=1$ if $x\approx0$ and $i(x):=0$ otherwise cannot
be a GSF on any large neighborhood of $x=0$. This example stems from
the property that different standard real numbers can always be separated
by infinitesimal balls. 
\end{rem}

At interior points $x\in[a,b]$ in the sharp topology, the definition
of derivative $f^{(k)}(x)$ follows from the Fermat-Reyes Theorem
\ref{thm:FR-forGSF}. At boundary points, we have the following 
\begin{thm}
\label{thm:existenceOfDerivativesAtBorderPoints}Let $a$, $b\in\rcrho$
with $a<b$, and $f\in{}^{\rho}\Gcinf([a,b],\rcrho)$ be a generalized
smooth function. Then for all $x\in[a,b]$, the following limit exists
in the sharp topology 
\[
\lim_{\substack{y\to x\\
y\in\text{\emph{int}}\left([a,b]\right)
}
}f^{(k)}(y)=:f^{(k)}(x).
\]
Moreover, if the net $f_{\eps}\in\cinfty(\Omega_{\eps},\R)$ defines
$f$ and $x=[x_{\eps}]$, then $f^{(k)}(x)=[f_{\eps}^{(k)}(x_{\eps})]$
and hence $f^{(k)}\in{}^{\rho}\Gcinf([a,b],\rcrho)$. 
\end{thm}

\begin{proof}
We have 
\[
{\displaystyle \lim_{\substack{y\to x\\
y\in\text{int}\left([a,b]\right)
}
}}f^{(k)}(y)={\displaystyle \lim_{\substack{y\to x\\
y\in\text{int}\left([a,b]\right)
}
}}\left[f_{\eps}^{(k)}(y_{\eps})\right]=[f{}_{\eps}^{(k)}(x_{\eps})],
\]
where the last equality follows due to the sharp continuity of $[f_{\eps}^{(k)}(-)]$
at every point $x\in[a,b]\subseteq\sint{\Omega_{\eps}}$ (Thm.~\ref{thm:GSF-continuity}
\ref{enu:GSF-cont} and Lem.~\ref{lem:approxOfBoundaryPointsWithInterior}). 
\end{proof}
We can now prove existence and uniqueness of primitives of GSF: 
\begin{thm}
\label{thm:existenceUniquenessPrimitives}Let $a$, $b$, $c\in\rcrho$,
with $a<b$ and $c\in[a,b]$. Let $f\in{}^{\rho}\Gcinf([a,b],\rcrho)$
be a generalized smooth function. Then, there exists one and only
one generalized smooth function $F\in{}^{\rho}\Gcinf([a,b],\rcrho)$
such that $F(c)=0$ and $F'(x)=f(x)$ for all $x\in[a,b]$. Moreover,
if $f$ is defined by the net $f_{\eps}\in\Coo(\R,\R)$ and $c=[c_{\eps}]$,
then $F(x)=\left[\int_{c_{\eps}}^{x_{\eps}}f_{\eps}(s)\,\diff{s}\right]$
for all $x=[x_{\eps}]\in[a,b]$. 
\end{thm}

\begin{proof}
Fix representatives $(a_{\eps})$, $(b_{\eps})$ and $(c_{\eps})$
of $a$, $b$, $c$ such that 
\begin{equation}
a_{\eps}\le c_{\eps}\le b_{\eps}\label{eq:acb}
\end{equation}
for $\eps$ small. By Lemma \ref{lem:fromOmega_epsToRn}, we can assume
that $f$ is generated by a net $f_{\eps}\in\cinfty(\R,\R)$. Set
\begin{equation}
F_{\eps}(x):=\int_{c_{\eps}}^{x}f_{\eps}(s)\,\diff{s}\quad\forall x\in\R.\label{eq:primitive}
\end{equation}
We want to prove that the net $(F_{\eps})$ defines a GSF of type
$[a,b]\longrightarrow\rcrho$, and therefore we take $x\in[a,b]$
and $\alpha\in\N$. Choose a representative $(x_{\eps})$ of $x$
such that 
\begin{equation}
a_{\eps}\le x_{\eps}\le b_{\eps}\label{eq:a_x_b}
\end{equation}
for $\eps$ small. If $\alpha>0$, then $F_{\eps}^{(\alpha)}(x_{\eps})=f_{\eps}^{(\alpha-1)}(x_{\eps})$
and hence moderateness is clear since $x\in[a,b]$. For $\alpha=0$
we have $F_{\eps}(x_{\eps})=f_{\eps}(\sigma_{\eps})\cdot(x_{\eps}-c_{\eps})$,
where 
\begin{equation}
\sigma_{\eps}\in[c_{\eps},x_{\eps}]\cup[x_{\eps},c_{\eps}]\quad\forall\eps\in I\label{eq:sigma_c_x}
\end{equation}
is obtained by the integral mean value theorem. For $\eps$ small,
we have both \eqref{eq:acb} and \eqref{eq:a_x_b}, so that these
inequalities and \eqref{eq:sigma_c_x} yield $\sigma\in[a,b]\subseteq U$.
Therefore $(f_{\eps}(\sigma_{\eps}))$ and $(F_{\eps}(x_{\eps}))$
are moderate. This proves condition Def.~\ref{def:netDefMap} \ref{enu:partial-u-moderate}
for the net $(F_{\eps})$, and we can hence set $F(x):=[F_{\eps}(x_{\eps})]\in\rcrho$
for all $x=[x_{\eps}]\in[a,b]$.\\
 If $y\in\text{int}([a,b])$, we can apply our differential calculus
to the generalized smooth map $F|_{\text{int}([a,b])}=[F_{\eps}(-)]|_{\text{int}([a,b])}$,
obtaining $F'(y)=[f_{\eps}(y_{\eps})]=f(y)$. From this, if $x\in[a,b]$,
we get 
\[
F'(x)=\lim_{\substack{y\to x\\
y\in\text{int}([a,b])
}
}F'(y)=\lim_{\substack{y\to x\\
y\in\text{int}([a,b])
}
}f(y)=f(x)
\]
because $f$ is sharply continuous at $x\in[a,b]\subseteq U$. The
uniqueness part follows from Theorem \ref{thm:uniquenessOfPrimitives}. 
\end{proof}
\begin{defn}
\label{def:integral}Under the assumptions of Theorem \ref{thm:existenceUniquenessPrimitives},
we denote by $\int_{c}^{(-)}f:=\int_{c}^{(-)}f(s)\,\diff{s}\in{}^{\rho}\Gcinf([a,b],\rcrho)$
the unique generalized smooth function such that: 
\begin{enumerate}
\item $\int_{c}^{c}f=0$ 
\item $\left(\int_{c}^{(-)}f\right)'(x)=\frac{\diff{}}{\diff{x}}\int_{c}^{x}f(s)\,\diff{s}=f(x)$
for all $x\in[a,b]$. 
\end{enumerate}
\end{defn}

In Sec.~\ref{sec:n-dimIntegral}, we develop a generalization of
this concept of integration to GSF in several variables and to more
general domains of integration $M\subseteq\rcrho^{d}$. 
\begin{example}
~ 
\begin{enumerate}
\item Since $\rcrho$ contains both infinitesimal and infinite numbers,
our notion of definite integral also includes ``improper integrals''.
Let e.g.~$f(x)=\frac{1}{x}$ for $x\in\rcrho_{>0}$ and $a=1$, $b=\diff{\rho}^{-q}$,
$q>0$. Then 
\begin{equation}
\int_{a}^{b}f(s)\,\diff{s}=\left[\int_{1}^{\rho_{\eps}^{-q}}\frac{1}{s}\,\diff{s}\right]=[\log\rho_{\eps}^{-q}]-\log1=-q\log\diff{\rho},\label{eq:imprInt}
\end{equation}
which is, of course, a positive infinite generalized number. This
apparently trivial result is closely tied to the possibility to define
GSF on arbitrary domains, like $F\in{}^{\rho}\Gcinf([a,b],\rcrho)$
in Thm.~\ref{thm:existenceUniquenessPrimitives} where $b$ is an
infinite number as in \eqref{eq:imprInt}, which is one of the key
properties allowing one to get the closure with respect to composition. 
\item If $p$, $q\in\rti$, $p<0<q$ and both $p$ and $q$ are not infinitesimal,
then $\int_{p}^{q}\delta(t)\,\diff{t}\approx1$. 
\end{enumerate}
\end{example}

\begin{thm}
\label{thm:intRules}Let $f\in{}^{\rho}\Gcinf(X,\rcrho)$ and $g\in{}^{\rho}\Gcinf(Y,\rcrho)$
be generalized smooth functions defined on arbitrary domains in $\rcrho$.
Let $a$, $b\in\rcrho$ with $a<b$ and $[a,b]\subseteq X\cap Y$,
then 
\begin{enumerate}
\item \label{enu:additivityFunction}$\int_{a}^{b}\left(f+g\right)=\int_{a}^{b}f+\int_{a}^{b}g$ 
\item \label{enu:homog}$\int_{a}^{b}\lambda f=\lambda\int_{a}^{b}f\quad\forall\lambda\in\rcrho$ 
\item \label{enu:additivityDomain}$\int_{a}^{b}f=\int_{a}^{c}f+\int_{c}^{b}f$
for all $c\in[a,b]$ 
\item \label{enu:chageOfExtremes}$\int_{a}^{b}f=-\int_{b}^{a}f$ 
\item \label{enu:foundamental}$\int_{a}^{b}f'=f(b)-f(a)$ 
\item \label{enu:intByParts}$\int_{a}^{b}f'\cdot g=\left[f\cdot g\right]_{a}^{b}-\int_{a}^{b}f\cdot g'$ 
\item \label{enu:intMonotone}If $f(x)\le g(x)$ for all $x\in[a,b]$, then
$\int_{a}^{b}f\le\int_{a}^{b}g$. 
\end{enumerate}
\end{thm}

\begin{proof}
This follows directly from \eqref{eq:primitive} and the usual rules
of the integral calculus, or from Def.~\ref{def:integral} and Thm.~\ref{thm:FR-forGSF}
for property \ref{enu:intMonotone}. 
\end{proof}
\begin{thm}
\label{thm:changeOfVariablesInt}Let $f\in{}^{\rho}\Gcinf(T,\rcrho)$
and $\phi\in{}^{\rho}\Gcinf(S,T)$ be generalized smooth functions
defined on arbitrary domains in $\rcrho$. Let $a$, $b\in\rcrho$,
with $a<b$, such that $[a,b]\subseteq S$, $\phi(a)<\phi(b)$ and
$[\phi(a),\phi(b)]\subseteq T$. Finally, assume that $\phi([a,b])\subseteq[\phi(a),\phi(b)]$.
Then 
\[
\int_{\phi(a)}^{\phi(b)}f(t)\,\diff{t}=\int_{a}^{b}f\left[\phi(s)\right]\cdot\phi'(s)\,\diff{s}.
\]
\end{thm}

\begin{proof}
Define\textit{\emph{ 
\begin{align*}
F(x) & :=\int_{\phi(a)}^{x}f(t)\,\diff{t}\quad\forall x\in[\phi(a),\phi(b)]\\
H(y) & :=\int_{\phi(a)}^{\phi(y)}f(t)\,\diff{t}\quad\forall y\in[a,b]\\
G(y) & :=\int_{a}^{y}f\left[\phi(s)\right]\cdot\phi'(s)\,\diff{s}\quad\forall y\in[a,b],
\end{align*}
Each one of these functions is generalized smooth by Def.~\ref{def:integral}
of the integral or by Thm.~\ref{thm:GSFcategory}, because it can
be written as a composition of generalized smooth maps. We have $H(a)=G(a)=0$,
$H(y)=F\left[\phi(y)\right]$ for every $y\in[a,b]$ and, by the chain
rule (Prop.~\ref{thm:chainRule}) $H'(y)=F'[\phi(y)]\cdot\phi'(y)=f\left[\phi(y)\right]\cdot\phi'(y)=G'(y)$,
the last two equalities following by Def.~\ref{def:integral} of
the integral. From the uniqueness Theorem \ref{thm:uniquenessOfPrimitives},
the conclusion $H=G$ follows.}} 
\end{proof}
\begin{rem}
(Relation to distributional primitives) Let $a,b\in\R$, $a<b$, and
set $\Omega=(a,b)\sse\R$. By \cite[Ch. II, \S 4]{Sch} there exists
a sequentially continuous operator $R:\D'(\Om)\to\D'(\Om)$ assigning
to any $T\in\D'(\Om)$ a primitive $R(T)$, i.e., $R(T)'=T$ in $\D'(\Om)$.
Now let $\iota_{\Omega}^{b}:\D'(\Omega)\to{}^{\rho}\Gcinf(\csp{\Omega},\RC{\rho})$
be an embedding as in Theorem \ref{thm:embeddingD'}, and fix any
$c\in\rcrho$ with $a\le c\le b$. Then 
\[
\iota_{\Omega}^{b}(R(T))'=\iota_{\Omega}^{b}(R(T)')=\iota(T)=\Big(\int_{c}^{(-)}\iota(T)\Big)'.
\]
Therefore, Theorem \ref{thm:intRules} \ref{enu:foundamental} implies
that 
\[
\int_{r}^{s}\iota(T)=\iota(R(T))(s)-\iota(R(T))(r)
\]
for all $s,t\in\rcrho$ with $a\le s,t\le b$. 
\end{rem}

\section{\label{sec:Some-classical-theorems}Some classical theorems for generalized
smooth functions}

It is natural to expect that several classical theorems of differential
and integral calculus can be extended from the ordinary smooth case
to the generalized smooth framework. Once again, we underscore that
these faithful generalizations are possible because we don't have
a priori limitations in the evaluation $f(x)$ for GSF. For example,
one does not have similar results in Colombeau theory, where an arbitrary
generalized function can be evaluated only at compactly supported
points.

\noindent We start from the intermediate value theorem. 
\begin{cor}
\label{cor:intermValue}Let $f\in{}^{\rho}\Gcinf(X,\rcrho)$ be a
generalized smooth function defined on the subset $X\subseteq\rcrho$.
Let $a$, $b\in\rcrho$, with $a<b$, such that $[a,b]\subseteq X$.
Assume that $f(a)<f(b)$. Then 
\[
\forall y\in\rcrho:\ f(a)\le y\le f(b)\ \Rightarrow\ \exists c\in[a,b]:\ y=f(c).
\]
\end{cor}

\begin{proof}
Let $f$ be defined by the net $f_{\eps}\in\cinfty(\R,\R)$. For small
$\eps$ and for suitable representatives $(a_{\eps})$, $(b_{\eps})$,
$(y_{\eps})$, we have 
\[
a_{\eps}<b_{\eps}\quad,\quad f_{\eps}(a_{\eps})\le y_{\eps}\le f_{\eps}(b_{\eps}).
\]
By the classical intermediate value theorem we get some $c_{\eps}\in[a_{\eps},b_{\eps}]$
such that $f_{\eps}(c_{\eps})=y_{\eps}$. Therefore $c:=[c_{\eps}]\in[a,b]\subseteq X$
and hence $f(c)=[f_{\eps}(c_{\eps})]=[y_{\eps}]=y$. 
\end{proof}
\noindent Using this theorem we can conclude that no GSF can assume
only a finite number of values which are comparable with respect to
the relation $<$ on any nontrivial interval $[a,b]\subseteq X$,
unless it is constant. For example, this provides an alternative way
of seeing that the function $i$ of Rem.~\ref{rem:I-function} cannot
be a generalized smooth map.

\noindent We note that the solution $c\in[a,b]$ of the previous generalized
smooth equation $y=f(x)$ need not even be continuous in $\eps$.
Indeed, let us consider the net of smooth functions depicted in Figure
\ref{fig:MikeC}, where it is understood that, as $\eps$ approaches
$0$, the two waves at the extremes oscillate around the dashed rectilinear
positions shown in the figure. Set $f(x)=\left[\int_{0}^{1}f_{\eps}(s)\,\diff{s}-f_{\eps}(x_{\eps})\right]\in\rcrho$
for $x\in[0,1]\subseteq\rcrho$, and analyze the generalized smooth
equation $f(x)=0$. Let $\eps_{k}=\frac{1}{k}$ be the ``times''
where the two waves of the net $(f_{\eps})$ are rectilinear. At these
times the solution $f(x_{\eps_{k}})=0$ can be any point $x_{\eps_{k}}\in[b,c]$.
Assume that for $\eps\in\left[\frac{1}{k},\frac{1}{k}+\delta_{k}\right]$
only the wave on the left is rectilinear and for $\eps\in\left[\frac{1}{k}-\delta_{k},\frac{1}{k}\right]$
only the wave on the right is rectilinear (where $\delta_{k}\downarrow0$
is sufficiently small). Therefore, in the first case, any solution
must be of the form $x_{\eps}\in[c,1]$ and in the second case $x_{\eps}\in[0,b]$.
Thus any solution must jump at every time $\eps_{k}$ and the height
of the jump must be at least $c-b$.

This example allows us to draw the following general conclusion: if
we consider generalized numbers as solutions of smooth equations,
then we are forced to work on a non-totally ordered ring of scalars
derived from discontinuous (in $\eps$) representatives. To put it
differently: if we choose a ring of scalars with a total order or
continuous representatives, we will not be able to solve every smooth
equation, and the given ring can be considered, in some sense, incomplete.
Of course, this does not mean that the study of better behaved (non-totally
ordered) subrings of $\rcrho$, useful for special purposes, is not
interesting.

\begin{figure}
\noindent \begin{centering}
\includegraphics[scale=0.65]{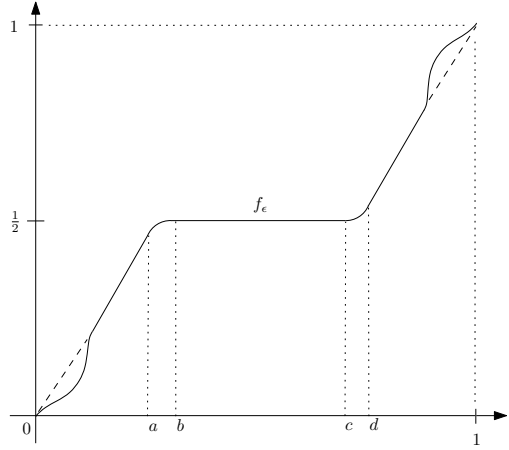} 
\par\end{centering}
\caption{\label{fig:MikeC}A net $(f_{\eps})$ defining a discontinuous solution
of a smooth equation.}
\end{figure}

\begin{thm}
\label{thm:classicalThms}Let $f\in{}^{\rho}\Gcinf(X,\rcrho^{d})$
be a generalized smooth function defined in the sharply open set $X\subseteq\rcrho^{n}$.
Let $a$, $b\in\rcrho^{n}$ such that $[a,b]\subseteq X$. Then 
\begin{enumerate}
\item \label{enu:meanValue}If $n=d=1$, then $\exists c\in[a,b]:\ f(b)-f(a)=(b-a)\cdot f'(c)$. 
\item \label{enu:integralMeanValue}If $n=d=1$, then $\exists c\in[a,b]:\ \int_{a}^{b}f(t)\,\diff{t}=(b-a)\cdot f(c)$. 
\item \label{enu:meanValueSevVars}If $d=1$, then $\exists c\in[a,b]:\ f(b)-f(a)=\nabla f(c)\cdot(b-a)$. 
\item \label{enu:IMVSevVars}Let $h:=b-a$, then $f(a+h)-f(a)=\int_{0}^{1}\diff{f}(a+t\cdot h).h\,\diff{t}$. 
\end{enumerate}
\end{thm}

\begin{proof}
Using the usual notations, for small $\eps$ we have $a_{\eps}<b_{\eps}$
and 
\begin{align}
\exists c_{\eps} & \in[a_{\eps},b_{\eps}]:\ f_{\eps}(b_{\eps})-f_{\eps}(a_{\eps})=(b_{\eps}-a_{\eps})\cdot f'_{\eps}(c_{\eps})\label{eq:epsMean}\\
\exists c_{\eps} & \in[a_{\eps},b_{\eps}]:\ \int_{a_{\eps}}^{b_{\eps}}f_{\eps}=(b_{\eps}-a_{\eps})\cdot f{}_{\eps}(c_{\eps}),\label{eq:epsIntMean}
\end{align}
from which the conclusions \ref{enu:meanValue} and \ref{enu:integralMeanValue}
follow directly. The several variables and vector valued cases \ref{enu:meanValueSevVars},
\ref{enu:IMVSevVars} follow as usual by reduction to the one-variable
and scalar valued case. 
\end{proof}
Internal sets generated by a sharply bounded net of compact sets serve
as a substitute for compact subsets for GSF, as can be seen from the
following extreme value theorem: 
\begin{lem}
\label{lem:extremeValueCont}Let $\emptyset\ne K=[K_{\eps}]\subseteq\rti^{n}$
be an internal set generated by a sharply bounded net $(K_{\eps})$
of compact sets $K_{\eps}\comp\R^{n}$ Assume that $\alpha:K\ra\rti$
is a well-defined map given by $\alpha(x)=[\alpha_{\eps}(x_{\eps})]$
for all $x\in K$, where $\alpha_{\eps}:K_{\eps}\ra\R$ are continuous
maps (e.g.~$\alpha(x)=|x|$). Then 
\[
\exists m,M\in K\,\forall x\in K:\ \alpha(m)\le\alpha(x)\le\alpha(M).
\]
\end{lem}

\begin{proof}
Since $K\ne\emptyset$, for $\eps$ sufficiently small, let us say
for $\eps\in(0,\eps_{0}]$, $K_{\eps}$ is non empty and, by our assumptions,
it is also compact. Since each $\alpha_{\eps}$ is continuous, for
all $\eps\in(0,\eps_{0}]$ we have 
\[
\exists m_{\eps},M_{\eps}\in K_{\eps}\,\forall x\in K_{\eps}:\ \alpha_{\eps}(m_{\eps})\le\alpha_{\eps}(x)\le\alpha_{\eps}(M_{\eps}).
\]
Since the net $(K_{\eps})$ is sharply bounded, both the nets $(m_{\eps})$
and $(M_{\eps})$ are moderate. Therefore $m=[m_{\eps}]$, $M=[M_{\eps}]\in K$.
Take any $x\in[K_{\eps}]$, then there exists a representative $(x_{\eps})$
such that $x_{\eps}\in K_{\eps}$ for $\eps$ small. Therefore $\alpha(m)=[\alpha_{\eps}(m_{\eps})]\le[\alpha_{\eps}(x{}_{\eps})]=\alpha(x)\le\alpha(M)$. 
\end{proof}
\begin{cor}
\label{cor:extremeValues}Let $f\in{}^{\rho}\Gcinf(X,\rcrho)$ be
a generalized smooth function defined in the subset $X\subseteq\rcrho^{n}$.
Let $\emptyset\ne K=[K_{\eps}]\subseteq X$ be an internal set generated
by a sharply bounded net $(K_{\eps})$ of compact sets $K_{\eps}\comp\R^{n}$,
then 
\begin{equation}
\exists m,M\in K\,\forall x\in K:\ f(m)\le f(x)\le f(M).\label{eq:epsExtreme}
\end{equation}
\end{cor}

These results motivate the following 
\begin{defn}
\label{def:functCmpt} A subset $K$ of $\rcrho^{n}$ is called \emph{functionally
compact}, denoted by $K\fcmp\rcrho^{n}$, if there exists a net $(K_{\eps})$
such that 
\begin{enumerate}
\item \label{enu:defFunctCmpt-internal}$K=[K_{\eps}]\subseteq\rcrho^{n}$ 
\item \label{enu:defFunctCmpt-sharpBound}$(K_{\eps})$ is sharply bounded 
\item \label{enu:defFunctCmpt-cmpt}$\forall\eps\in I:\ K_{\eps}\Subset\R^{n}$ 
\end{enumerate}
If, in addition, $K\subseteq U\subseteq\rcrho^{n}$ then we write
$K\fcmp U$. Finally, we write $[K_{\eps}]\fcmp U$ if \ref{enu:defFunctCmpt-sharpBound},
\ref{enu:defFunctCmpt-cmpt} and $[K_{\eps}]\subseteq U$ hold. 
\end{defn}

We refer to \cite{GK15} for a deeper study of this type of compact
sets in the case $\rho=(\eps)$. Note that any interval $[a,b]\subseteq\rti$
with $b-a\in\R_{>0}$, is \emph{not} connected: in fact if $c\in(a,b)$,
then both $c+D_{\infty}$ and $[a,b]\setminus\left(c+D_{\infty}\right)$
are sharply open in $[a,b]$. Once again, this is a general property
in several non-Archimedean frameworks (see e.g.~\cite{Rob73,Koc}).
On the other hand, as in the case of functionally compact sets, GSF
behave on intervals as if they were connected, in the sense that both
the intermediate value theorem Cor.~\ref{cor:intermValue} and the
extreme value theorem Cor.~\ref{cor:extremeValues} hold for them
(therefore, $f\left([a,b]\right)=\left[f(m),f(M)\right]$, where we
used the notations from the results just mentioned).

We close this section with generalizations of Taylor's theorem in
various forms. In the following statement, $\diff{^{k}f}(x):\rti^{dk}\ra\rti$
is the $k$-th differential of the GSF $f$, viewed as an $\rti$-multilinear
map $\rti^{d}\times\ptind^{k}\times\rti^{d}\ra\rti$, and we use the
common notation $\diff{^{k}f}(x)\cdot h^{k}:=\diff{^{k}f}(x)(h,\ldots,h)$.
Clearly, $\diff{^{k}f}(x)\in\gsf(\rti^{dk},\rti)$. For multilinear
maps $A:\rti^{p}\ra\rti^{q}$, we set $|A|:=[|A_{\eps}|]\in\rti$,
the generalized number defined by the norms of the operators $A_{\eps}:\R^{p}\ra\R^{q}$. 
\begin{thm}
\label{thm:Taylor}Let $f\in\gsf(U,\rcrho)$ be a generalized smooth
function defined in the sharply open set $U\subseteq\rcrho^{d}$.
Let $a$, $b\in\rcrho^{d}$ such that the line segment $[a,b]\subseteq U$,
and set $h:=b-a$. Then, for all $n\in\N$ we have 
\begin{enumerate}
\item \label{enu:LagrangeRest}$\exists\xi\in[a,b]:\ f(a+h)=\sum_{j=0}^{n}\frac{\diff{^{j}f}(a)}{j!}\cdot h^{j}+\frac{\diff{^{n+1}f}(\xi)}{(n+1)!}\cdot h^{n+1}.$ 
\item \label{enu:integralRest}$f(a+h)=\sum_{j=0}^{n}\frac{\diff{^{j}f}(a)}{j!}\cdot h^{j}+\frac{1}{n!}\cdot\int_{0}^{1}(1-t)^{n}\,\diff{^{n+1}f}(a+th)\cdot h^{n+1}\,\diff{t}.$ 
\end{enumerate}
\noindent Moreover, there exists some $R\in\rcrho_{>0}$ such that
\begin{equation}
\forall k\in B_{R}(0)\,\exists\xi\in[a,a+k]:\ f(a+k)=\sum_{j=0}^{n}\frac{\diff{^{j}f}(a)}{j!}\cdot k^{j}+\frac{\diff{^{n+1}f}(\xi)}{(n+1)!}\cdot k^{n+1}\label{eq:LagrangeInfRest}
\end{equation}
\begin{equation}
\frac{\diff{^{n+1}f}(\xi)}{(n+1)!}\cdot k^{n+1}=\frac{1}{n!}\cdot\int_{0}^{1}(1-t)^{n}\,\diff{^{n+1}f}(a+tk)\cdot k^{n+1}\,\diff{t}\approx0.\label{eq:integralInfRest}
\end{equation}
\end{thm}

Formulas \ref{enu:LagrangeRest} and \ref{enu:integralRest} correspond
to a plain generalization of Taylor's theorem for ordinary smooth
functions with Lagrange and integral remainder, respectively. Dealing
with generalized functions, it is important to note that this direct
statement also includes the possibility that the differential $\diff{^{n+1}f}(\xi)$
may be infinite at some point. For this reason, in \eqref{eq:LagrangeInfRest}
and \eqref{eq:integralInfRest}, considering a sufficiently small
increment $k$, we get more classical infinitesimal remainders $\diff{^{n+1}f}(\xi)\cdot k^{n+1}\approx0$. 
\begin{proof}
Let $f_{\eps}\in\cinfty(\R^{d},\R)$ be a net of smooth functions
that defines $f$. We have $a+h=b\in[a,b]\subseteq U$ and $U$ is
sharply open, so by the Taylor formula applied to $f_{\eps}$ and
by Theorem \ref{thm:FR-forGSF} we have 
\begin{align*}
f(a+h) & =[f_{\eps}(a_{\eps}+h_{\eps})]\\
 & =\left[\sum_{j=0}^{n}\frac{\diff{^{j}f_{\eps}}(a_{\eps})}{j!}h_{\eps}^{j}+\frac{\diff{^{n+1}f_{\eps}}(\xi_{\eps})}{(n+1)!}h_{\eps}^{n+1}\right]\\
 & =\sum_{j=0}^{n}\frac{\diff{^{j}f}(a)}{j!}h^{j}+\frac{\diff{^{n+1}f}(\xi)}{(n+1)!}h^{n+1}
\end{align*}
for some $\xi_{\eps}\in(a_{\eps},b_{\eps})$, and where $\xi=[\xi_{\eps}]\in\rcrho$
so that $\xi\in[a,b]$. Analogously, we can prove \ref{enu:integralRest}.

To prove the second part of the theorem, we start by considering a
sharp ball $B_{r}(a)\subseteq U$, where $r=[r_{\eps}]>0$. Set $H:=\left[\overline{\Eball_{r_{\eps}/2}(a_{\eps})}\right]$,
and 
\[
K:=\max\left(\left|\diff{^{n+1}f}(M)\right|,\left|\diff{^{n+1}f}(m)\right|\right)\in\rti,
\]
where $\left|\diff{^{n+1}f}(M)\right|$ and $\left|\diff{^{n+1}f}(m)\right|$
are the maximum and the minimum values of $\left|\diff{^{n+1}f}\right|=\left[\left|\diff{^{n+1}f_{\eps}}\right|\right]:U\ra\rti$
on $H\subseteq U$, see Cor.~\ref{cor:extremeValues}. We hence have
$\left|\diff{^{n+1}f}(\xi)\right|\le K$ for all $\xi\in H$. Take
any strictly positive number $P\in\rcrho_{>0}$ such that $P\ge K$
and any strictly positive infinitesimal $p\in\rcrho_{>0}$ so that
$\frac{p}{P}\approx0$ and hence $\left(\frac{p}{P}\right)^{n+1}\le\frac{p}{P}$.
Set $R:=\min\left(\frac{r}{2},\frac{p}{P}\right)$, then $R\in\rcrho_{>0}$
since both $r$ and $\frac{p}{P}$ are invertible. If $k\in B_{R}(0)$
then $[a,a+k]\subseteq H\subseteq U$. We can therefore apply \ref{enu:LagrangeRest}
to get \eqref{eq:LagrangeInfRest}. Finally 
\[
\left|\frac{\diff{^{n+1}f}(\xi)}{(n+1)!}k^{n+1}\right|\le\frac{K}{(n+1)!}R^{n+1}\le\frac{P}{(n+1)!}\cdot\left(\frac{p}{P}\right)^{n+1}\le\frac{P}{(n+1)!}\cdot\frac{p}{P}\approx0.
\]
\end{proof}
The following definitions allow us to state Taylor formulas in Peano
and in infinitesimal form. The latter has no remainder term thanks
to the use of an equivalence relation that permits the introduction
of a language of nilpotent infinitesimals, see e.g.~\cite{Gio10a}
for a similar formulation. For simplicity, we only present the 1-dimentional
case. 
\begin{defn}
\label{def:little-oh-equality_k} 
\begin{enumerate}
\item \label{enu:little-oh}Let $U\subseteq\rcrho$ be a sharp neighborhood
of $0$ and $P$, $Q:U\longrightarrow\rcrho$ be maps defined on $U$.
Then we say that 
\[
P(u)=o(Q(u))\quad\text{as }u\to0
\]
if there exists a function $R:U\longrightarrow\rcrho$ such that 
\[
\forall u\in U:\ P(u)=R(u)\cdot Q(u)\quad\text{and}\quad\lim_{u\to0}R(u)=0,
\]
where the limit is taken in the sharp topology. 
\item \label{enu:equalityUpTo-k-thOrderInf}Let $x$, $y\in\rcrho$ and
$k$, $j\in\R_{>0}$, then we write $x=_{j}y$ if there exist representatives
$(x_{\eps})$, $(y_{\eps})$ of $x$, $y$, respectively, such that
\begin{equation}
|x_{\eps}-y_{\eps}|=O(\rho_{\eps}^{\frac{1}{j}}).\label{eq:equalityUpTo_j}
\end{equation}
We will read $x=_{j}y$ as $x$ \emph{is equal to} $y$ \emph{up to}
$j+1$\emph{-th order infinitesimals}. Finally, if $k\in\N_{>0}$,
we set $D_{kj}:=\left\{ x\in\rcrho\mid x^{k+1}=_{j}0\right\} $, which
is called the \emph{set of} $k$\emph{-th order infinitesimals for
the equality $=_{j}$}, and 
\[
D_{\infty j}:=\left\{ x\in\rcrho\mid\exists k\in\N_{>0}:\ x^{k+1}=_{j}0\right\} 
\]
which is called the \emph{set of infinitesimals for the equality $=_{j}$}. 
\end{enumerate}
\end{defn}

Of course, the reformulation of Def.~\ref{def:little-oh-equality_k}
\ref{enu:little-oh} for the classical Landau's little-oh is particularly
suited to the case of a ring like $\rcrho$, instead of a field. The
intuitive interpretation of $x=_{j}y$ is that for particular (physics-related)
problems one is not interested in distinguishing quantities whose
difference $|x-y|$ is less than an infinitesimal of order $j$. The
idea behind taking $\frac{1}{j}$ in \eqref{eq:equalityUpTo_j} is
to obtain the property that the greater the order $j$ of the infinitesimal
error, the greater the difference $|x-y|$ is allowed to be. This
is a typical property in rings with nilpotent infinitesimals (see
e.g.~\cite{Gio10a,Koc}). The set $D_{ki}$ represents the neighborhood
of infinitesimals of $k$-th order for the equality $=_{j}$. Once
again, the greater the order $k$, the bigger is the neighborhood
(see Theorem \ref{thm:TaylorPeano-Infinitesimals} \ref{enu:chainOfInfNeighborhoods}
below). 
\begin{thm}
\label{thm:TaylorPeano-Infinitesimals}Let $f\in\gsf(U,\rcrho)$ be
a generalized smooth function defined in the sharply open set $U\subseteq\rcrho$.
Let $x$, $\delta\in\rcrho$, with $\delta>0$ and $[x-\delta,x+\delta]\subseteq U$.
Let $k$, $l$, $j\in\R_{>0}$. Then 
\begin{enumerate}
\item \label{enu:Peano}$\forall n\in\N:\ f(x+u)=\sum_{r=0}^{n}\frac{f^{(r)}(x)}{r!}u^{r}+o(u^{n})$
as $u\to0$. 
\item \label{enu:indepedRepr}The definition of $x=_{j}y$ does not depend
on the representatives of $x$, $y$. 
\item \label{enu:equivalenceRel}$=_{j}$ is an equivalence relation on
$\rcrho$. 
\item \label{enu:relationBetweenD_j-And-D_l}If $x=_{j}y$ and $l\ge j$,
then $x=_{l}y$. 
\item \label{enu:relationDandD_j}If $\forall^{0}j\in\R_{>0}:\ x=_{j}y$,
then $x=y$. 
\item \label{enu:=00003D00003Dj-sumProd}If $x=_{j}y$ and $z=_{j}w$ then
$x+z=_{j}y+w$. If $x$ and $z$ are finite, then $x\cdot z=_{j}y\cdot w$. 
\item \label{enu:D_k-infinitesimal}$\forall h\in D_{kj}:\ h\approx0$. 
\item \label{enu:chainOfInfNeighborhoods}$D_{mj}\subseteq D_{kj}\subseteq D_{\infty j}$
if $m\le k$. 
\item \label{enu:D_j-AlmostIdeal}$D_{kj}$ is a subring of $\rcrho$. For
all $h\in D_{kj}$ and all finite $x\in\rcrho$, we have $x\cdot h\in D_{kj}$. 
\item \label{enu:infTaylor}Let $n\in\N_{>0}$ and assume that $j$ satisfies
\begin{equation}
\forall z\in\rcrho\,\forall\xi\in[x-\delta,x+\delta]:\ z=_{j}0\then z\cdot f^{(n+1)}(\xi)=_{j}0.\label{eq:relBetween-i-AndInfinite-f^(n+1)}
\end{equation}
Then there exist $k\in\R_{>0}$ such that $k\le n$ and 
\[
\forall u\in D_{kj}:\ f(x+u)=_{j}\sum_{r=0}^{n}\frac{f^{(r)}(x)}{r!}u^{r}.
\]
\item \label{enu:infTaylor2}For all $n\in\N_{>0}$ there exist $e$, $k\in\R_{>0}$
such that $e\le j$, $k\le n$ and $\forall u\in D_{ke}:\ f(x+u)=_{e}\sum_{r=0}^{n}\frac{f^{(r)}(x)}{r!}u^{r}$. 
\end{enumerate}
\end{thm}

\begin{proof}
In order to prove \ref{enu:Peano} we set $P(u)=f(x+u)-\sum_{r=0}^{n}\frac{f^{(r)}(x)}{r!}u^{r}$,
$Q(u)=u^{n}$ and $R(u)=u\cdot\int_{0}^{1}\frac{f^{(n+1)}(x+tu)}{n!}(1-t)^{n}\,\diff{t}$
for $u\in B_{\delta}(0)$. The segment $[x-u,x+u]\subseteq B_{\delta}(x)\subseteq U$,
so Thm.~\ref{thm:Taylor} \ref{enu:integralRest} yields $P(u)=Q(u)\cdot R(u)$
for all $u\in U_{x}$. As in the previous proof, set 
\[
K:=\max\left(\left|f^{(n+1)}(M)\right|,\left|f^{(n+1)}(m)\right|\right)
\]
so that $\left|f^{(n+1)}(\xi)\right|\le K$ for all $\xi\in[x-\delta,x+\delta]$,
then 
\[
|R(u)|\le|u|\cdot\left|\int_{0}^{1}\frac{f^{(n+1)}(x+tu)}{n!}(1-t)^{n}\,\diff{t}\right|\le|u|\cdot\frac{K}{(n+1)!}
\]
which goes to $0$ as $u\to0$ in the sharp topology.

The proofs of \ref{enu:indepedRepr}-\ref{enu:D_j-AlmostIdeal} are
simple. We only prove that $D_{kj}$ is closed with respect to sums.
Let $x$, $y\in D_{kj}$ so that 
\begin{equation}
\left|\frac{x_{\eps}^{k+1}}{\rho_{\eps}^{\frac{1}{j}}}\right|\le M\quad,\quad\left|\frac{y_{\eps}^{k+1}}{\rho_{\eps}^{\frac{1}{j}}}\right|\le N\label{eq:x-y-In-D_j}
\end{equation}
for $\eps$ small and for some $M$, $N\in\R_{>0}$. Then 
\begin{align*}
\left|\frac{(x_{\eps}+y_{\eps})^{k+1}}{\rho_{\eps}^{\frac{1}{j}}}\right| & \le\sum_{r=0}^{k+1}{k+1 \choose r}\left|\frac{x_{\eps}^{k+1}}{\rho_{\eps}^{\frac{1}{j}}}\right|^{\frac{r}{k+1}}\left|\frac{y_{\eps}^{k+1}}{\rho_{\eps}^{\frac{1}{j}}}\right|^{\frac{k+1-r}{k+1}}\\
 & \le\sum_{r=0}^{k+1}{k+1 \choose r}M^{\frac{r}{k+1}}N^{\frac{k+1-r}{k+1}},
\end{align*}
proving the claim.

In order to show \ref{enu:infTaylor}, we first note that $x=_{j}y$
is equivalent to 
\[
\exists A\in\R_{>0}:\ |x-y|\le A\cdot\diff{\rho}{}^{\frac{1}{j}}.
\]
We again use the notation $K:=\max\left(\left|f^{(n+1)}(M)\right|,\left|f^{(n+1)}(m)\right|\right)$
and set $p:=\inf\left\{ q\in\R\mid\exists B\in\R_{>0}:\ K\le B\diff{\rho}{}^{-q}\right\} $
so that for any $r\in\R_{>0}$ we have 
\begin{align}
\left|f(x+u)-\sum_{r=0}^{n}\frac{f^{(r)}(x)}{r!}u^{r}\right| & \le\frac{K}{(n+1)!}\cdot|u|^{n+1}\nonumber \\
 & \le B\diff{\rho}^{-p-r}\cdot|u|^{n+1}\label{eq:infTaylorDiffEstim}
\end{align}
for some $B\in\R_{>0}$. Now assume that $u\in D_{kj}$. Then 
\[
|u|^{k+1}\le A\cdot\diff{\rho}^{\frac{1}{j}}
\]
for some $A\in\R_{>0}$, so that 
\[
|u|^{n+1}\le A\cdot\diff{\rho}^{\frac{n+1}{j(k+1)}}.
\]
This and \eqref{eq:infTaylorDiffEstim} yield 
\[
\left|f(x+u)-\sum_{r=0}^{n}\frac{f^{(r)}(x)}{r!}u^{r}\right|\le AB\cdot\diff{\rho}^{\frac{n+1}{j(k+1)}-p-r},
\]
which gives the desired conclusion if $\frac{n+1}{j(k+1)}-p-r\ge\frac{1}{j}$.
There exists such an $r\in\R_{>0}$ if and only if $\frac{n+1}{j(k+1)}-p-\frac{1}{j}>0$,
i.e.~for $(k+1)(jp+1)<n+1$. If $jp+1\le0$, then any $k>0$ satisfies
this inequality. Otherwise, we need to prove the existence of $k<\frac{n-jp}{jp+1}$.
Since $n\in\N_{>0}$, a sufficient condition for this is that $jp<1$.
This gives a relation between the order of the equality $=_{j}$ and
the order of infinity $p$ of the derivative $f^{(n+1)}$ on the interval
$[x-\delta,x+\delta]$. In order to prove that $jp<1$, we thus need
to use the assumption \eqref{eq:relBetween-i-AndInfinite-f^(n+1)}.
Set $z:=\diff{\rho}^{\frac{s}{j}}$, where $1\le s<2$. Then $0\le z\le\diff{\rho}^{\frac{1}{j}}$,
i.e.~$z=_{j}0$, and \eqref{eq:relBetween-i-AndInfinite-f^(n+1)}
yields 
\begin{equation}
\forall\xi\in[x-\delta,x+\delta]:\ z\cdot\left|f^{(n+1)}(\xi)\right|=_{j}0\label{eq:z-time-f(n+1)}
\end{equation}
Taking $\xi=M$ in \eqref{eq:z-time-f(n+1)} we obtain 
\[
\left|f^{(n+1)}(M)\right|\le C\cdot\diff{\rho}^{-\frac{s-1}{j}}
\]
for some $C\in\R_{>0}$. Analogously 
\[
\left|f^{(n+1)}(m)\right|\le D\cdot\diff{\rho}^{-\frac{s-1}{j}}.
\]
for some $D$ that we may assume to be $\ge C$ so that $K\le D\diff{\rho}^{-\frac{s-1}{j}}$,
which implies $p\le\frac{s-1}{j}$ by definition of $p$. Therefore,
$jp\le s-1<1$, which proves our claim.

To prove \ref{enu:infTaylor2}, we only need to note that there always
exists an $e\in\R_{>0}$ such that $ep<1$, $e\le j$, so that we
can repeat the previous deduction with $e$ instead of $j$. 
\end{proof}

\section{\label{sec:n-dimIntegral}Multidimensional integration and hyperlimits}

In this section we want to introduce integration of GSF over functionally
compact sets with respect to an arbitrary Borel measure $\mu$.

The possibility to achieve results mirroring classical limit theorems
for this notion of integral is closely linked to the introduction
of the notion of hyperlimit, i.e.~of limits of sequences of generalized
numbers $a=(a_{n})_{n\in\N}:\hyperNrho\ra\RC{\sigma}$, where $\sigma$
and $\rho$ are two gauges (see Def.~\ref{def:RCGN}) and $n\to+\infty$
along generalized natural numbers, i.e.~for 
\[
n\in\hyperNrho:=\left\{ [n_{\eps}]\in\rcrho\mid n_{\eps}\in\N\ \forall\eps\right\} .
\]
Mimicking nonstandard analysis, the numbers $n\in\hyperNrho$ are
called \emph{hypernatural} numbers. To glimpse the necessity of studying
$\hyperNrho$, it suffices to note that $\frac{1}{n}<\diff{\rho}^{q}$
is always false for $n\in\N$ but it can be satisfied for suitable
$n\in\hyperNrho$. Therefore, if $\lim_{n\to+\infty}a_{n}=0$ in the
classical sense, i.e.~for $n\in\N$ and with respect to the sharp
topology, then necessarily $a_{n}$ is infinitesimal for $n\in\N$
sufficiently large. This represents a severe limitation for this notion
of limit. It is also clear from the fact that $\rcrho$ with the sharp
topology is an ultra-pseudometric space, see e.g.~\cite{S0}, and
hence a series in $\rcrho$ converges in the sharp topology if and
only if its general term $a_{n}\to0$ as $n\to+\infty$, $n\in\N$,
in the sharp topology, see \cite{Kob84}.

\subsection{Integration over functionally compact sets}

In section \ref{sec:Integral-calculus}, we already defined a notion
of integral over intervals using the notion of primitive. This notion
does not help if we want to define the integral $\int_{D}f$ of a
GSF $f$ over a domain $D\subseteq\rcrho^{n}$ which is more general
than an interval. In this case, it is natural to try an $\eps$-wise
definition of the type $\int_{D}f\,\diff{\mu}:=\left[\int_{D_{\eps}}f_{\eps}\,\diff{\mu}\right]\in\rcrho$,
where the net $(f_{\eps})$ defines the GSF $f$ and the net $(D_{\eps})$
determines, in some way, the subset $D\subseteq\rcrho^{n}$, e.g.~$D=[D_{\eps}]$
in case of internal sets. In pursuing this idea, it is important to
recall that the internal set (interval) $[0,1]=\left[[0,1]_{\R}\right]$
can also be defined by a net of finite sets. Indeed, if $\text{int}(-)$
is the integer part function, and we set 
\begin{align}
N_{\eps} & :=\text{int}\left(\rho_{\eps}^{-1/\eps}\right)\nonumber \\
K_{\eps} & :=\{\rho_{\eps}^{1/\eps},2\rho_{\eps}^{1/\eps},\ldots,N_{\eps}\rho_{\eps}^{1/\eps}\}\label{eq:B_eps}
\end{align}
then the Hausdorff distance $d_{\text{H}}([0,1]_{\R},K_{\eps})=\rho_{\eps}^{1/\eps}$
and hence $[0,1]=[K_{\eps}]$ (see also \cite{Ver08,GKV}). Consequently,
if $\lambda$ is the Lebesgue measure on $\R$, we have that the generalized
number $\left[\lambda([0,1]_{\R})\right]=1$, whereas $\left[\lambda(K_{\eps})\right]=0$
and, in general, $\left[\int_{[0,1]_{\R}}f_{\eps}\,\diff{\lambda}\right]\ne\left[\int_{K_{\eps}}f_{\eps}\,\diff{\lambda}\right]=0$.
Therefore, even the definition of integral over an interval cannot
be easily accomplished by proceeding $\eps$-wise, i.e.~on defining
nets.

If we try to understand when such an $\eps$-wise definition can be
accomplished, it turns out that we have to consider an enlargement
$\overline{\Eball}_{\rho_{\eps}^{m}}(K_{\eps})$ and then take $m\to+\infty$.
This is indeed quite natural if one keeps in mind that $[K_{\eps}]=[L_{\eps}]$
if and only if the Hausdorff distance $d_{\text{H}}(K_{\eps},L_{\eps})$
defines a negligible nets, (see \cite{Ver08,GKV}). In the following,
we say that $(K_{\eps})$ is a \emph{representative of }$K\fcmp\rcrho^{n}$
if $K=[K_{\eps}]$, $(K_{\eps})$ is sharply bounded, and $K_{\eps}\Subset\R^{n}$
for all $\eps$. 
\begin{defn}
\label{def:intOverCompact}Let $\mu$ be a Borel measure on $\R^{n}$
and let $K$ be a functionally compact subset of $\RC{\rho}^{n}$.
Then we call $K$ $\mu$-measurable if the limit 
\begin{equation}
\mu(K):=\lim_{\substack{m\to\infty\\
m\in\N
}
}[\mu(\overline{\Eball}_{\rho_{\eps}^{m}}(K_{\eps}))]\label{eq:muMeasurable}
\end{equation}
exists for some representative $(K_{\eps})$ of $K$. The limit is
taken in the sharp topology on $\RC{\rho}$, and $\overline{\Eball}_{r}(A):=\{x\in\R^{n}:d(x,A)\le r\}$. 
\end{defn}

In the following result, we will prove that this definition satisfies
our requirements. We will occasionally integrate generalized functions
more general than GSF: 
\begin{defn}
\label{def:integrableMap}Let $K\fcmp\RC{\rho}^{n}$. Let $(\Omega_{\eps})$
be a net of open subsets of $\R^{n}$, and $(f_{\eps})$ be a net
of continuous maps $f_{\eps}$: $\Omega_{\eps}\longrightarrow\R$.
Then we say that 
\[
(f_{\eps})\textit{ defines a generalized integrable map}:K\longrightarrow\RC{\rho}
\]
if 
\begin{enumerate}
\item $K\subseteq\sint{\Omega_{\eps}}$ and $[f_{\eps}(x_{\eps})]\in\RC{\rho}$
for all $[x_{\eps}]\in K$. 
\item $\forall(x_{\eps}),(x'_{\eps})\in\R_{\rho}^{n}:\ [x_{\eps}]=[x'_{\eps}]\in K\ \Rightarrow\ (f_{\eps}(x_{\eps}))\sim_{\rho}(f_{\eps}(x'_{\eps}))$. 
\end{enumerate}
\noindent If $f\in\Set(K,\RC{\rho})$ is such that 
\begin{equation}
\forall[x_{\eps}]\in K:\ f\left([x_{\eps}]\right)=\left[f_{\eps}(x_{\eps})\right]
\end{equation}
we say that $f:K\longrightarrow\RC{\rho}$ is a \emph{generalized
integrable function}.

\noindent We will again say that $f$ \emph{is defined by the net}
$(f_{\eps})$ or that the net $(f_{\eps})$ \emph{represents} $f$.
The set of all these generalized integrable functions will be denoted
by $\GI(K,\RC{\rho})$. 
\end{defn}

\noindent E.g., if $f=[f_{\eps}(-)]|_{K}\in\gsf(K,\RC{\rho})$, then
both $f$ and $|f|=[|f_{\eps}(-)|]|_{K}$ are integrable on $K$.

As in Lemma \ref{lem:fromOmega_epsToRn}, we may assume without loss
of generality that $f_{\eps}$ are continuous maps defined on the
whole of $\R^{n}$. 
\begin{thm}
\label{thm:muMeasurableAndIntegral}Let $K\subseteq\RC{\rho}^{n}$
be $\mu$-measurable. 
\begin{enumerate}
\item \label{enu:indepRepr}The definition of $\mu(K)$ is independent of
the representative $(K_{\eps})$. 
\item \label{enu:existsRepre}There exists a representative $(K_{\eps})$
of $K$ such that $\mu(K)=[\mu(K_{\eps})]$. 
\item \label{enu:epsWiseDefInt}Let $(K_{\eps})$ be any representative
of $K$ and let $f=[f_{\eps}(-)]|_{K}\in\GI(K,\RC{\rho})$. Then 
\[
\int_{K}f\,\diff{\mu}:=\lim_{m\to\infty}\biggl[\int_{\overline{\Eball}_{\rho_{\eps}^{m}}(K_{\eps})}f_{\eps}\,\diff{\mu}\biggr]
\]
exists and its value is independent of the representative $(K_{\eps})$. 
\item \label{enu:existsReprDefInt}There exists a representative $(K_{\eps})$
of $K$ such that 
\begin{equation}
\int_{K}f\,\diff{\mu}=\biggl[\int_{K_{\eps}}f_{\eps}\,\diff{\mu}\biggr]\label{eq:measurable}
\end{equation}
for each $f=[f_{\eps}(-)]|_{K}\in\GI(K,\RC{\rho})$. 
\item \label{enu:indepReprDefInt}If \eqref{eq:measurable} holds, then
the same holds for any representative $(L_{\eps})$ of $K$ with $L_{\eps}\supseteq K_{\eps}$,
$\forall^{0}\eps$. 
\end{enumerate}
\end{thm}

\begin{proof}
\ref{enu:indepRepr} Let $(L_{\eps})$ be another representative.
As $[K_{\eps}]\subseteq[L_{\eps}]$, we have that $(\sup_{x\in K_{\eps}}d(x,L_{\eps}))_{\eps}=:(n_{\eps})$
is negligible, so $K_{\eps}\subseteq\overline{\Eball}_{n_{\eps}}(L_{\eps})$,
and $\mu(\overline{\Eball}_{\rho_{\eps}^{m}}(K_{\eps}))\le\mu(\overline{\Eball}_{\rho_{\eps}^{m-1}}(L_{\eps}))$.
Also using this inequality with the roles of $K_{\eps}$ and $L_{\eps}$
interchanged, we see that $\lim_{m\to\infty}[\mu(\overline{\Eball}_{\rho_{\eps}^{m}}(L_{\eps}))]$
exists and that it equals $\lim_{m\to\infty}[\mu(\overline{\Eball}_{\rho_{\eps}^{m}}(K_{\eps}))]$.

\ref{enu:existsRepre} Call $\left[c_{\eps}\right]:=\mu(K)$ and let
$K=[L_{\eps}]$. By definition of $\mu$-measurable set and by the
previous point \ref{enu:indepRepr}, for any $q\in\N$, there exists
$m_{q}\in\N$ (w.l.o.g.\ $m_{q}\ge q$) and $\eps_{q}>0$ (w.l.o.g.\ $\eps_{q}<\eps_{q-1}$
and $\eps_{q}<1/q$) such that 
\[
|\mu(\overline{\Eball}_{\rho_{\eps}^{m_{q}}}(L_{\eps}))-c_{\eps}|\le\rho_{\eps}^{q},\qquad\forall\eps\le\eps_{q}
\]
Now let $q_{\eps}:=q$ if $\eps\in(\eps_{q+1},\eps_{q}]$. Then $q_{\eps}\to\infty$
as $\eps\to0$ and 
\[
[\mu(\overline{\Eball}_{\rho_{\eps}^{m_{q_{\eps}}}}(L_{\eps}))]=[c_{\eps}]=\mu(K).
\]
As also $(\rho_{\eps}^{m_{q_{\eps}}})$ is negligible, we have $K=[\overline{\Eball}_{\rho_{\eps}^{m_{q_{\eps}}}}(L_{\eps})]$
and hence the conclusion follows for $K_{\eps}:=\overline{\Eball}_{\rho_{\eps}^{m_{q_{\eps}}}}(L_{\eps})$.

\ref{enu:epsWiseDefInt}–\ref{enu:existsReprDefInt}. Choose a representative
$(K_{\eps})$ as in part \ref{enu:existsRepre}. Then 
\[
\left|\int_{\overline{\Eball}_{\rho_{\eps}^{m}}(K_{\eps})}f_{\eps}\,\diff{\mu}-\int_{K_{\eps}}f_{\eps}\,\diff{\mu}\right|\le\mu(\overline{\Eball}_{\rho_{\eps}^{m}}(K_{\eps})\setminus K_{\eps})\sup_{\overline{\Eball}_{\rho_{\eps}^{m}}(K_{\eps})}|f_{\eps}|.
\]
As $[\mu(\overline{\Eball}_{\rho_{\eps}^{m}}(K_{\eps})\setminus K_{\eps})]=[\mu(\overline{\Eball}_{\rho_{\eps}^{m}}(K_{\eps}))]-[\mu(K_{\eps})]\to0$
as $m\to\infty$ and since $(\sup_{\overline{\Eball}_{\rho_{\eps}^{m}}(K_{\eps})}|f_{\eps}|)$
is moderate for some $m$ and decreasing in $m$, we find that 
\[
\lim_{m\to\infty}\left[\int_{\overline{\Eball}_{\rho_{\eps}^{m}}(K_{\eps})}f_{\eps}\,\diff{\mu}\right]=\left[\int_{K_{\eps}}f_{\eps}\,\diff{\mu}\right]
\]
exists. Independence of the representative of $K$ follows as in part
\ref{enu:indepRepr}, if $f_{\eps}\ge0$. The general case follows
by considering the positive and negative part of $f_{\eps}$.

\ref{enu:indepReprDefInt} Let $f_{\eps}\ge0$. Then by assumption,
$[\int_{K_{\eps}}f_{\eps}\,\diff{\mu}]\le[\int_{L_{\eps}}f_{\eps}\,\diff{\mu}]$.
For the converse inequality, observe that $[\int_{L_{\eps}}f_{\eps}\,\diff{\mu}]\le[\int_{\overline{\Eball}_{\rho_{\eps}^{m}}(L_{\eps})}f_{\eps}\,\diff{\mu}]$
for each $m\in\N$. Again the general case follows by considering
the positive and negative part of $f_{\eps}$. 
\end{proof}
The following Lemma provides an alternative characterization of $\mu$-measurability: 
\begin{lem}
\label{lem:altCharMeasurability}A functionally compact set $K$ is
$\mu$-measurable if and only if there exists a representative $(K_{\eps})$
of $K$ such that $[\mu(K_{\eps})]=[\mu(L_{\eps})]$, for each representative
$(L_{\eps})$ of $K$ with $L_{\eps}\supseteq K_{\eps}$, $\forall^{0}\eps$. 
\end{lem}

\begin{proof}
$\Rightarrow$: by the previous Thm.~\ref{thm:muMeasurableAndIntegral}.

$\Leftarrow$: it suffices to show that $\lim_{m\to\infty}[\mu(\overline{\Eball}_{\rho_{\eps}^{m}}(K_{\eps}))]=[\mu(K_{\eps})]$.
Seeking a contradiction, suppose that there exists $q\in\N$ for which
it does not hold that 
\[
\exists M\,\forall m\ge M\,\forall^{0}\eps:\ |\mu(\overline{\Eball}_{\rho_{\eps}^{m}}(K_{\eps}))-\mu(K_{\eps})|\le\rho_{\eps}^{q}.
\]
Then we can construct a strictly increasing sequence $(m_{k})_{k}\to\infty$
and a strictly decreasing sequence $(\eps_{k})_{k}\to0$ such that
$|\mu(\overline{\Eball}_{\rho_{\eps_{k}}^{m_{k}}}(K_{\eps_{k}}))-\mu(K_{\eps_{k}})|>\rho_{\eps_{k}}^{q}$,
$\forall k$.\\
 Let $L_{\eps}:=\overline{\Eball}_{\rho_{\eps}^{m_{k}}}(K_{\eps})$,
whenever $\eps\in(\eps_{k+1},\eps_{k}]$, $\forall k$. Then $K=[L_{\eps}]$,
but $[\mu(K_{\eps})]\ne[\mu(L_{\eps})]$, as $|\mu(L_{\eps})-\mu(K_{\eps})|>\rho_{\eps}^{q}$,
for each $\eps=\eps_{k}$ ($k\in\N$). 
\end{proof}
\begin{example}
\label{exa:measurability} Let $\lambda$ denote the Lebesgue-measure. 
\begin{enumerate}
\item If $K=\prod_{i=1}^{n}[a_{i},b_{i}]$, then $K$ is $\lambda$-measurable
with 
\[
\int_{K}f\,\diff{\lambda}=\left[\int_{a_{1,\eps}}^{b_{1,\eps}}\,dx_{1}\dots\int_{a_{n,\eps}}^{b_{n,\eps}}f_{\eps}(x_{1},\dots,x_{n})\,\diff{x_{n}}\right]
\]
for any representatives $(a_{i,\eps})$, $(b_{i,\eps})$ of $a_{i}$
and $b_{i}$, respectively. 
\item Let $\rho_{\eps}=\eps$, and 
\[
{\textstyle K:=\left\{ \frac{1}{n}\mid n\in\N_{>0}\right\} \cup\{0\}.}
\]
Then $[K]$ is $\lambda$-measurable with $\lambda([K])=0$. Indeed,
the contribution of $\{1/n\mid n>\eps^{-m/2}\}$ to $\lambda(\overline{\Eball}_{\eps^{m}}(K))$
is at most $\eps^{m/2}+2\eps^{m}$, while the contribution of $\{1/n\mid n\le\eps^{-m/2}\}$
is at most $2\eps^{m}\eps^{-m/2}=2\eps^{m/2}$. Thus $\lim_{m\to\infty}[\lambda(\overline{\Eball}_{\eps^{m}}(K))]=0$. 
\item Let $\rho_{\eps}=\eps$, and 
\[
K:={\textstyle \left\{ \frac{1}{\log n}\mid n\in\N_{>1}\right\} \cup\{0\}.}
\]
Then $[K]$ is not $\lambda$-measurable. For, if $n\ge\frac{\eps^{-m}}{(\log\eps^{-m})^{2}}$,
then, by the mean value theorem, 
\[
\frac{1}{\log n}-\frac{1}{\log(n+1)}\le\frac{1}{n(\log n)^{2}}\le\frac{\eps^{m}(\log\eps^{-m})^{2}}{\Bigl(\log\Bigl(\frac{\eps^{-m}}{(\log\eps^{-m})^{2}}\Bigr)\Bigr)^{2}}\le2\eps^{m}
\]
for small $\eps$. So the contribution of $\left\{ \frac{1}{\log n}\mid n\ge\frac{\eps^{-m}}{(\log\eps^{-m})^{2}}\right\} $
to $\lambda(\overline{\Eball}_{\eps^{m}}(K))$ lies between $\frac{1}{\log(\eps^{-m})}$
and $\frac{2}{\log(\eps^{-m})}$ for small $\eps$. The contribution
of $\Bigl\{\frac{1}{\log n}\mid n<\frac{\eps^{-m}}{(\log\eps^{-m})^{2}}\Bigr\}$
to $\lambda(\overline{\Eball}_{\eps^{m}}(K))$ is at most $\frac{2}{(\log\eps^{-m})^{2}}$,
which is of a lower order. Thus $\lim_{m\to\infty}[\lambda(\overline{\Eball}_{\eps^{m}}(K))]$
does not exist. 
\end{enumerate}
\end{example}

\subsection{Hyperfinite limits}

We start by defining the set of hypernatural numbers in $\rcrho$
and the set of $\rho$-moderate nets of natural numbers. For a deeper
study of these notions, see \cite{MTAG20}. 
\begin{defn}
\label{def:hypernatural}We set 
\begin{enumerate}
\item $\hyperNrho:=\left\{ [n_{\eps}]\in\rcrho\mid n_{\eps}\in\N\quad\forall\eps\right\} $ 
\item $\N_{\rho}:=\left\{ (n_{\eps})\in\R_{\rho}\mid n_{\eps}\in\N\quad\forall\eps\right\} .$ 
\end{enumerate}
\end{defn}

\noindent Therefore, $n\in\hyperNrho$ if and only if there exists
$(x_{\eps})\in\R_{\rho}$ such that $n=[\text{int}(|x_{\eps}|)]$.
Clearly, $\N\subset\hyperNrho$. Note that the integer part function
$\text{int}(-)$ is not well-defined on $\rcrho$. In fact, if $x=1=\left[1-\rho_{\eps}^{1/\eps}\right]=\left[1+\rho_{\eps}^{1/\eps}\right]$,
then $\text{int}\left(1-\rho_{\eps}^{1/\eps}\right)=0$, whereas $\text{int}\left(1+\rho_{\eps}^{1/\eps}\right)=1$,
for $\eps$ sufficiently small. Similar counterexamples can be constructed
for floor and ceiling functions.

\noindent However, the nearest integer function is well defined on
$\hyperNrho$. 
\begin{lem}
\label{lem:nearestInt}Let $(n_{\eps})\in\N_{\rho}$ and $(x_{\eps})\in\R_{\rho}$
be such that $[n_{\eps}]=[x_{\eps}]$. Let $\text{\emph{rpi}}:\R\ra\N$
be the function rounding to the nearest integer with tie breaking
towards positive infinity. Then $\text{\emph{rpi}}(x_{\eps})=n_{\eps}$
for $\eps$ small. The same result holds using $\text{\emph{rni}}:\R\ra\N$,
the function rounding half towards $-\infty$. 
\end{lem}

\begin{proof}
We have $\text{rpi}(x)=\lfloor x+\frac{1}{2}\rfloor$, where $\lfloor-\rfloor$
is the floor function. For $\eps$ small, $\rho_{\eps}<\frac{1}{2}$
and, since $[n_{\eps}]=[x_{\eps}]$, for such $\eps$ we can also
have $n_{\eps}-\rho_{\eps}+\frac{1}{2}<x_{\eps}+\frac{1}{2}<n_{\eps}+\rho_{\eps}+\frac{1}{2}$.
But $n_{\eps}\le n_{\eps}-\rho_{\eps}+\frac{1}{2}$ and $n_{\eps}+\rho_{\eps}+\frac{1}{2}<n_{\eps}+1$.
Therefore $\lfloor x_{\eps}+\frac{1}{2}\rfloor=n_{\eps}$. An analogous
argument can be applied to $\text{rni}(-)$. 
\end{proof}
\noindent Actually, this lemma does not allow us to define a \emph{nearest
integer }function $\nint{}:\hyperNrho\ra\N_{\rho}$ as $\nint{([x_{\eps}])}:=\text{rpi}(x_{\eps})$
because if $[x_{\eps}]=[n_{\eps}]$, the equality $n_{\eps}=\text{rpi}(x_{\eps})$
holds only for $\eps$ small. We should therefore consider the function
$\nint{}$ as valued in the germs for $\eps\to0^{+}$ generated by
nets in $\N_{\rho}$. A simpler approach is to choose a representative
$(n_{\eps})\in\N_{\rho}$ for each $x\in\hyperNrho$ and to define
$\nint{(x)}:=(n_{\eps})$. Clearly, we must consider the net $\left(\nint{(x)}_{\eps}\right)$
only for $\eps$ small, such as in equalities of the form $x=\left[\nint{(x)}_{\eps}\right]$.
This is what we do in the following 
\begin{defn}
\label{def:nint}The nearest integer function $\nint(-)$ is defined
by: 
\begin{enumerate}
\item $\nint:\hyperNrho:\ra\N_{\rho}$ 
\item If $[x_{\eps}]\in\hyperNrho$ and $\nint\left([x_{\eps}]\right)=(n_{\eps})$
then $\forall^{0}\eps:\ n_{\eps}=\text{rpi}(x_{\eps})$. 
\end{enumerate}
In other words, if $x\in\hyperNrho$, then $x=\left[\nint(x)_{\eps}\right]$
and $\nint(x)_{\eps}\in\N$ for all $\eps$. 
\end{defn}

We first consider the notion of hyperlimit. As we will see clearly
in Example \ref{exa:hyperlimits}.\ref{enu:twoGauges}, a key point
in the definition of hyperlimit is to consider \emph{two} gauges.
This is a natural way of proceeding because different gauges define
different topologies. On the other hand, the notion of hyperlimit
corresponds exactly to that of limit in the sharp topology on $\rcrho$
of a generalized sequence (hypersequence), i.e.~defined on the directed
set $\hyperN{\sigma}$. 
\begin{defn}
\label{def:hyperlimit}Let $\rho$, $\sigma$ be two gauges (see Def.~\ref{def:RCGN}).
Let $(a_{n})_{n}:\hyperN{\sigma}\ra\rcrho$ be a $\sigma$-hypersequence
of $\rho$-generalized numbers. Finally let $l\in\rcrho$. Then we
say that 
\[
l\text{ is the hyperlimit of }(a_{n})_{n}
\]
if 
\begin{equation}
\forall q\in\N\,\exists M\in\hyperN{\sigma}\,\forall n\in\hyperN{\sigma}:\ n\ge M\ \Rightarrow\ \left|a_{n}-l\right|<\diff{\rho}^{q}.\label{eq:hyperlimit}
\end{equation}
\end{defn}

\begin{rem}
\label{rem:hyperlim}~ 
\begin{enumerate}
\item In a hyperlimit, we are considering $\hyperN{\sigma}$ as an ordered
set directed by $\le$: 
\[
n,m\in\hyperN{\sigma}\ \Rightarrow\ n\vee m=\left[\max\left(\nint(n)_{\eps},\nint(m)_{\eps}\right)\right]\in\hyperN{\sigma}.
\]
On the other hand, on $\rcrho$ we are considering the sharp topology
(which is Hausdorff). In fact, if $l$, $\lambda$ are hyperlimits
of $a:\hyperN{\sigma}\ra\rcrho$, then 
\[
\left|l-\lambda\right|\le\left|l-a_{M}\right|+\left|a_{M}-\lambda\right|\le2\diff{\rho}^{q+1}<\diff{\rho}^{q}
\]
for all $q$. So $l=\lambda\in\rcrho$. We will therefore use the
notations 
\[
l=\hyperlim{\rho}{\sigma}a_{n}
\]
or simply $l=\hyperlimrho a_{n}$ if $\sigma=\rho$. 
\item \label{enu:extension}A sufficient condition to extend an ordinary
sequence $a:\N\ra\rcrho$ of $\rho$-generalized numbers to the whole
of $\hyperN{\sigma}$ is 
\begin{equation}
\forall n\in\hyperN{\sigma}:\ \left(a_{\nint(n)_{\eps}}\right)\in\R_{\rho}.\label{eq:ext_a}
\end{equation}
In fact, in this way $a_{n}$ is well-defined because of Lem.~\ref{lem:nearestInt};
on the other hand, using \eqref{eq:ext_a}, we have defined an extension
of the old sequence $a$ because if $n\in\N$, then $\nint(n)_{\eps}=n$
for $\eps$ small and hence we get $a_{n}=[a_{n}]$. For example,
the sequence of infinities $a_{n}=\frac{1}{n}+\diff{\rho}^{-1}$ for
all $n\in\N$ can be extended to any $\hyperN{\sigma}$, whereas $a_{n}=\diff{\sigma}^{-n}$
can be extended as $a:\hyperN{\sigma}\ra\rcrho$ only for certain
gauges $\rho$, e.g.~if the gauges satisfy 
\[
\exists N\in\N\,\forall n\in\N\,\forall^{0}\eps:\ \sigma_{\eps}^{n}\ge\rho_{\eps}^{N},
\]
e.g.~$\sigma_{\eps}\ge-\log(\rho_{\eps})^{-1}$. 
\end{enumerate}
\end{rem}

\begin{example}
\label{exa:hyperlimits}~ 
\begin{enumerate}
\item \label{enu:twoGauges}The following example strongly motivates the
use of two gauges. Let $\rho$ be a gauge and set $\sigma_{\eps}:=\exp\left(-\rho_{\eps}^{-\frac{1}{\rho_{\eps}}}\right)$,
so that also $\sigma$ is a gauge. We have 
\[
\hyperlim{\rho}{\sigma}\frac{1}{\log n}=0\in\rcrho\quad\text{whereas}\quad\not\exists\,\hyperlim{\rho}{\rho}\frac{1}{\log n}.
\]
In fact, if $n>1$, we have $0<\frac{1}{\log n}<\diff{\rho}^{q}$
if and only if $\log n>\diff{\rho}^{-q}$, i.e.~$n>e^{\diff{\rho}^{-q}}$
(in $\RC{\sigma}$). We can thus take $M:=\left[\text{int}\left(e^{\rho_{\eps}^{-q}}\right)+1\right]\in\hyperN{\sigma}$
because $e^{\rho_{\eps}^{-q}}<\exp\left(\rho_{\eps}^{-\frac{1}{\rho_{\eps}}}\right)=\sigma_{\eps}^{-1}$
for $\eps$ small.\\
 Vice versa, by contradiction, if $\exists\,\hyperlim{\rho}{\rho}\frac{1}{\log n}=:l\in\rcrho$,
then by the definition of hyperlimit from $\hyperNrho$ to $\rcrho$
we would get the existence of $M\in\hyperN\rho$ such that 
\begin{equation}
\forall n\in\hyperNrho:\ n\ge M\ \Rightarrow\ \frac{1}{\log n}-\diff{\rho}<l<\frac{1}{\log n}+\diff{\rho}.\label{eq:existHyperlimAbs}
\end{equation}
Since $M$ is $\rho$-moderate, we always have $0<\frac{1}{\log M}-\diff{\rho}$,
so $l>0$. Thus $\diff{\rho}^{p}<|l|$ for some $p\in\N$. Setting
\[
q:=\min\left\{ p\in\N\mid\diff{\rho}^{p}<|l|\right\} +1,
\]
we get that $\left|l_{\bar{\eps}_{k}}\right|<\rho_{\bar{\eps}_{k}}^{q}$
for some sequence $(\bar{\eps}_{k})_{k}\downarrow0$. Therefore 
\[
\frac{1}{\log M_{\bar{\eps}_{k}}}<l_{\bar{\eps}_{k}}+\rho_{\bar{\eps}_{k}}\le\left|l_{\bar{\eps}_{k}}\right|+\rho_{\bar{\eps}_{k}}<\rho_{\bar{\eps}_{k}}^{q}+\rho_{\bar{\eps}_{k}}
\]
and hence $M_{\bar{\eps}_{k}}>\exp\left(\frac{1}{\rho_{\bar{\eps}_{k}}^{q}+\rho_{\bar{\eps}_{k}}}\right)$
for all $k\in\N$, which is in contradiction with $M\in\rcrho$ because
$q\ge1$. 
\item For all $k\in\N_{>0}$, we have $\hyperlimrho\frac{1}{n^{k}}=0$.
In fact, for all $n\in\hyperNrho_{>0}$, we have $0<\frac{1}{n^{k}}<\diff{\rho}^{q}$
if and only if $n^{k}>\diff{\rho}^{-q}$, i.e.~$n>\diff{\rho}^{-\frac{q}{k}}$.
Thus, it suffices to take $M_{\eps}:=\text{int}\left(\rho_{\eps}^{-\frac{q}{k}}\right)+1$
in the definition of hyperlimit. Analogously, we can treat rational
functions having degree of denominator greater or equal to that of
the numerator. 
\end{enumerate}
\end{example}

\subsection{Properties of multidimensional integral}

We start by proving the change of variable formula. 
\begin{lem}
\label{lem:epsWiseInjectivity}Let $K=[K_{\eps}]$ be functionally
compact and $\phi=[\phi_{\eps}]\in\gsf(K,\RC{\rho}^{d})$ with $\det(d\phi)(x)$
invertible for each $x\in K$. If $\phi$ is injective on $K$, then
the $\phi_{\eps}$ are injective on $K_{\eps}$, $\forall^{0}\eps$. 
\end{lem}

\begin{proof}
By contradiction, suppose that for each $\eta>0$, there exists $\eps<\eta$
such that $\phi_{\eps}$ is not injective on $K_{\eps}$. Then we
find for such $\eps$ some $x_{\eps}$, $y_{\eps}\in K_{\eps}$ with
$x_{\eps}\ne y_{\eps}$ and $\phi_{\eps}(x_{\eps})=\phi_{\eps}(y_{\eps})$.
For all other $\eps$, define $x_{\eps}=y_{\eps}\in K_{\eps}$ arbitrary.
Then $x:=[x_{\eps}]$, $y:=[y_{\eps}]\in K$ and $\phi(x)=\phi(y)$.
As $\phi$ is injective, $x=y$. But then this contradicts the local
injectivity of $\phi_{\eps}$ on $\Eball_{r_{\eps}}(x_{\eps})$ for
some $[r_{\eps}]>0$, see also \cite[Thm.~6]{GK17} and \cite{EG:13}. 
\end{proof}
\begin{thm}
\label{thm:changeVarsMultid}Let $K\subseteq\RC{\rho}^{n}$ be $\lambda$-measurable,
where $\lambda$ is the Lebesgue measure, and let $\phi\in\gsf(K,\RC{\rho}^{d})$
be such that $\phi^{-1}\in\gsf(\phi(K),\RC{\rho}^{n})$. Then $\phi(K)$
is $\lambda$-measurable and 
\[
\int_{\phi(K)}f\,\diff{\lambda}=\int_{K}(f\circ\phi)\left|\det(\diff{\phi})\right|\,\diff{\lambda}
\]
for each $f\in\gsf(\phi(K),\RC{\rho})$. 
\end{thm}

\begin{proof}
Let $x\in\overline{\Eball}_{\rho_{\eps}^{m}}(K_{\eps})$. Then there
exists $y\in K_{\eps}$ such that $|x-y|\le\rho_{\eps}^{m}$. As $\phi\in\gsf(K,\RC{\rho}^{d})$,
\[
|\phi_{\eps}(x)-\phi_{\eps}(y)|\le|x-y|\sup_{\overline{\Eball}_{\rho_{\eps}^{m}}(K_{\eps})}\|\diff{\phi_{\eps}}\|\le\rho_{\eps}^{m-M}
\]
for some $M\in\N$ (not depending on $m$). Thus $\phi_{\eps}(\overline{\Eball}_{\rho_{\eps}^{m}}(K_{\eps}))\subseteq\overline{\Eball}_{\rho_{\eps}^{m-M}}(\phi_{\eps}(K_{\eps}))$.
Applying this to $\phi^{-1}$, we find that also $\overline{\Eball}_{\rho_{\eps}^{m}}(\phi_{\eps}(K_{\eps}))\subseteq\phi_{\eps}(\overline{\Eball}_{\rho_{\eps}^{m-M}}(K_{\eps}))$
for some $M\in\N$. Now let $f_{\eps}\ge0$. As $(\det(\diff{\phi_{\eps}^{-1}}))$
is moderate, $|\det(\diff{\phi_{\eps}})(x)|>0$ for each $x\in K$.
Thus by Lem.~\ref{lem:epsWiseInjectivity}, w.l.o.g~$\phi_{\eps}$
are injective. Then 
\[
\int_{\overline{\Eball}_{\rho_{\eps}^{m}}(\phi_{\eps}(K_{\eps}))}f_{\eps}\,\diff{\lambda}\le\int_{\phi_{\eps}(\overline{\Eball}_{\rho_{\eps}^{m-M}}(K_{\eps}))}f_{\eps}\,\diff{\lambda}=\int_{\overline{\Eball}_{\rho_{\eps}^{m-M}}(K_{\eps})}(f_{\eps}\circ\phi_{\eps})|\det\diff{\phi_{\eps}}|\,\diff{\lambda}
\]
and 
\[
\int_{\overline{\Eball}_{\rho_{\eps}^{m+M}}(K_{\eps})}(f_{\eps}\circ\phi_{\eps})|\det\diff{\phi_{\eps}}|\,\diff{\lambda}=\int_{\phi_{\eps}(\overline{\Eball}_{\rho_{\eps}^{m+M}}(K_{\eps}))}f_{\eps}\,\diff{\lambda}\le\int_{\overline{\Eball}_{\rho_{\eps}^{m}}(\phi_{\eps}(K_{\eps}))}f_{\eps}\,\diff{\lambda}
\]
Since $\lim_{m\to\infty}\bigl[\int_{\overline{\Eball}_{\rho_{\eps}^{m-M}}(K_{\eps})}(f_{\eps}\circ\phi_{\eps})|\det\diff{\phi_{\eps}}|\,\diff{\lambda}\bigr]$
exists, it follows from the previous inequalities that also $\lim_{m\to\infty}\bigl[\int_{\overline{\Eball}_{\rho_{\eps}^{m}}(\phi_{\eps}(K_{\eps}))}f_{\eps}\,\diff{\lambda}\bigr]$
exists, with the same value. The general case follows by considering
the positive and negative part of $f_{\eps}$. 
\end{proof}
We now consider the problem of additivity of the integral. 
\begin{defn}
\label{def:stronglyDisjoint}Let $K$, $L$ be functionally compact.
Then we call $K$ and $L$ \emph{strongly disjoint} if the following
equivalent conditions hold: 
\begin{enumerate}
\item for each representative $(K_{\eps})$ of $K$ and $(L_{\eps})$ of
$L$, $K_{\eps}\cap L_{\eps}=\emptyset$, $\forall^{0}\eps$ 
\item for some (and thus each) representative $(K_{\eps})$ of $K$ and
$(L_{\eps})$ of $L$, there exists $m\in\N$ such that $\overline{\Eball}_{\rho_{\eps}^{m}}(K_{\eps})\cap\overline{\Eball}_{\rho_{\eps}^{m}}(L_{\eps})=\emptyset$,
$\forall^{0}\eps$ 
\item $\forall e\in\RC{\rho}:\ e^{2}=e,e\ne0\ \Rightarrow\ Ke\cap Le=\emptyset$ 
\item $\forall H\subzero I:\ K|_{H}\cap L|_{H}=\emptyset$ 
\item $x\ne y$ for each subpoint $x$ of $K$ and $y$ of $L$. 
\end{enumerate}
\end{defn}

\ 
\begin{defn}
\label{def:almostStronglyDisjoint}Let $K$, $L$ be functionally
compact. Then we call $K$ and $L$ \emph{almost strongly disjoint}
if the following equivalent conditions hold: 
\begin{enumerate}
\item for each representative $(K_{\eps})$ of $K$ and $(L_{\eps})$ of
$L$, $[\mu(K_{\eps}\cap L_{\eps})]=0$ 
\item for some (and thus each) representative $(K_{\eps})$ of $K$ and
$(L_{\eps})$ of $L$ 
\[
\lim_{m\to\infty}[\mu(\overline{\Eball}_{\rho_{\eps}^{m}}(K_{\eps})\cap\overline{\Eball}_{\rho_{\eps}^{m}}(L_{\eps}))]=0.
\]
\end{enumerate}
\end{defn}

\noindent The equivalence of the conditions follows by a similar argument
as in Lem.~\ref{lem:altCharMeasurability}, e.g.: 
\begin{lem}
The conditions in Def.\ \ref{def:almostStronglyDisjoint} are equivalent. 
\end{lem}

\begin{proof}
$(i)\Rightarrow(ii)$: let $K=[K_{\eps}]$ and $L=[L_{\eps}]$. Seeking
a contradiction, suppose that there exists $q\in\N$ for which it
does not hold that 
\[
\exists M\,\forall m\ge M\,\forall^{0}\eps:\ \mu(\overline{\Eball}_{\rho_{\eps}^{m}}(K_{\eps})\cap\overline{\Eball}_{\rho_{\eps}^{m}}(L_{\eps}))\le\rho_{\eps}^{q}.
\]
Then we can construct a strictly increasing sequence $(m_{k})_{k}\to\infty$
and a strictly decreasing sequence $(\eps_{k})_{k}\to0$ s.t.\ $\mu(\overline{\Eball}_{\rho_{\eps_{k}}^{m_{k}}}(K_{\eps_{k}})\cap\overline{\Eball}_{\rho_{\eps_{k}}^{m_{k}}}(L_{\eps_{k}}))>\rho_{\eps_{k}}^{q}$,
$\forall k$.\\
 Let $K'_{\eps}:=\overline{\Eball}_{\rho_{\eps}^{m_{k}}}(K_{\eps})$
and $L'_{\eps}:=\overline{\Eball}_{\rho_{\eps}^{m_{k}}}(L_{\eps})$,
whenever $\eps\in(\eps_{k+1},\eps_{k}]$, $\forall k$. Then $K=[K'_{\eps}]$
and $L=[L'_{\eps}]$, but $[\mu(K'_{\eps}\cap L'_{\eps})]\ne0$, as
$\mu(K'_{\eps}\cap L'_{\eps})>\rho_{\eps}^{q}$, for each $\eps=\eps_{k}$
($k\in\N$).\\
 $(ii)\Rightarrow(i)$: let $(K_{\eps})$, $(L_{\eps})$ as in (ii),
and $K=[K'_{\eps}]$ and $L=[L'_{\eps}]$. For each $q\in\N$, we
have that 
\[
[\mu(\overline{\Eball}_{\rho_{\eps}^{m}}(K_{\eps})\cap\overline{\Eball}_{\rho_{\eps}^{m}}(L_{\eps}))]\le\rho^{q}
\]
for sufficiently large $m\in\N$. As $[K'_{\eps}]\subseteq[K_{\eps}]$,
$K'_{\eps}\subseteq\overline{\Eball}_{\rho_{\eps}^{m}}(K_{\eps})$,
$\forall^{0}\eps$, and similarly for $L$, and thus also $[\mu(K'_{\eps}\cap L'_{\eps})]\le\rho^{q}$. 
\end{proof}
\noindent E.g., if $a\le b\le c$, then $[a,b]$ and $[b,c]$ are
almost strongly disjoint. Obviously, strongly disjoint sets are almost
strongly disjoint. Recall that the union of two internal sets is usually
not internal, but 
\[
K\vee L:=[K_{\eps}\cup L_{\eps}]=\{e_{S}x+e_{\co S}y:x\in K,y\in L,S\subseteq\left]0,1\right]\}
\]
is the smallest internal set containing $K$ and $L$ \cite{Ob-Ve}.
E.g., $[a,b]\vee[b,c]=[a,c]$. 
\begin{thm}
If $K$, $L$ are almost strongly disjoint $\mu$-measurable subsets
of $\RC{\rho}^{n}$, then $K\vee L$ is also $\mu$-measurable and
for each $f\in\GI(K\vee L,\RC{\rho})$ 
\[
\int_{K\vee L}f\,\diff{\mu}=\int_{K}f\,\diff{\mu}+\int_{L}f\,\diff{\mu}.
\]
\end{thm}

\begin{proof}
As $\overline{\Eball}_{\rho_{\eps}^{m}}(K_{\eps}\cup L_{\eps})=\overline{\Eball}_{\rho_{\eps}^{m}}(K_{\eps})\cup\overline{\Eball}_{\rho_{\eps}^{m}}(L_{\eps})$,
\begin{align*}
\left[\int_{\overline{\Eball}_{\rho_{\eps}^{m}}(K_{\eps}\cup L_{\eps})}f\,\diff{\mu}\right] & =\left[\int_{\overline{\Eball}_{\rho_{\eps}^{m}}(K_{\eps})}f\,\diff{\mu}\right]+\left[\int_{\overline{\Eball}_{\rho_{\eps}^{m}}(L_{\eps})}f\,\diff{\mu}\right]-\\
 & \phantom{=}-\left[\int_{\overline{\Eball}_{\rho_{\eps}^{m}}(K_{\eps})\cap\overline{\Eball}_{\rho_{\eps}^{m}}(L_{\eps})}f\,\diff{\mu}\right].
\end{align*}
Since 
\begin{align*}
\left|\left[\int_{\overline{\Eball}_{\rho_{\eps}^{m}}(K_{\eps})\cap\overline{\Eball}_{\rho_{\eps}^{m}}(L_{\eps})}f\,\diff{\mu}\right]\right| & \le\left[\mu(\overline{\Eball}_{\rho_{\eps}^{m}}(K_{\eps})\cap\overline{\Eball}_{\rho_{\eps}^{m}}(L_{\eps}))\right]\cdot\\
 & \phantom{\le}\cdot\left[\sup_{\overline{\Eball}_{\rho_{\eps}^{m}}(K_{\eps})\cap\overline{\Eball}_{\rho_{\eps}^{m}}(L_{\eps})}|f|\right]\stackrel{m\to\infty}{\to}0
\end{align*}
we see that $K\vee L$ is $\mu$-measurable with $\int_{K\vee L}f\,\diff{\mu}=\int_{K}f\,\diff{\mu}+\int_{L}f\,\diff{\mu}$. 
\end{proof}
\begin{example}
Let $S\subseteq[0,1[$ with $0\in\overline{S}$ and $0\in\overline{\co S}$.
Let $K_{\eps}=\begin{cases}
[0,1], & \eps\in S\\{}
[0,3], & \eps\in\co S
\end{cases}$ and $L_{\eps}=\begin{cases}
[2,3], & \eps\in S\\{}
[0,3], & \eps\in\co S.
\end{cases}$ Then $K\cap L=\emptyset$ and $\mu(K\vee L)=2e_{S}+3e_{\co S}\ne2e_{S}+6e_{\co S}=\mu(K)+\mu(L)$.
Thus the condition that $K$ and $L$ are almost strongly disjoint
cannot be replaced by the condition that $K\cap L=\emptyset$. 
\end{example}

\begin{thm}
\label{thm:uniformConv} Let $K\fcmp\rcrho^{n}$. Let $f_{n}\in\GI(K,\rcrho^{d})$,
$\forall n\in\hyperN{\sigma}$. If $\hyperlim{\rho}{\sigma}f_{n}(x)$
exists for each $x\in K$, then the convergence is uniform over $K$
and the limit function is integrable on $K$. 
\end{thm}

\begin{proof}
We first show that the sequence is uniformly Cauchy, i.e.\ that for
each $m\in\N$ 
\begin{equation}
\exists N\in\N\,\forall k,l\in\hyperN{\sigma},k,l\ge\sigma^{-N}\,\forall x\in K:|f_{k}(x)-f_{l}(x)|\le\rho^{m}.\label{eq:UniformlyCauchy}
\end{equation}
Seeking a contradiction, suppose that for some $m\in\N$, we have
\[
\forall N\in\N\,\exists k,l\in\hyperN{\sigma},k,l\ge\sigma^{-N}\,\exists x\in K:\ |f_{k}(x)-f_{l}(x)|\not\le\rho^{m}.
\]
We thus construct sequences $(k_{N})_{N}$ and $(l_{N})_{N}$ in $\hyperN{\sigma}$,
with $k_{N},l_{N}\ge\sigma^{-N}$ for which there exist $x_{N}\in K$
s.t.\ $|f_{k_{N}}(x_{N})-f_{l_{N}}(x_{N})|\not\le\rho^{m}$, $\forall N\in\N$.
Let $K=[K_{\eps}]$ and $f_{k}=[f_{k,\eps}]$. We thus find $S_{N}\subseteq(0,1]$
with $0\in\overline{S_{N}}$ such that $|f_{k_{N},\eps}(x_{N,\eps})-f_{l_{N},\eps}(x_{N,\eps})|\ge\rho_{\eps}^{m}$
for each $\eps\in S_{N}$. Then choose a decreasing sequence $(\eps_{n})_{n}\to0$
such that 
\begin{align*}
 & \eps_{1}\in S_{1}\\
 & \eps_{2}\in S_{2};\eps_{3}\in S_{1};\\
 & \eps_{4}\in S_{3};\eps_{5}\in S_{2};\eps_{6}\in S_{1};\\
 & \dots
\end{align*}
Let $x_{\eps_{n}}:=x_{N,\eps_{n}}$ if $\eps_{n}\in S_{N}$, for each
$n\in\N$. Extend to a net $(x_{\eps})_{\eps}$ with $x_{\eps}\in K_{\eps}$,
$\forall\eps$. Then $xe_{U_{N}}=x_{N}e_{U_{N}}$ for some $U_{N}\subseteq S_{N}$
with $0\in\overline{U_{N}}$, and therefore $|f_{k_{N}}(x)-f_{l_{N}}(x)|e_{U_{N}}\ge\rho^{m}e_{U_{N}}$,
$\forall N$. We thus contradict the fact that $(f_{n}(x))_{n\in\hyperN{\sigma}}$
is a convergent hypersequence.\\
 By taking the limit for $l\to\infty$ in \eqref{eq:UniformlyCauchy},
we conclude that $(f_{n})_{n\in\hyperN{\sigma}}$ is uniformly convergent
on $K$.\\
 For each $n\in\N$, fix a representative $(f_{\sigma^{-n},\eps})_{\eps}$
of $f_{\sigma^{-n}}$ with $\sup_{x\in K_{\eps}}|f_{\sigma^{-n},\eps}|\le\eps^{-M}$,
$\forall\eps\in(0,1]$, $\forall n\in\N$ (some $M\in\N$, independent
of $n$). For each $m\in\N$, there exists $N_{m}\in\N$ such that
\[
\forall n,n'\ge N_{m}\,\forall^{0}\eps:\sup_{x\in K_{\eps}}|f_{\sigma^{-n},\eps}(x)-f_{\sigma^{-n'},\eps}(x)|\le\rho_{\eps}^{m}.
\]
Then for each $k\in\N$, there exists some $\eps_{k}>0$ such that
\[
\forall\eps\le\eps_{k}\,\forall m\le k\,\forall n,n'\in[N_{m},k]:\sup_{x\in K_{\eps}}|f_{\sigma^{-n},\eps}(x)-f_{\sigma^{-n'},\eps}(x)|\le\rho_{\eps}^{m}.
\]
W.l.o.g., $(\eps_{k})_{k}\downarrow0$. Let $n_{\eps}:=k$, for $\eps\in(\eps_{k+1},\eps_{k}]$.
Then 
\[
\forall m\in\N\,\forall n\in\N,n\ge N_{m}\,\forall^{0}\eps:\sup_{x\in K_{\eps}}|f_{\sigma^{-n},\eps}(x)-f_{\sigma^{-n_{\eps}},\eps}(x)|\le\rho_{\eps}^{m}.
\]
Thus the limit function $f=[f_{\sigma^{-n_{\eps}},\eps}]\in\GI(K,\RC\rho)$. 
\end{proof}
\begin{thm}
Let $K\subseteq\RC{\rho}^{n}$ be $\mu$-measurable and $f\in\gsf(K,\RC{\rho})$.
If $\hyperlim{\rho}{\sigma}f_{n}(x)$ exists for each $x\in K$, then
$\hyperlim{\rho}{\sigma}f_{n}$ is integrable on $K$ and 
\[
\hyperlim{\rho}{\sigma}\int_{K}f_{n}\,\diff{\mu}=\int_{K}\hyperlim{\rho}{\sigma}f_{n}\,\diff{\mu}.
\]
\end{thm}

\begin{proof}
By Thm.~\ref{thm:uniformConv}, $(f_{n})_{n}$ uniformly converges
to some $f=[f_{\eps}]$. Then 
\[
\left|\int_{K}f_{n}-\int_{K}f\right|
\le\int_{K}|f_{n}-f|\le\rho^{m}\mu(K)
\]
as soon as $n\in\hyperN{\sigma}$ is large enough. 
\end{proof}

\section{\label{sec:Sheaf-properties}Sheaf properties}

The aim of this section is to establish appropriate sheaf properties
for GSF. That this task is not entirely straightforward can be seen
from the following example, which can be easily reformulated in other
non-Archimedean settings: 
\begin{example}
\label{exa:i_revisited} Let $i:\rcrho\to\rcrho$ be as in Rem.~\ref{rem:I-function},
i.e., $i(x):=1$ if $x\approx0$ and $i(x):=0$ otherwise. The domain
$\rcrho$ of this function is the disjoint union of the sharply open
sets $D_{\infty}=\{x\in\rcrho\mid x\approx0\}$ and its complement
$D_{\infty}^{\text{c}}$. Moreover, $i|_{D_{\infty}}\equiv1$ and
$i|_{D_{\infty}^{\text{c}}}\equiv0$ are both GSF. However, as we
have seen in the remark following Cor.~\ref{cor:intermValue}, $i$
itself is \emph{not} a GSF. This shows that $^{\rho}\Gcinf$ is not
a sheaf with respect to the sharp topology. 
\end{example}

\noindent Trivially, if we introduce the space of (sharply) locally
defined GSF by means of $f\in\gsf_{\text{loc}}(X,Y)$ if $f:X\to Y$,
and $\forall x\in X\,\exists r\in\rcrho_{>0}:\ f|_{B_{r}(x)\cap X}\in\gsf(B_{r}(x)\cap X,Y)$,
then $\gsf_{\text{loc}}(-,Y)$ is naturally a sheaf with respect to
the sharp topology. By Example \ref{exa:i_revisited}, however, $\gsf_{\text{loc}}(X,Y)$
is strictly larger than $\gsf(X,Y)$. This fact can be viewed as a
necessary trade-off between the classical statement of locality for
generalized functions, on the one hand, and the requirement to preserve
classical theorems from smooth analysis on the other. In the above
example, it is the validity of an intermediate value theorem in our
setting that precludes the function $i$ from qualifying as a GSF.
Conversely, it follows that this result does not hold in $\gsf_{\text{loc}}(X,Y)$.
Any theory of generalized functions that is based on set-theoretical
functions and includes actual infinitesimals has to face these dichotomies
related to the total disconnectedness of its non-Archimedean ring
of scalars.

The general scheme of this section is: 
\begin{enumerate}[label=(\alph*)]
\item \label{enu:1GeneralSchemeSheaf}We are searching for a new compatibility/coherence
condition for an arbitrarily indexed family $\left(f_{j}\right)_{j\in J}$
of GSF (throughout this section, $J$ will be an arbitrary set), which
allows us to prove a corresponding sheaf property. We will call this
property \emph{dynamic compatibility condition} (DCC). 
\item \label{enu:2GeneralSchemeSheaf}The DCC must imply the classical one.
Note that for particular types of covers the classic coherence condition
may still work, e.g.~covers made of near-standard points and large
open sets (see Cor.~\ref{cor:sheafFermat} below) or those made of
increasing sequences of internal sets (see Thm\@.~\ref{thm:sheafIncreasingInternal}
below). 
\item \label{enu:3GeneralSchemeSheaf}The DCC must be a necessary condition
if we assume that the sections $\left(f_{j}\right)_{j\in J}$ glue
together into a GSF. 
\item \label{enu:4GeneralSchemeSheaf}The sheaf property based on the DCC
should be a particular case of the general abstract notion of sheaf
(see Sec.~\ref{sec:Gtopos}). 
\end{enumerate}
We start from the following sheaf property (originally proved in \cite{Ver08}): 
\begin{thm}
\label{thm:sheafIncreasingInternal}Let $X\subseteq\rcrho^{n}$, $Y\subseteq\rcrho^{d}$
and $K_{q}\fcmp\rcrho^{n}$, $f_{q}\in\gsf(K_{q},Y)$ for all $q\in\N$,
where $X=\bigcup_{q\in\N}K_{q}$, $K_{q}\subseteq\text{\emph{int}}\left(K_{q+1}\right)$
and $f_{q+1}|_{K_{q}}=f_{q}$ for each $q\in\N$. Then there exists
a unique $f\in\gsf(X,Y)$ such that $f|_{K_{q}}=f_{q}$ for all $q\in\N$. 
\end{thm}

\begin{proof}
Let $f_{q}=[f_{q,\eps}(-)]$ and $K_{q}=[K_{q,\eps}]$, for each $q\in\N$.
By Lem.~\ref{Lem:OmegaEpsAroundBoundedAndCptSupp}.\ref{enu:int},
there exist $k_{q}\in\N$ ($k_{q}$ recursively chosen so that $(k_{q})_{q}$
is increasing) such that $\Eball_{\rho_{\eps}^{k_{q}}}(K_{q,\eps})\subseteq K_{q+1,\eps}$,
for each $q$, $\eps$. We may assume $Y\subseteq\rcrho$ (in general,
one can apply the one-dimensional case componentwise). Let $\theta\in\cinfty(\R^{n})$
with $\theta(x)=0$, if $\abs x\ge1$ and $\theta(x)\ge0$, for each
$x\in\R^{n}$ with $\int_{\R^{n}}\theta=1$ and let $r\odot\theta_{r}(x):=r^{-n}\theta(\inv rx)$,
for $r\in\R_{>0}$. Let $1_{A}$ denote the characteristic function
of a set $A\subseteq\R^{n}$, and set 
\[
\phi_{q,\eps}:=1_{K_{q+3,\eps}\setminus K_{q,\eps}}*\rho_{\eps}^{k_{q+3}}\odot\theta,\quad\forall q,\eps.
\]
If $y\in\Eball_{\rho_{\eps}^{k_{q+3}}}(x)\cap K_{q,\eps}$, then $x\in\Eball_{\rho_{\eps}^{k_{q}}}(K_{q,\eps})\subseteq K_{q+1,\eps}$,
and hence $\Eball_{\rho_{\eps}^{k_{q+3}}}(x)\cap K_{q,\eps}=\emptyset$
if $x\notin K_{q+1,\eps}$. If $x\in K_{q+2,\eps}$, then $\Eball_{\rho_{\eps}^{k_{q+3}}}(x)\subseteq\Eball_{\rho_{\eps}^{k_{q+2}}}(K_{q+2,\eps})\subseteq K_{q+3,\eps}$.
Thereby, $\phi_{q,\eps}(x)=1$, for each $x\in K_{q+2,\eps}\setminus K_{q+1,\eps}$.
Moreover, $\supp\phi_{q,\eps}\subseteq K_{q+3,\eps}+\Eball_{\rho_{\eps}^{k_{q+3}}}(0)=\Eball_{\rho_{\eps}^{k_{q+3}}}(K_{q+3,\eps})\subseteq K_{q+4,\eps}$.
Further, $\sup_{x\in\R^{n}}\abs{\partial^{\alpha}\phi_{q,\eps}(x)}\le\rho_{\eps}^{-k_{q+3}\abs\alpha}\int_{\R^{n}}\abs{\partial^{\alpha}\theta}$
by the properties of the convolution. Let $\phi_{\eps}:=\sum_{q\in\N}\phi_{q,\eps}$.
Then $\phi_{\eps}\in\cinfty(\bigcup_{q\in\N}K_{q,\eps})$ and for
each $q$, $(\sup_{x\in K_{q,\eps}}\abs{\partial^{\alpha}\phi_{\eps}(x)})\in\R_{\rho}$.
Also $\phi_{\eps}(x)\ge1$, for each $x\in\bigcup_{q\in\N}K_{q,\eps}$.
Let $\psi_{q,\eps}:=\phi_{q,\eps}/\phi_{\eps}\in\cinfty(\R^{n})$.
Then $\sum_{q\in\N}\psi_{q,\eps}(x)=1$, for each $x\in\bigcup_{q\in\N}K_{q,\eps}$.
Since $\sup_{x\in K_{q,\eps}}\abs{1/\phi_{\eps}(x)}\le1$, we find
that $(\sup_{x\in\R^{n}}\abs{\partial^{\alpha}\psi_{q,\eps}(x)})\in\R_{\rho}$,
for each $q$. Let $f_{\eps}:=\sum_{q\in\N}\psi_{q,\eps}\cdot f_{q+3,\eps}\in\cinfty(\bigcup_{q\in\N}K_{q,\eps})$,
for each $\eps$ (recall that $\supp\psi_{q,\eps}\subseteq K_{q+3,\eps}$).
Then for each $N\in\N$, $\alpha\in\N^{d}$ and $x=[x_{\eps}]\in K_{N}$
(without loss of generality, $x_{\eps}\in K_{N,\eps}$, for each $\eps$),
\[
\abs{\partial^{\alpha}f_{\eps}(x_{\eps})}\le\sum_{q\le N+3}\abs{\partial^{\alpha}(\psi_{q,\eps}\cdot f_{q+3,\eps})(x_{\eps})}\in\R_{\rho}
\]
and 
\[
\abs{f_{\eps}(x_{\eps})-f_{N,\eps}(x_{\eps})}\le\sum_{q\le N+3}\abs{\psi_{q,\eps}(x_{\eps})}\abs{f_{q+3,\eps}(x_{\eps})-f_{N,\eps}(x_{\eps})}\sim_{\rho}0
\]
since $\supp\psi_{q,\eps}\subseteq K_{q+3,\eps}$. Hence $f\in\gsf(X,Y)$
with $\restr{f}{K_{N}}=f_{N}$. 
\end{proof}
\noindent This property allows us to firstly prove a sheaf property
for $K=[K_{\eps}]\fcmp\rcrho^{n}$ and secondly to use Thm.~\ref{thm:sheafIncreasingInternal}
to extend it to domains $X$ that satisfy the previous assumptions,
i.e.~the following 
\begin{defn}
\label{def:fcnlCmptExh}Let $X\subseteq\rcrho^{n}$, then we say that
$X$ \emph{admits a functionally compact exhaustion} if there exists
a sequence $(K_{q})_{q\in\N}$ such that 
\begin{enumerate}
\item $K_{q}\fcmp\rcrho^{n}$; 
\item $K_{q}\subseteq\text{int}\left(K_{q+1}\right)$; 
\item $X=\bigcup_{q\in\N}K_{q}$. 
\end{enumerate}
\end{defn}

\noindent For example, every strongly internal set $X=\sint{A_{\eps}}$
admits a functionally compact exhaustion since we can consider 
\begin{align}
K_{q\eps} & :=\overline{\Eball}_{\rho_{\eps}^{-q}}(0)\cap\overline{\Eball}_{-\rho_{\eps}^{q}}(A_{\eps})\qquad\forall q\in\N\nonumber \\
K_{q} & :=\left[{K_{q\eps}}\right]\fcmp X\label{eq:K_q-1}\\
X & =\bigcup_{q\in\N}K_{q}\nonumber 
\end{align}
where $\overline{\Eball}_{-r}(A):=\{x\in A\mid d(x,A^{c})\ge r\}$.
Other simple examples are e.g.~the intervals $(0,a]$ or $(-\infty,a]$.

\subsection{The Lebesgue generalized number}

We first introduce a notation for a specified Lebesgue number: 
\begin{lem}
\label{lem:Lebnum}Let $K\Subset\R^{n}$ and $(V_{j})_{j\in J}$ be
an open cover of $K$. For $x\in K$, set 
\begin{align}
\sigma(x) & :=\sup\{r\in\R_{>0}\mid\exists j\in J:\ \Eball_{r}(x)\subseteq V_{j}\}\label{eq:LebNumSup}\\
\sigma & :=\frac{1}{2}\min\{\sigma(x)\mid x\in K\};\text{ if }K=\emptyset,\text{ set }\sigma:=1.\nonumber 
\end{align}
Then $\sigma(-):K\ra\R_{>0}$ is a continuous function and $\sigma$
is a Lebesgue number of $(V_{j})_{j\in J}$ for $K$, i.e. 
\begin{equation}
\forall x\in K\,\exists j\in J:\ \Eball_{\sigma}(x)\subseteq V_{j}.\label{eq:propLebnum}
\end{equation}
We use the notation $\text{\emph{Lebnum}}\left((V_{j})_{j\in J},K\right)=:\sigma$. 
\end{lem}

\begin{proof}
See the proof of \cite[Thm.~1.6.11]{Bu-Bu-Iv01}. 
\end{proof}
\begin{lem}
\label{lem:genLebNum}Let $K=[K_{\eps}]\fcmp\rcrho^{n}$ with $K_{\eps}\Subset\R^{n}$
for all $\eps$. Assume that 
\begin{equation}
K\subseteq\bigcup_{j\in J}U_{j},\label{eq:coverK}
\end{equation}
where $U_{j}=\sint{U_{j\eps}}$ are strongly internal sets, and set
\begin{align*}
s_{\eps} & :=\text{\emph{Lebnum}}\left(\left(U_{j\eps}\right)_{j\in J},K_{\eps}\right)\\
s & :=[s_{\eps}]\in\rcrho_{\ge0}.
\end{align*}
Then $s>0$. 
\end{lem}

\begin{proof}
By contradiction, assume that $s\not>0$, so that $s_{\eps_{k}}<\rho_{\eps_{k}}^{k}$
for all $k\in\N$ and for some sequence $(\eps_{k})_{k\in\N}\downarrow0$.
By definition of $\text{Lebnum}$ we can write $s_{\eps_{k}}=\frac{1}{2}\sigma(x_{\eps_{k}})<\rho_{\eps_{k}}^{k}$
for some $x_{\eps_{k}}\in K_{\eps_{k}}\Subset\R^{n}$. By \eqref{eq:LebNumSup}
we hence have 
\begin{equation}
\forall k\in\N\,\forall j\in J:\ \Eball_{2\rho_{\eps_{k}}^{k}}(c_{\eps_{k}})\not\subseteq U_{j\eps_{k}}.\label{eq:notIncl}
\end{equation}
Note that the conclusion is trivial if $K_{\eps}=\emptyset$ for $\eps$
small. We can therefore assume that there exists some $h_{\eps}\in K_{\eps}$
for all $\eps$. Let $x_{\eps}:=x_{\eps_{k}}$ if $\eps=\eps_{k}$
and $x_{\eps}:=h_{\eps}$ otherwise, so that $x_{\eps}\in K_{\eps}\subseteq\bigcup_{j\in J}U_{j\eps}$
for small $\eps$ (see Lem.~\ref{Lem:OmegaEpsAroundBoundedAndCptSupp}.\ref{enu:1.10}
and recall that $K\subseteq\bigcup_{j\in J}U_{j}\subseteq\sint{\bigcup_{j\in J}U_{j\eps}}$).
Thereby, $x:=[x_{\eps}]\in K\subseteq\bigcup_{j\in J}U_{j}$. So,
$x\in U_{j}$ for some $j\in J$, and hence $\overline{B_{R}(x)}\subseteq U_{j}=\sint{U_{j\eps}}$
for some $R\in\rcrho_{>0}$, so that $\Eball_{R_{\eps_{k}}}(x_{\eps_{k}})\subseteq U_{j\eps_{k}}$
for $k\in\N$ sufficiently large by Lem.~\ref{Lem:OmegaEpsAroundBoundedAndCptSupp}.\ref{enu:1.10}
and \eqref{eq:closureBall}. Since also $2\rho_{\eps_{k}}^{k}<R_{\eps_{k}}$
for $k\in\N$ sufficiently large, we can finally say that $\Eball_{2\rho_{\eps_{k}}^{k}}(x_{\eps_{k}})\subseteq\Eball_{R_{\eps_{k}}}(x_{\eps_{k}})\subseteq U_{j\eps_{k}}$,
which contradicts \eqref{eq:notIncl}. 
\end{proof}
On the basis of this result, we can set 
\[
\text{Lebnum}\left(\left(U_{j}\right)_{j\in J},K\right)=:s\in\rcrho_{>0}.
\]
Assumption \eqref{eq:coverK} cannot be replaced by the weaker $K\subseteq\sint{\bigcup_{j\in J}U_{j\eps}}$:
let $K_{\eps}:=[-1,1]_{\R}$, $J:=\{1,2\}$, $c_{1\eps}:=-1$, $c_{2\eps}:=1$,
$r_{1\eps}:=1+e^{-1/\eps}=:r_{2\eps}$, then 
\[
[K_{\eps}]\subseteq\sint{\Eball_{r_{1\eps}}(c_{1\eps})\cup\Eball_{r_{2\eps}}(c_{2\eps})}=\sint{(-2-e^{-1/\eps},2+e^{-1/\eps})_{\R}}
\]
but $s_{\eps}=\text{Lebnum}((U_{j\eps})_{j\in J},K_{\eps})\le e^{-1/\eps}$
because for $x=0$ the largest ball contained in a set of the covering
is $\Eball_{e^{-1/\eps}}(0)$.

\subsection{The dynamic compatibility condition}
\begin{defn}
\label{def:DCC}\ 
\begin{enumerate}
\item \label{enu:U_j}Let $[J]:=\{(j_{\eps})\mid j_{\eps}\in J,\forall\eps\}=J^{I}$. 
\item Let $X\subseteq\rcrho^{n}$, $Y\subseteq\rcrho^{d}$ and $f\in\Set(X,Y)$.
Let $K\fcmp X\subseteq\bigcup_{j\in J}U_{j}$, and assume that for
all $j\in J$ we have $f_{j}:=f|_{U_{j}\cap X}\in\gsf(U_{j}\cap X,Y)$.
Then we say that $\left(f_{j}\right)_{j\in J}$ \emph{satisfies the
dynamic compatibility condition (DCC) on the cover $\left(K\cap U_{j}\right)_{j\in J}$}
if for all $j\in J$ there exist nets $\left(f_{j\eps}\right)$ defining
$f_{j}$, for each $j\in J$, such that setting $U_{\bar{\jmath}}:=\sint{U_{j_{\eps},\eps}}$,
we have: 
\begin{enumerate}[label=(\alph*)]
\item \label{enu:DCCmoder} $\forall\bar{\jmath}=(j_{\eps})\in[J]\,\forall[x_{\eps}]\in U_{\bar{\jmath}}\cap K\,\forall\alpha\in\N^{n}:\ \left(\partial^{\alpha}f_{j_{\eps},\eps}(x_{\eps})\right)\in\R_{\rho}^{d}$. 
\item \label{enu:DCCequality} $\forall\bar{\jmath}=(j_{\eps}),\bar{h}=(h_{\eps})\in[J]\,\forall[x_{\eps}]\in K\cap U_{\bar{\jmath}}\cap U_{\bar{h}}:\ \left[f_{j_{\eps},\eps}(x_{\eps})\right]=\left[f_{h_{\eps},\eps}(x_{\eps})\right]$. 
\end{enumerate}
\end{enumerate}
\noindent Finally, we say that $\left(f_{j}\right)_{j\in J}$ \emph{satisfies
the DCC on the cover $\left(U_{j}\right)_{j\in J}$} if it satisfies
the DCC on each functionally compact set contained in $X$. The adjective
\emph{dynamic} underscores that we are considering $\eps$-depending
indices $\bar{\jmath}=(j_{\eps})\in[J]$. 
\end{defn}

\begin{rem}
\noindent \label{rem:SCC}~ 
\begin{enumerate}
\item Taking constant $\bar{\jmath}$ and $\bar{h}$ in Def.~\ref{def:DCC}.\ref{enu:DCCequality},
we have that DCC is stronger than the classical compatibility condition
for $\left(f_{j}\right)_{j\in J}$ on $K\fcmp X$. 
\item DCC is a necessary condition if the sections $\left(f_{j}\right)_{j\in J}$
glue into a GSF $f=[f_{\eps}(-)]$ because in this case we can take
$f_{j\eps}=f_{\eps}$ for all $j\in J$. 
\item The notation $U_{\bar{\jmath}}:=\sint{U_{j_{\eps},\eps}}$ used to
state the DCC was introduced merely for simplicity of notations. In
fact, in general it is not possible to prove the independence from
the representative net $(U_{j\eps})$: if $U_{j}=\sint{V_{j\eps}}$,
we can have $d_{\text{H}}(U_{j\eps}^{c},V_{j\eps}^{c})=\rho_{\eps}^{j/\eps}$,
but taking $\bar{\jmath}=(\eps)$ we would have $d_{\text{H}}(U_{\eps,\eps}^{c},V_{\eps,\eps}^{c})=\rho_{\eps}$
and hence $U_{\bar{\jmath}}=\sint{U_{\eps,\eps}}\ne V_{\bar{\jmath}}=\sint{V_{\eps,\eps}}$. 
\end{enumerate}
\end{rem}

In the next and final subsection we prove that the DCC implies the
sheaf property.

\subsection{Proof of the sheaf property}
\begin{lem}
\label{lem:sheaf} Let $K$, $K^{+}$ be functionally compact sets
with $K\subseteq\text{int}(K^{+})\subseteq\rcrho^{n}$. Let $Y\subseteq\rcrho^{d}$
and $f\in\Set(K^{+},Y)$. Let $K^{+}\subseteq\bigcup_{j\in J}U_{j}$,
where $U_{j}=\sint{U_{j\eps}}$ are strongly internal sets and for
all $j\in J$ we have $f_{j}:=f|_{U_{j}\cap K^{+}}\in\gsf(U_{j}\cap K^{+},Y)$.
Let $K^{+}=[K_{\eps}^{+}]$. Assume that, for some representatives
$(f_{j\eps})_{\substack{j\in J\\
\eps\in I
}
}$ of $\left(f_{j}\right)_{j\in J}$ we have 
\begin{equation}
\forall m\in\N\,\forall^{0}\eps\,\forall j,k\in J:\ \sup_{K_{\eps}^{+}\cap U_{j\eps}\cap U_{k\eps}}|f_{k\eps}-f_{j\eps}|\le\rho_{\eps}^{m}.\label{eq:GCC}
\end{equation}
Then $f\in\gsf(K,Y)$. 
\end{lem}

\begin{proof}
Let $K=[K_{\eps}]$. By changing the representative $(K_{\eps}^{+})_{\eps}$
of $K^{+}$, we may assume that there exists $S\in\N$ such that $B_{\rho_{\eps}^{S}}(K_{\eps})\subseteq K_{\eps}^{+}$,
$\forall\eps$ \cite[Lemma 3.12]{Ver08}. We have that $K_{\eps}^{+}\subseteq\bigcup_{j\in J}U_{j,\eps}$,
$\forall\eps\le\eps_{0}$, for some $\eps_{0}>0$ (proof by contradiction).
Let $\eps\in(0,\eps_{0}]$. By compactness of $K_{\eps}^{+}$, $K_{\eps}^{+}\subseteq U_{j_{1},\eps}\cup\dots\cup U_{j_{l_{\eps}},\eps}$
for some $l_{\eps}\in\N$. Call $V_{k}:=U_{j_{k},\eps}\setminus(U_{j_{1},\eps}\cup\dots\cup U_{j_{k-1},\eps})$,
and let 
\[
f_{\eps}:=\sum_{k=1}^{l_{\eps}}f_{j_{k},\eps}1_{V_{k}}.
\]
Then $f_{\eps}$ is locally integrable. Let $m\in\N$. By \eqref{eq:GCC},
we find $\eps_{m}>0$ such that 
\[
\forall j,k\in J\,\forall\eps\le\eps_{m}:\ \sup_{K_{\eps}^{+}\cap U_{j,\eps}\cap U_{k,\eps}}|f_{k,\eps}-f_{j,\eps}|\le\rho_{\eps}^{m}.
\]
Further, by the definition of $f_{\eps}$, we then also have 
\[
\forall j\in J\,\forall\eps\le\eps_{m}:\ \sup_{K_{\eps}^{+}\cap U_{j,\eps}}|f_{\eps}-f_{j,\eps}|\le\rho_{\eps}^{m}.
\]
W.l.o.g., $\eps_{m}\downarrow0$. Let $q_{\eps}:=q$, for each $\eps\in(\eps_{(q+1)^{2}},\eps_{q^{2}}]$
($q\in\N$). Then 
\[
\forall j\in J\,\forall\eps\le\eps_{0}:\ \sup_{K_{\eps}^{+}\cap U_{j,\eps}}|f_{\eps}-f_{j,\eps}|\le\rho_{\eps}^{q_{\eps}^{2}}.
\]
Let $b$ be a smooth map $\R^{n}\to\R$ with $\int b=1$ and $\supp(b)\subseteq B_{1}(0)$,
and let 
\[
\delta_{q,\eps}(x):=\rho_{\eps}^{-nq}b\left(\frac{x}{\rho_{\eps}^{q}}\right).
\]
We show that $(f_{\eps}*\delta_{q_{\eps},\eps})_{\eps}$ is a representative
of $f$, thereby proving that $f\in\gsf(K,\rti^{d})$. We therefore
take $x=[x_{\eps}]\in K$, and we prove that 
\begin{enumerate}
\item \label{enu:derConv}$(\partial^{\alpha}(f_{\eps}*\delta_{q_{\eps},\eps})(x_{\eps}))\in\R_{\rho}^{d}$,
$\forall\alpha\in\N^{n}$ 
\item \label{enu:eqConv}$[(f_{\eps}*\delta_{q_{\eps},\eps})(x_{\eps})]=f(x)$. 
\end{enumerate}
Let $j\in J$ such that $x\in U_{j}$. Then there exists $S\in\N$
such that $B_{\rho_{\eps}^{S}}(x_{\eps})\subseteq K_{\eps}^{+}\cap U_{j,\eps}$,
$\forall^{0}\eps$. To prove \ref{enu:derConv} and \ref{enu:eqConv},
it suffices to see that $(\partial^{\alpha}(f_{\eps}*\delta_{q_{\eps},\eps})(x_{\eps})-\partial^{\alpha}f_{j,\eps}(x_{\eps}))_{\eps}$
is negligible for each $\alpha\in\N^{n}$.

Let $\alpha\in\N^{n}$. Then 
\begin{align*}
 & \left|\partial^{\alpha}(f_{\eps}*\delta_{q_{\eps},\eps})(x_{\eps})-\partial^{\alpha}(f_{j,\eps}*\delta_{q_{\eps},\eps})(x_{\eps})\right|=\left|((f_{\eps}-f_{j,\eps})*\partial^{\alpha}\delta_{q_{\eps},\eps})(x_{\eps})\right|\\
 & =\rho_{\eps}^{-|\alpha|q_{\eps}}\left|\int(f_{\eps}-f_{j,\eps})(x_{\eps}-\rho_{\eps}^{q_{\eps}}u)\partial^{\alpha}b(u)\,du\right|\le C_{\alpha}\rho_{\eps}^{-|\alpha|q_{\eps}}\sup_{B_{\rho_{\eps}^{q_{\eps}}}(x_{\eps})}|f_{\eps}-f_{j,\eps}|\\
 & \le C_{\alpha}\rho_{\eps}^{(q_{\eps}-|\alpha|)q_{\eps}}
\end{align*}
and 
\begin{align*}
 & \left|\partial^{\alpha}(f_{j,\eps}*\delta_{q_{\eps},\eps})(x_{\eps})-\partial^{\alpha}f_{j,\eps}(x_{\eps})\right|=\left|\int(\partial^{\alpha}f_{j,\eps}(x_{\eps}-y)-\partial^{\alpha}f_{j,\eps}(x_{\eps}))\delta_{q_{\eps},\eps}(y)\,dy\right|\\
 & \le C\sup_{u\in B_{\rho_{\eps}^{q_{\eps}}}(x_{\eps})}|\partial^{\alpha}f_{j,\eps}(u)-\partial^{\alpha}f_{j,\eps}(x_{\eps})|\le C\rho_{\eps}^{q_{\eps}}\max_{|\beta|=|\alpha|+1}\sup_{u\in B_{\rho_{\eps}^{q_{\eps}}}(x_{\eps})}|\partial^{\beta}f_{j,\eps}(u)|\\
 & \le C\rho_{\eps}^{q_{\eps}-N_{j,\alpha}}
\end{align*}
for sufficiently small $\eps$, by Thm.\ \ref{thm:GSF-continuity}~\ref{enu:modOnEpsDepBall}.
Combining both inequalities, we conclude that $(\partial^{\alpha}(f_{\eps}*\delta_{q_{\eps},\eps})(x_{\eps})-\partial^{\alpha}f_{j,\eps}(x_{\eps}))_{\eps}$
is $\rho$-negligible. 
\end{proof}
\begin{lem}
\label{lem:sheafFcmp} Let $K$, $K^{+}$ be functionally compact
sets with $K\subseteq\text{int}(K^{+})\subseteq\rcrho^{n}$. Let $Y\subseteq\rcrho^{d}$
and $f\in\Set(K^{+},Y)$. Let $K^{+}\subseteq\bigcup_{j\in J}U_{j}$,
where for all $j\in J$ we have $f_{j}:=f|_{U_{j}\cap K^{+}}\in\gsf(U_{j}\cap K^{+},Y)$
and $U_{j}$ is a strongly internal set. Assume that $\left(f_{j}\right)_{j\in J}$
satisfies the DCC on\emph{ the cover $\left(K^{+}\cap U_{j}\right)_{j\in J}$.}
Then $f\in\gsf(K,Y)$. 
\end{lem}

\begin{proof}
Let $s_{\eps}:=\frac{1}{2}\text{Lebnum}\left(\left(U_{j\eps}\right)_{j\in J},K_{\eps}^{+}\right)$.
By Lemma \ref{lem:genLebNum}, $s:=[s_{\eps}]>0$.\\
 We define a cover $(V_{x})_{x\in K^{+}}$ of $K^{+}$ as follows.
Let $x=[x_{\eps}]\in K^{+}$, where $x_{\eps}\in K_{\eps}^{+}$, $\forall\eps$.
Then there exist $j_{\eps}\in J$ such that $B_{2s_{\eps}}^{E}(x_{\eps})\subseteq U_{j_{\eps},\eps}$,
$\forall\eps$. We then define $V_{x}:=\sint{\Eball_{s_{\eps}}(x_{\eps})}$,
and we denote $\bar{\jmath}(x):=[j_{\eps}]\in[J]$. By condition \ref{enu:DCCmoder},
$(f_{j_{\eps},\eps})$ defines a generalized smooth map $f_{\bar{\jmath}(x)}\in\gsf(V_{x}\cap K^{+},Y)$.\\
 In order to conclude that $f\in\gsf(K,Y)$ by Lemma \ref{lem:sheaf},
it suffices to show that: 
\begin{enumerate}
\item \label{enu:i}$f_{\bar{\jmath}(x)}=f|_{V_{x}\cap K^{+}}$, $\forall x\in K^{+}$. 
\item \label{enu:ii}$\forall m\in\N\,\forall^{0}\eps\,\forall x,y\in K^{+}:\ \sup_{K_{\eps}^{+}\cap V_{x,\eps}\cap V_{y,\eps}}|f_{\bar{\jmath}(y),\eps}-f_{\bar{\jmath}(x),\eps}|\le\rho_{\eps}^{m}$. 
\end{enumerate}
Proof of \ref{enu:i}: Let $x\in K^{+}$. Let $y\in V_{x}\cap K^{+}$.
We want to show that $f_{\bar{\jmath}(x)}(y)=f(y)$. As $y\in K^{+}$,
we have $y\in U_{j}$ for some $j\in J$. Thus $y\in K^{+}\cap U_{j}\cap V_{x}\subseteq K^{+}\cap U_{j}\cap U_{\bar{\jmath}(x)}$,
and condition \ref{enu:DCCequality} yields $f(y)=f_{j}(y)=f_{\bar{\jmath}(x)}(y)$.\\
 Proof of \ref{enu:ii}: By contradiction, suppose that there exists
$m\in\N$ and, for each $n\in\N$, there exist $\eps_{n}>0$ with
$(\eps_{n})_{n}$ decreasingly tending to $0$ and $x_{n},y_{n}\in K^{+}$
and $z_{\eps_{n}}\in K_{\eps_{n}}^{+}\cap V_{x_{n},\eps_{n}}\cap V_{y_{n},\eps_{n}}$
such that 
\begin{equation}
|f_{j_{\eps_{n}},\eps_{n}}(z_{\eps_{n}})-f_{h_{\eps_{n}},\eps_{n}}(z_{\eps_{n}})|>\rho_{\eps_{n}}^{m}\label{eq:SCC-contradiction}
\end{equation}
where we denote $j_{\eps_{n}}:=\bar{\jmath}(x_{n})_{\eps_{n}}$ and
$h_{\eps_{n}}:=\bar{\jmath}(y_{n})_{\eps_{n}}\in J$. Then there exist
$x_{\eps_{n}},y_{\eps_{n}}\in K_{\eps_{n}}^{+}$ such that $V_{x_{n},\eps_{n}}=B_{s_{\eps_{n}}}^{E}(x_{\eps_{n}})$
and $V_{y_{n},\eps_{n}}=B_{s_{\eps_{n}}}^{E}(y_{\eps_{n}})$. For
$\eps\notin\{\eps_{n}:n\in\N\}$, let $z_{\eps}\in K_{\eps}^{+}$
be arbitrary. Then $z\in K^{+}$. Let $j_{\eps}:=h_{\eps}:=\bar{\jmath}(z)_{\eps}$,
for $\eps\notin\{\eps_{n}:n\in\N\}$. Let $(z'_{\eps})$ be any representative
of $z$. Then 
\[
\begin{cases}
z'_{\eps}\in B_{s_{\eps}}^{E}(z_{\eps})\subseteq U_{j_{\eps},\eps}=U_{h_{\eps},\eps} & \forall^{0}\eps\notin\{\eps_{n}:n\in\N\}\\
z'_{\eps_{n}}\in B_{s_{\eps_{n}}}^{E}(z_{\eps_{n}})\subseteq B_{2s_{\eps_{n}}}^{E}(x_{\eps_{n}})\subseteq U_{j_{\eps_{n}},\eps_{n}} & \forall n\in\N\text{ large enough}
\end{cases}
\]
and similarly $z'_{\eps_{n}}\in U_{h_{\eps_{n}},\eps_{n}}$ for large
enough $n\in\N$. Thus $z\in K^{+}\cap U_{\bar{\jmath}}\cap U_{\bar{h}}$,
and condition \ref{enu:DCCequality} contradicts \eqref{eq:SCC-contradiction}. 
\end{proof}
Using Thm.~\ref{thm:sheafIncreasingInternal} and a functionally
compact exhaustion, see Def.~\ref{def:fcnlCmptExh}, we get 
\begin{thm}
\label{thm:sheafFcnlCmpExh} Let $X\subseteq\rcrho^{n}$ be a set
that admits a functionally compact exhaustion, $Y\subseteq\rcrho^{d}$
and $f\in\Set(X,Y)$. Let $X\subseteq\bigcup_{j\in J}U_{j}$, where
for all $j\in J$ we have $f_{j}:=f|_{U_{j}\cap X}\in\gsf(U_{j}\cap X,Y)$
and $U_{j}$ is a strongly internal set. Assume that $\left(f_{j}\right)_{j\in J}$
satisfies the DCC on\emph{ the cover $\left(U_{j}\right)_{j\in J}$.}
Then $f\in\gsf(X,Y)$. 
\end{thm}

The usual sheaf properties both for Schwartz distributions and for
Colombeau generalized functions do not need any stronger compatibility
condition, which ultimately stems from the possibility to use, for
these generalized functions, only (near-)standard points (recall Thm\@.~\ref{thm:nearStdInfEquality}
and Thm.~\ref{thm:nearStdInfModerate}). This is proved in the following
result, which generalizes the aforementioned sheaf properties (see
e.g.~\cite{GKOS}). 
\begin{thm}
\noindent \label{thm:sheafFermatBalls}Let $X\subseteq\nrst{\left(\rti^{n}\right)}$,
$Y\subseteq\rcrho^{d}$ and let $f:X\ra Y$ be a set-theoretical map.
Suppose that $X\subseteq\bigcup_{x\in X}B_{r_{x}}(x)$, where $r_{x}\in\R_{>0}$
for all $x$, and that 
\[
f|_{B_{r_{x}}(x)\cap X}\in\gsf(B_{r_{x}}(x)\cap X,Y)
\]
for all $x\in X$. Then $f\in\gsf(X,Y)$. 
\end{thm}

\begin{proof}
For every $x\in X$, let $f|_{B_{r_{x}}(x)}=:v^{x}\in\gsf(B_{r_{x}}(x)\cap X)$
and let $v^{x}$ be defined by the net $(v_{\eps}^{x})$ with $v_{\eps}^{x}\in{\mathcal{C}}^{\infty}(\R^{n},\R^{d})$.
Recall that by $x^{\circ}\in\R^{n}$ we denote the standard part of
any $x\in(\rcrho^{n})^{\bullet}$. Pick a countable, locally finite
open (in $\R^{n}$) refinement $(U_{i})_{i\in\N}$ of $(\Eball{}_{r_{x}/2}(x^{\circ}))_{x\in X}$
and let $(\chi_{i})_{i\in\N}$ be a partition of unity with $\supp\chi_{i}\comp U_{i}$
for all $i\in\N$. For any $i\in\N$ pick $x_{i}\in X$ such that
$U_{i}\subseteq\Eball_{r_{x_{i}}/2}(x_{i}^{\circ})$ and set 
\[
f_{\eps}:=\sum_{i\in\N}\chi_{i}v_{\eps}^{x_{i}}\in\cinfty(\R^{n},\R^{d}).
\]
Then the net $(f_{\eps})$ defines a GSF of the type $X\ra Y$: indeed,
we will show that $f(z)=v^{z}(z)=[f_{\eps}(z_{\eps})]$ for all $z=[z_{\eps}]\in X$:

\begin{align*}
f_{\eps}(z_{\eps})-v_{\eps}^{z}(z_{\eps}) & =\sum_{i\in\N}\chi_{i}(z_{\eps})(v_{\eps}^{x_{i}}(z_{\eps})-v_{\eps}^{z}(z_{\eps}))=\\
 & =\sum_{\{i\mid z^{\circ}\in B_{3r_{x_{i}}/4}(x_{i})\}}\chi_{i}(z_{\eps})(v_{\eps}^{x_{i}}(z_{\eps})-v_{\eps}^{z}(z_{\eps}))\ +\\
 & \phantom{=}+\sum_{\{i\mid z^{\circ}\not\in B_{3r_{x_{i}}/4}(x_{i})\}}\chi_{i}(z_{\eps})(v_{\eps}^{x_{i}}(z_{\eps})-v_{\eps}^{z}(z_{\eps}))=\\
 & =:A_{\eps}+B_{\eps}.
\end{align*}
Since $z_{\eps}\to z^{\circ}$, for small $\eps$ all $z_{\eps}$
remain in a compact set and since the supports of the $\chi_{i}$
form a locally finite family it follows that both $A_{\eps}$ and
$B_{\eps}$ are in fact finite sums for small $\eps$. To estimate
the summands in $A_{\eps}$, note that $z^{\circ}\in B_{3r_{x_{i}}/4}(x_{i})$
implies that $z\in B_{r_{x_{i}}}(x_{i})$, so $v^{x_{i}}(z)=f|_{B_{r_{x_{i}}}(x_{i})}(z)=f(z)=v^{z}(z)$.
Hence $\left[A_{\eps}\right]=0$. Concerning $B_{\eps}$, $z^{\circ}\not\in B_{3r_{x_{i}}/4}(x_{i})$
implies that $|z^{\circ}-x_{i}^{\circ}|>r_{x_{i}}/2$. On the other
hand, if $\chi_{i}(z_{\eps})\not=0$ then $z_{\eps}\in U_{i}\subseteq\Eball_{r_{x_{i}}/2}(x_{i}^{\circ})$,
implying $|z^{\circ}-x_{i}^{\circ}|\le r_{x_{i}}/2$. Hence $B_{\eps}=0$.
Consequently, $\left[f_{\eps}(z_{\eps})\right]=\left[v_{\eps}^{z}(z_{\eps})\right]$,
as claimed. 
\end{proof}
\noindent The following is the sheaf property for Fermat covers. 
\begin{cor}
\label{cor:sheafFermat}Let $X\subseteq\rti^{n}$, $Y\subseteq\rcrho^{d}$
and let $f:X\ra Y$ be a set-theoretical map. Suppose that $X\subseteq\bigcup_{j\in J}U_{j}$,
where each $U_{j}$ is a large open set, and that 
\[
f|_{U_{j}\cap X}\in\gsf(U_{j}\cap X,Y)
\]
for all $j\in J$. Then 
\begin{enumerate}
\item \label{enu:nearStX}If $X\subseteq\nrst{\left(\rti^{n}\right)}$,
then $f\in\gsf(X,Y)$. 
\item \label{enu:Xsubpoints}If $X$ contains its subpoints and all points
of $X$ are finite, then $f\in\gsf(X,Y)$. 
\end{enumerate}
\end{cor}

\begin{proof}
To prove property \ref{enu:nearStX}, let $s=\st{x}$, $x\in X$.
Then $x\in U_{j_{x}}$ for some $j_{x}\in J$, and hence $B_{r_{j_{x}}}(x)\subseteq U_{j_{x}}$
for some $r_{j_{x}}\in\R_{>0}$. Therefore $s=\st{x}\in\Eball_{r_{j_{x}}}(\st{x})$
and so $\st{X}\subseteq\bigcup_{x\in X}\Eball_{r_{j_{x}}}(\st{x})$.
Claim \ref{enu:nearStX} now follows directly from Cor.~\ref{thm:sheafFermatBalls}.

\noindent Property \ref{enu:Xsubpoints} follows by \ref{enu:nearStX}
and Thm.~\ref{thm:nearStdInfModerate} applied to $f|_{X'}$, where
$X':=\{x\in X\mid x\text{ is near-standard}\}$. 
\end{proof}
It is now natural to ask whether the sheaf property Thm.~\ref{thm:sheafFcnlCmpExh}
could be inscribed into the general notion of sheaf on a site. This
is one of the aims of the next Sec.~\ref{sec:Gtopos}.

\section{\textcolor{red}{\label{sec:Gtopos}}The Grothendieck topos of generalized
smooth functions}

As we argued in the introduction, function spaces and Cartesian closedness
are considered by many authors as important features for mathematics
and mathematical phy\-si\-cs. Even if Colombeau's theory of generalized
functions can be extended to any locally convex space $E$, on the
other hand, in \cite{KM} (p.\ 2) it is stated that: \emph{``locally
convex topology is not appropriate for non-linear questions in infinite
dimensions}'', and indeed\emph{ }a different approach to infinite
dimensional spaces is to embed smooth manifolds into a Cartesian closed
category $\mathcal{C}$ (see \cite{Gio10c} for a review of this type
of approaches). Similar lines of thought can be found in \cite{Ko-Re03,Ko-Re04},
but where generalized functions are seen as functionals, hence not
following Cauchy-Dirac's original conception but Schwartz' conception
instead. We first motivate and introduce the few notions of category
theory that we need in the present section. Indeed, only basic preliminaries
of category theory are needed to understand this section: definition
of category and basic examples, functors and natural transformation.
Our basic references for this section are \cite{MaMo,Joh02,Ba-Ho11}.
As it is customary, we write $D\in\mathbb{D}$ to denote that $D$
is an object of the category $\mathbb{D}$, we write $A\xra{f}B$
\emph{in} $\mathbb{D}$ to say that $f\in\mathbb{D}(A,B)$ and $\mathbb{D}^{\text{op}}$
for the opposite of $\mathbb{D}$ (see e.g.~\cite{Mac}). Only in
this section, we use both the notations $f\cdot g:=g\circ f$ for
arrows $X\xra{f}Y\xra{g}Z$ in some category, and the notation $\bar{\jmath}=\left(\bar{\jmath}_{\eps}\right)\in[J]$.

\subsection{Coverages, sheaves and sites}

The notion of coverage on a category allows one to define more abstractly
the concept of sheaf without being forced to consider a topological
space. Nevertheless, the classical example to keep in mind to have
a first understanding of the following definitions is a sheaf (e.g.~of
continuous functions) defined on the poset of open sets $\mathbb{D}=\mathbb{D}(X)$
in some topological space $X$.

We first define families with common codomain $D$: 
\begin{defn}
\label{def:famCommonCod}Let $\mathbb{D}$ be a category and let $D\in\mathbb{D}$.
Then we say that $\mathcal{F}\in\text{Fam}(D)$ is a \emph{family
with common codomain $D$} if there exist a set $J\in\Set$ and families
$\left(D_{j}\right)_{j\in J}$, $\left(i_{j}\right)_{j\in J}$ such
that: 
\begin{enumerate}
\item \label{enu:famCommonCodArrow}$D_{j}\xra{i_{j}}D$ in $\mathbb{D}$
for all $j\in J$. 
\item $\mathcal{F}=\left(D_{j}\xra{i_{j}}D\right)_{j\in J}$. 
\end{enumerate}
\end{defn}

\noindent A coverage is a class of families with a common codomain
that is closed with respect to pullback, in the precise sense stated
in the following 
\begin{defn}
\label{def:coverage}Let $\mathbb{D}$ be a category, then we say
that $\Gamma$ is a \emph{coverage on }$\mathbb{D}$ if: 
\begin{enumerate}
\item \label{enu:coverageArrow}$\Gamma:\text{Obj}(\mathbb{D})\ra\Set$,
where $\text{Obj}(\mathbb{D})$ is the class of objects of the category
$\mathbb{D}$. 
\item $\forall D\in\mathbb{D}:\ \Gamma(D)\subseteq\text{Fam}(D)$. Families
in $\Gamma(D)$ are called \emph{covering families of $D$}. 
\item \label{enu:covPullback}If $D\in\mathbb{D}$, $\left(D_{j}\xra{i_{j}}D\right)_{j\in J}\in\Gamma(D)$
is a covering family of $D$, and $C\xra{g}D$ is an arbitrary arrow
of $\mathbb{D}$, then there exists a covering family of $C$, $\left(C_{k}\xra{h_{k}}C\right)_{k\in K}\in\Gamma(C)$
such that 
\begin{equation}
\forall k\in K\,\exists j\in J\,\exists\bar{g}:\begin{array}{cc}
\xymatrix{C_{k}\ar[r]^{h_{k}}\ar[d]^{\bar{g}} & C\ar[d]^{g}\\
D_{j}\ar[r]^{i_{j}} & D
}
\end{array}\label{eq:pullbackCoverage}
\end{equation}
\item A pair $(\mathbb{D},\Gamma)$, of a category and a coverage on it,
is called a \emph{site}. 
\end{enumerate}
\end{defn}

\noindent For example, let $\Ogsf$ be the category of sharply open
sets $U\subseteq\rti^{u}$ (all possible dimensions $u\in\N$ are
included) and GSF. Let $\Gamma(U)$ contains open coverings and inclusions:
$\left(U_{j}\xra{i_{j}}U\right)_{j\in J}\in\Gamma(U)$ if and only
if $U_{j}\in\Ogsf$, $i_{j}:U_{j}\hookrightarrow U$ and $\bigcup_{j\in J}U_{j}=U$.
Then $(\Ogsf,\Gamma)$ is a site and property \eqref{eq:pullbackCoverage}
holds simply by taking $K=J$ and $C_{j}:=g^{-1}(U_{j})\in\Ogsf$
as covering family of $C$, and $\bar{g}:=g|_{C_{j}}$. Note that
these simple steps do not work in the category $\Sgsf$ of strongly
internal sets and GSF because in general $g^{-1}(U_{j})$ is not strongly
internal (only the inclusion $g^{-1}\left(\sint{A_{\eps}}\right)\subseteq\sint{g_{\eps}^{-1}(A_{\eps})}$
holds). In this case, a general method is to express the open set
$g^{-1}(U_{j})$ as a union of strongly internal sets. This implies
that we have to take a different index set $K$ for the covering family
$C_{k}\hookrightarrow C$.

Using the notion of coverage, we can define the notion of compatible
family: 
\begin{defn}
\label{def:compatible}Let $(\mathbb{D},\Gamma)$ be a site and $F:\mathbb{D}^{\text{op}}\ra\Set$
be a presheaf. Let $\mathcal{F}=\left(D_{j}\xra{i_{j}}D\right)_{j\in J}\in\Gamma(D)$
be a covering family of $D\in\mathbb{D}$. Then, we say that $(f_{j})_{j\in J}$
\emph{are compatible on $\mathcal{F}$ (rel.~$F$)} if the following
conditions hold: 
\begin{enumerate}
\item \label{enu:compSections}$f_{j}\in F(D_{j})$ for all $j\in J$. In
this case $f_{j}$ is called a \emph{section}. 
\item For all $g$, $c$ and $j$, $h\in J$, we have 
\begin{equation}
\begin{array}{cc}
\xymatrix{C\ar[r]^{c}\ar[d]^{g} & D_{h}\ar[d]^{i_{h}}\\
D_{j}\ar[r]^{i_{j}} & D
}
\end{array}\quad\Rightarrow\quad F(g)(f_{j})=F(c)(f_{h}).\label{eq:compInt}
\end{equation}
\end{enumerate}
\end{defn}

\noindent A typical way to apply \eqref{eq:compInt} is to construct
a sort of intersection object $C=D_{h}\cap D_{j}$ and to take as
$c$, $g$ the inclusions. Then, if $F=\mathbb{D}(-,Y)$, the equality
in \eqref{eq:compInt} reduces to the usual compatibility condition
$f_{j}|_{D_{h}\cap D_{j}}=f_{h}|_{D_{h}\cap D_{j}}$.

We can finally define the notion of sheaf on a site: 
\begin{defn}
\label{def:sheaf}Let $(\mathbb{D},\Gamma)$ be a site, then we say
that $F$ is a \emph{sheaf on }$(\mathbb{D},\Gamma)$, and we write
$F\in\text{Sh}(\mathbb{D},\Gamma)$ if 
\begin{enumerate}
\item $F:\mathbb{D}^{\text{op}}\ra\Set$ (i.e.~$F$ is a presheaf). 
\item If $\mathcal{F}=\left(D_{j}\xra{i_{j}}D\right)_{j\in J}\in\Gamma(D)$
is a covering family of $D\in\mathbb{D}$ and $(f_{j})_{j\in J}$
are compatible on $\mathcal{F}$ (rel.~$F$), then 
\begin{equation}
\exists!f\in F(D)\,\forall j\in J:\ F(i_{j})(f)=f_{j}.\label{eq:gluedSect}
\end{equation}
\end{enumerate}
\end{defn}

\noindent In the classical example of continuous functions on a topological
space, $F=\mathcal{C}^{0}(-,Y)$ and the equality in \eqref{eq:gluedSect}
becomes $f|_{D_{j}}=f_{j}$. Note that, even if the category of open
sets and GSF $\Ogsf$ is a site, example \ref{exa:i_revisited} shows
that $\Ogsf(-,Y)$ is not a sheaf.

\subsection{The category of glueable functions}

As we mentioned in the introduction to Sec.~\ref{sec:Sheaf-properties},
our main aim here is to show that the DCC is strictly related to the
aforementioned definition of sheaf on a site. The strategy we will
follow is: 
\begin{enumerate}[label=(\alph*)]
\item \label{enu:schemaBaseCat}Define a category $\gluable\supseteq\Sgsf$. 
\item \label{enu:schemaCov}Define a coverage on $\gluable$. 
\item \label{enu:schemaSheaf}Show that $\gluable(-,\mathcal{Y})$ is a
sheaf using Thm.~\ref{thm:sheafFcnlCmpExh} and hence the DCC. 
\end{enumerate}
\noindent Intuitively, we already think that a family of strongly
internal sets $\left(U_{j}\right)_{j\in J}$ is a coverage of the
strongly internal set $U$ if, simply, $U\subseteq\bigcup_{j\in J}U_{j}$;
we also intuitively think that $(f_{j})_{j\in J}$ are compatible
sections if DCC holds. The following definition reflects this intuition: 
\begin{defn}
\label{def:gluable}Let $\gluable$ be the category of \emph{glueable
families}, whose objects are non empty families $\left(U_{j}\right)_{j\in J}\in\gluable$
of strongly internal sets in some space $\rti^{u}$: 
\[
J\ne\emptyset,\ \exists u\in\N\,\forall j\in J:\ \rti^{u}\supseteq U_{j}\in\Sgsf.
\]
We say that 
\[
\mathcal{X}\xra{\phi}\mathcal{Y}\quad\text{in}\quad\gluable
\]
if $\mathcal{X}=\left(U_{j}\right)_{j\in J}$, $\mathcal{Y}=\left(V_{h}\right)_{h\in H}\in\gluable$
and $\phi=\left(\left(f_{j}\right)_{j\in J},\alpha\right)$, where: 
\begin{enumerate}
\item \label{enu:Gl-reparam}The map $\alpha\in\Set(J,H)$ is called a \emph{reparametrization}. 
\item \label{enu:Gl-DCC}The family of GSF $f_{j}\in\Sgsf(U_{j},V_{\alpha(j)})$,
$j\in J$, satisfies the DCC on $U:=\bigcup_{j\in J}U_{j}$. 
\end{enumerate}
\noindent To state condition \ref{enu:Gl-DCC} more explicitly, let
$u$, $v\in\N$ be the dimensions of $\left(U_{j}\right)_{j\in J}$
and $\left(V_{h}\right)_{h\in H}$ resp.~(i.e.~$\rti^{u}\supseteq U_{j}$
and $\rti^{v}\supseteq V_{h}$ for all $j$, $h$), and set $V:=\bigcup_{h\in H}V_{h}$.
Then \ref{enu:Gl-DCC} asks that there exists $\left(U_{j\eps}\right)_{\substack{j\in J\\
\eps\in I
}
}$ such that for all $K\fcmp U$ there exists $(f_{j\eps})_{\substack{j\in J\\
\eps\in I
}
}\in\mathcal{C}^{\infty}(\R^{u},\R^{v})$ such that: 
\begin{enumerate}[label=(ii.\alph*)]
\item \label{enu:Gl-reg1}$f_{j}=\left[f_{j\eps}(-)\right]|_{U_{j}}$ for
each $j\in J$. 
\item \label{enu:Gl-reg2}$\left[f_{\bar{\jmath}_{\eps},\eps}(-)\right]\in\gsf(U_{\bar{\jmath}}\cap K,V_{\bar{\jmath}\cdot\alpha}\cap V)$
for all $\bar{\jmath}=\left(\bar{\jmath}_{\eps}\right)\in[J]$, where
$U_{\bar{\jmath}}:=\sint{U_{\bar{\jmath}_{\eps},\eps}}$. Note that
$I\xra{\bar{\jmath}}J\xra{\alpha}H$ and hence $\bar{\jmath}\cdot\alpha=\alpha\circ\bar{\jmath}\in[H]$. 
\item \label{enu:Gl-glu}$\left[f_{\bar{\jmath}_{\eps},\eps}(-)\right]=\left[f_{\bar{h}_{\eps},\eps}(-)\right]$
on $U_{\bar{\jmath}}\cap U_{\bar{h}}\cap K$ for all $\bar{\jmath}$,
$\bar{h}\in[J]$. 
\end{enumerate}
Composition and identities in $\gluable$ are defined as follows:
Let 
\begin{equation}
\mathcal{X}\xra{\phi}\mathcal{Y}\xra{\psi}\mathcal{Z}\quad\text{in}\quad\gluable\label{eq:Gl-DefComp}
\end{equation}
and set $\mathcal{X}=\left(U_{j}\right)_{j\in J}$, $\mathcal{Y}=\left(V_{h}\right)_{h\in H}$,
$\mathcal{Z}=\left(W_{l}\right)_{l\in L}$, $\phi=\left(\left(f_{j}\right)_{j\in J},\alpha\right)$,
$\psi=\left(\left(g_{h}\right)_{h\in H},\beta\right)$. Then 
\begin{align*}
 & U_{j}\xra{f_{j}}V_{\alpha(j)}\xra{g_{\alpha(j)}}W_{\beta(\alpha(j))}\quad\forall j\in J\\
 & J\xra{\alpha}H\xra{\beta}L,
\end{align*}
and we hence set 
\begin{align*}
\phi\cdot\psi & :=\left(\left(f_{j}\cdot g_{\alpha(j)}\right)_{j\in J},\alpha\cdot\beta\right)\\
1_{\mathcal{X}} & :=\left(\left(1_{U_{j}}\right)_{j\in J},1_{J}\right).
\end{align*}
\end{defn}

\noindent The following lemma confirms the correctness of this definition. 
\begin{lem}
\label{lem:Gl-category}$\gluable$ is a category. 
\end{lem}

\begin{proof}
We essentially have to prove the closure with respect to composition,
i.e.~that \eqref{eq:Gl-DefComp} implies $\phi\cdot\psi:\mathcal{X}\ra\mathcal{Z}$
in $\gluable$. We implicitly use the notations of the previous definition.
For all $K=[K_{\eps}]\fcmp U$, we have $f_{\bar{\jmath}}:=\left[f_{\bar{\jmath}_{\eps},\eps}(-)\right]\in\gsf(U_{\bar{\jmath}}\cap K,V_{\bar{\jmath}\cdot\alpha}\cap V)$,
and hence for $\bar{\jmath}\in[J]$ we get $f_{\bar{\jmath}}\left(\left[U_{\bar{\jmath}_{\eps},\eps}\cap K_{\eps}\right]\right)=:\hat{K}\fcmp V_{\bar{\jmath}\cdot\alpha}\cap V$.
Using $\hat{K}$ and $\bar{\jmath}\cdot\alpha=:\bar{h}\in[H]$ with
the arrow $\psi$, we obtain $g_{\bar{h}}:=\left[g_{\bar{h}_{\eps},\eps}(-)\right]\in\gsf(V_{\bar{h}}\cap\hat{K},W_{\bar{h}\cdot\beta}\cap W)$
for some nets $\left(g_{h\eps}\right)_{h,\eps}$. Therefore 
\[
\left(f_{\bar{\jmath}}\cdot g_{\bar{\jmath}\cdot\alpha}\right)|_{U_{\bar{\jmath}}\cap K}=\left[g_{\alpha(\bar{\jmath}_{\eps}),\eps}\left(f_{\bar{\jmath}_{\eps},\eps}(-)\right)\right]\in\gsf(U_{\bar{\jmath}}\cap K,W_{\bar{\jmath}\cdot\alpha\cdot\beta}\cap W).
\]
This shows that condition \ref{enu:Gl-reg2} of Def.~\ref{def:gluable}
holds for $\phi\cdot\psi=\left(\left(f_{j}\cdot g_{\alpha(j)}\right)_{j\in J},\alpha\cdot\beta\right)$.
To prove condition \ref{enu:Gl-glu} take $x=[x_{\eps}]\in U_{\bar{\jmath}}\cap U_{\bar{l}}\cap K$,
then $f_{\bar{\jmath}}(x)=f_{\bar{l}}(x)\in V_{\bar{\jmath}\cdot\alpha}\cap V_{\bar{l}\cdot\alpha}\cap\hat{K}$
and hence $g_{\bar{\jmath}\cdot\alpha}\left(f_{\bar{\jmath}}(x)\right)=g_{\bar{l}\cdot\alpha}\left(f_{\bar{l}}(x)\right)$,
which is our conclusion. Properties of identities trivially hold. 
\end{proof}
Note that $\Sgsf\subseteq\gluable$ through the embedding: 
\begin{align*}
U & \in\Sgsf\mapsto\left(U\right)_{\bar{\mathbb{1}}}\in\gluable\\
f & \in\Sgsf(U,V)\mapsto\left(\left(f\right)_{\bar{\mathbb{1}}},\bar{\mathbb{1}}\ra\bar{\mathbb{1}}\right)\in\gluable\left(\left(U\right)_{\bar{\mathbb{1}}},\left(V\right)_{\bar{\mathbb{1}}}\right),
\end{align*}
where $\bar{\mathbb{1}}:=\{*\}$ is any singleton set. The converse
is also possible using the sheaf Thm.~\ref{thm:sheafFcnlCmpExh}:
In fact, if $\left(\left(f_{j}\right)_{j\in J},\alpha\right)\in\gluable\left(\left(U_{j}\right)_{j\in J},\left(V_{h}\right)_{h\in H}\right)$,
then the DCC holds and hence there exists a unique $f\in\Sgsf(U,V)$
such that $f|_{U_{j}}=f_{j}$ for all $j\in J$. If we set $\text{gl}\left(\left(f_{j}\right)_{j\in J},\alpha\right):=f$
then 
\begin{align*}
\text{gl}\left(\phi\cdot\psi\right) & =\text{gl}\left(\left(f_{j}\cdot g_{\alpha(j)}\right)_{j\in J},\alpha\cdot\beta\right)=\text{gl}\left(\phi\right)\cdot\text{gl}\left(\psi\right)
\end{align*}
because setting $\text{gl}(\phi)=:f$ and $\text{gl}(\psi)=:g$, we
have $\left(f\cdot g\right)|_{U_{j}}=f|_{U_{j}}\cdot g|_{V_{\alpha(j)}}=f_{j}\cdot g_{\alpha(j)}$,
i.e.~the unique GSF obtained by gluing $\left(f_{j}\cdot g_{\alpha(j)}\right)_{j\in J}$
is $f\cdot g$. Finally $\text{gl}\left(1_{\mathcal{X}}\right)=1_{U}=1_{\text{gl}(\mathcal{X})}$
and hence $\text{gl}:\gluable\ra\Sgsf$ is a (clearly non injective)
functor.

\subsection{Coverage of glueable functions}

We now introduce a coverage on the category $\gluable$ of glueable
families: 
\begin{defn}
\label{def:covGlu}Let $\mathcal{E}=\left(W_{e}\right)_{e\in E}\in\gluable$,
then we say that $\gamma\in\Gamma(\mathcal{E})$ if there exists a
non empty $J\in\Set$ such that: 
\begin{enumerate}
\item \label{enu:covGluArr}$\gamma=\left(\gamma_{j}\right)_{j\in J}$ and
$\gamma_{j}=\left(\left(i_{h}\right)_{h\in J},\delta\right)$ for
all $j\in J$. 
\item \label{enu:covGluRef}$J\xra{\delta}E$ is a surjective map. 
\item \label{enu:covGluIncl}$\xymatrix{\!\!\!\!i_{j}:D_{j}\ar@{^{(}->}[r] & W_{\delta(j)}}
$for all $j\in J$, where $\left(D_{j}\right)_{j\in J}\in\gluable$
. Because of this property, $\delta$ is called a \emph{refinement
map}. 
\item \label{enu:covGluProj}$W_{e}=\bigcup\left\{ D_{j}\mid\delta(j)=e,\ j\in J\right\} $
for all $e\in E$. 
\end{enumerate}
\end{defn}

\begin{rem}
\noindent \label{rem:covGlu} 
\begin{enumerate}
\item \label{enu:indSet}Note that the index set $J$ of $\gamma=\left(\gamma_{j}\right)_{j\in J}$
is the same used in $\left(i_{h}\right)_{h\in J}$. Moreover, the
two components of $\gamma_{j}$ do not depend on $j\in J$. This may
appear to be a strange property for a coverage, but note that our
intuition here is guided by viewing the inclusions$\xymatrix{(D_{j}\ar@{^{(}->}[r]^{i_{j}} & W)_{j\in J}}
$as a coverage of $W:=\bigcup_{e\in E}W_{e}$. On the other hand, note
that in Def.~\ref{def:gluable} of glueable families, in the DCC
\ref{enu:Gl-DCC} we need the whole family $\left(f_{j}\right)_{j\in J}$
and not only the single GSF $f_{j}$. Similarly, we need to consider
the entire family of inclusions $\left(i_{h}\right)_{h\in J}$ and
not the single $i_{h}$. For this reason, we cannot directly consider
the family $\xymatrix{(D_{j}\ar@{^{(}->}[r]^{i_{j}} & W)_{j\in J},}
$made of single inclusions, as a coverage. 
\item \label{enu:covGluCov}Condition \ref{enu:covGluProj} implies $W=\bigcup_{j\in J}D_{j}=\bigcup_{j\in J}W_{\delta(j)}$.
We will see more clearly later that this condition allows us to prove
the uniqueness part of \eqref{eq:gluedSect}. 
\item If $W_{e}\ne\emptyset$ for all $e\in E$, then \ref{enu:covGluProj}
directly implies that $\delta$ has to be a surjective map. 
\end{enumerate}
\end{rem}

\begin{thm}
\noindent \label{thm:covGlu}$\Gamma$ is a coverage on $\gluable$. 
\end{thm}

\begin{proof}
From Def.~\ref{def:covGlu}.\ref{enu:covGluIncl} it follows that
$\gamma_{j}\in\gluable\left(\left(D_{j}\right)_{j\in J},\mathcal{E}\right)$,
i.e.~property Def.~\ref{def:famCommonCod}.\ref{enu:famCommonCodArrow}.
To prove the closure with respect to pullbacks, take $\eta\in\gluable(\mathcal{C},\mathcal{E})$,
where $\mathcal{C}=:\left(V_{c}\right)_{c\in C}$ and $\eta=:\left(\left(g_{c}\right)_{c\in C},\beta\right)$.
Since $\left(g_{c}\right)_{c\in C}$ satisfies the DCC (see Def.~\ref{def:gluable}.\ref{enu:Gl-DCC},
i.e.~Def.~\ref{def:gluable}.\ref{enu:Gl-reg1}, \ref{enu:Gl-reg2}),
we can use Thm.~\ref{thm:sheafFcnlCmpExh} to get 
\[
\exists!g\in\gsf(V,W)\,\forall c\in C:\ g|_{V_{c}}=g_{c},
\]
where $V:=\bigcup_{c\in C}V_{c}$. We can hence consider $g^{-1}(D_{j})$
and cover it with strongly internal sets: 
\[
\forall j\in J\,\exists H_{j}\ne\emptyset\,\exists\left(B_{jh}\right)_{h\in H_{j}}\forall h\in H_{j}:\ B_{jh}\in\Sgsf,\ g^{-1}(D_{j})=\bigcup_{h\in H_{j}}B_{jh}.
\]
Set $B_{jhc}:=B_{jh}\cap V_{c}\in\Sgsf$, $K:=\left\{ (j,h,c)\mid j\in J,\ h\in H_{j},\ c\in C\right\} $,
$\nu:(j,h,c)\in K\mapsto c\in C$,$\xymatrix{a_{k}:B_{k}\ar@{^{(}->}[r] & V_{\nu(k)},}
$ and $\alpha_{k}:=\left(\left(a_{k}\right)_{k\in K},\nu\right)$.
Then $K$ is non empty because $C$, $H_{j}$, $J\ne\emptyset$, and
we have 
\begin{align*}
B & =\bigcup_{k\in K}B_{k}=\bigcup_{c\in C}\bigcup_{j\in J}\bigcup_{h\in H_{j}}B_{jhc}=\bigcup_{c\in C}V_{c}\cap\bigcup_{j\in J}\bigcup_{h\in H_{j}}B_{jh}=V\cap\bigcup_{j\in J}g^{-1}(D_{j})=\\
 & =V\cap g^{-1}\left(\bigcup_{j\in J}D_{j}\right)=V\cap g^{-1}(D)=V\cap g^{-1}(W)=V.
\end{align*}
To prove property \ref{enu:covGluProj} of Def.~\ref{def:covGlu}
for the new covering family $\left(\alpha_{k}\right)_{k\in K}$, take
$x\in V_{c}\subseteq V$, so that $g(x)\in W$. Thereby, $x\in g^{-1}(D_{j})$
for some $j\in J$, and hence $x\in B_{jh}\subseteq B_{jhc}$ for
some $h\in H_{j}$. Setting $k:=(j,h,c)\in K$, we have $\nu(k)=c$
and $x\in B_{k}$. This shows that $\left(\alpha_{k}\right)_{k\in K}\in\Gamma(\mathcal{C})$.
Finally, let $\left(\delta\right)_{\text{l}}^{-1}$ be any left inverse
of $\delta$, i.e.~$\left(\delta\right)_{\text{l}}^{-1}\cdot\delta=\delta\circ\left(\delta\right)_{\text{l}}^{-1}=1_{E}$
(recall Def.~\ref{def:covGlu}.\ref{enu:covGluRef}), and setting
$\mathcal{B}_{k}:=\left(B_{k}\right)_{k\in K}$, $\mu:=\left(\left(g_{k}|_{B_{k}}\right)_{k\in K},\nu\cdot\beta\cdot\left(\delta\right)_{\text{l}}^{-1}\right)$,
we have 
\[
\forall k\in K\,\exists j\in J\,\exists\mu:\begin{array}{cc}
\xymatrix{\mathcal{B}_{k}\ar[r]^{\alpha_{k}}\ar[d]^{\mu} & \mathcal{C}\ar[d]^{\eta}\\
\mathcal{D}_{j}\ar[r]^{\gamma_{j}} & \mathcal{E}
}
\end{array}
\]
i.e.~the claim Def.~\ref{def:coverage}.\ref{enu:covPullback}. 
\end{proof}
\begin{defn}
\label{def:topos}The category of sheaves $\Tgsf:=\text{Sh}\left(\gluable,\Gamma\right)$
(and natural transformations as arrows) is called the \emph{Grothendieck
topos of generalized smooth functions} (see e.g.~\cite{MaMo,Joh02,Ba-Ho11}
and references therein). 
\end{defn}

\subsection{The sheaf of glueable functions}

We are now able to show that the DCC is the key property to prove
the following 
\begin{thm}
\label{thm:sheafGlueable}For each $\mathcal{Y}\in\gluable$, the
functor $\gluable(-,\mathcal{Y})$ is a sheaf on the site $\left(\gluable,\Gamma\right)$,
i.e.~it satisfies Def.~\ref{def:sheaf}: $\gluable(-,\mathcal{Y})\in\Tgsf\left(\gluable,\Gamma\right)$. 
\end{thm}

\begin{proof}
We use the notations of Def.~\ref{def:covGlu}. Let $\left(\gamma_{j}\right)_{j\in J}=\left(\left(i_{h}\right)_{h\in J},\delta\right)_{j\in J}\in\Gamma(\mathcal{E})$
be a covering family and let $\phi_{j}=\left(\left(f_{h}^{j}\right)_{h\in J},\alpha^{j}\right)\in\gluable(\mathcal{D}_{j},\mathcal{Y})$,
$j\in J$, be a compatible family of sections, where $\mathcal{D}_{j}:=\left(D_{h}\right)_{h\in J}=:\mathcal{D}_{0}$
(recall Rem.~\ref{rem:covGlu}.\ref{enu:indSet} about the independence
from $j\in J$). Note explicitly that by Def.~\ref{def:covGlu}.\ref{enu:covGluArr}
the covering family $\left(\gamma_{j}\right)_{j\in J}$ is indexed
by the same set $J$ as its inclusions $\left(i_{h}\right)_{h\in J}$;
moreover, the glueable family $\left(f_{h}^{j}\right)_{h\in J}$ is
also indexed by $J$ because by Def.~\ref{def:gluable}, any arrow
in the category $\gluable$ is indexed by the same set of its domain
which, in this case, is $\mathcal{D}_{0}=\left(D_{h}\right)_{h\in J}$.
Set $\mathcal{Y}=:\left(V_{l}\right)_{l\in L}$. We first want to
prove that the compatibility of sections $\left(\phi_{j}\right)_{j\in J}$
allows us to show that both $\left(f_{h}^{j}\right)_{h\in J}$ and
$\alpha^{j}$ do not actually depend on $j$. We therefore take $i\in J$
and define $\mathcal{D}_{0}\cap\mathcal{D}_{0}:=\left(D_{h}\cap D_{k}\right)_{(h,k)\in J^{2}}\in\gluable$,
$\xymatrix{i_{hk}:D_{h}\cap D_{k}\ar@{^{(}->}[r] & D_{h},}
$ $\xymatrix{i_{kh}:D_{h}\cap D_{k}\ar@{^{(}->}[r] & D_{k},}
$$\nu_{1}:(h,k)\in J^{2}\mapsto h\in J$, $\nu_{2}:(h,k)\in J^{2}\mapsto k\in J$,
$\iota_{1}:=\left(\left(i_{hk}\right)_{(h,k)\in J^{2}},\nu_{1}\right)$,
$\iota_{2}:=\left(\left(i_{kh}\right)_{(h,k)\in J^{2}},\nu_{2}\right)$.
The compatibility condition \eqref{eq:compInt} for the functor $\gluable(-,\mathcal{Y})$
yields 
\begin{align*}
\,\gluable(-,\mathcal{Y})(\iota_{1})(\phi_{j}) & =\gluable(-,\mathcal{Y})(\iota_{2})(\phi_{i})\\
\left(\left(i_{hk}\cdot f_{\nu_{1}(h,k)}^{j}\right)_{(h,k)\in J^{2}},\nu_{1}\cdot\alpha^{j}\right) & =\left(\left(i_{kh}\cdot f_{\nu_{2}(h,k)}^{i}\right)_{(h,k)\in J^{2}},\nu_{2}\cdot\alpha^{i}\right)\\
\left(\left(f_{h}^{j}|_{D_{h}\cap D_{k}}\right)_{hk},\nu_{1}\cdot\alpha^{j}\right) & =\left(\left(f_{k}^{i}|_{D_{h}\cap D_{k}}\right)_{hk},\nu_{2}\cdot\alpha^{i}\right).
\end{align*}
Thereby, for $h=k$ we get $f_{h}^{j}=f_{h}^{i}$ and, for arbitrary
$h$, $k\in J$, we also have $(\nu_{1}\cdot\alpha^{j})(h,k)=(\nu_{2}\cdot\alpha^{i})(h,k)$,
i.e.~$\alpha^{j}(h)=\alpha^{i}(k)$. This equality, since $J\ne\emptyset$,
implies that $\alpha^{i}=\alpha^{j}=:\alpha_{\text{c}}\in L$ is constant.
Therefore, as a consequence of the compatibility condition, we have
that both components of $\phi_{j}$ do not depend on $j\in J$: our
sections can hence be simply written as $\phi_{j}=:\left(\left(f_{h}\right)_{h\in J},\alpha_{\text{c}}\right)$.
Note also that all the GSF $f_{h}:D_{h}\ra V:=V_{\alpha_{\text{c}}}$
have the same codomain. The glueable family $\left(f_{h}\right)_{h\in J}$
satisfies the DCC because of Def.~\ref{def:gluable}.\ref{enu:Gl-DCC}
and we can hence apply Thm.~\ref{thm:sheafFcnlCmpExh} to obtain
a unique $f\in\gsf(D,V)$ such that $f|_{D_{j}}=f_{j}$ for each $j\in J$,
where $D=\bigcup_{j\in J}D_{j}$. We can finally set $\bar{\alpha}:e\in E\mapsto\alpha_{\text{c}}\in L$
and $\phi:=\left(\left(f|_{W_{e}}\right)_{e\in E},\bar{\alpha}\right)$
to obtain the existence part of the conclusion: 
\[
\gluable(-,\mathcal{Y})(\gamma_{j})(\phi)=\gamma_{j}\cdot\phi=\left(\left(f|_{W_{\delta(h)}\cap D_{h}}\right)_{h\in J},\delta\cdot\bar{\alpha}\right)=\left(\left(f|_{D_{h}}\right)_{h\in J},\alpha\right)=\phi_{j}.
\]
To prove the uniqueness of the glued section, assume that $\hat{\phi}=\left(\left(\hat{f}_{e}\right)_{e\in E},\hat{\alpha}\right)\in\gluable(\mathcal{E},\mathcal{Y})$
is another section such that $\gluable(-,\mathcal{Y})(\gamma_{j})(\hat{\phi})=\phi_{j}$
for all $j\in J$. This equality gives 
\begin{equation}
\left(\left(\hat{f}_{\delta(h)}|_{D_{h}}\right)_{h\in J},\delta\cdot\hat{\alpha}\right)=\left(\left(f_{h}\right)_{h\in J},\alpha\right)\label{eq:uniqueness}
\end{equation}
Take $e\in E$ and $x\in W_{e}$, then condition Def.~\ref{def:covGlu}.\ref{enu:covGluProj}
yields the existence of $h\in J$ such that $\delta(h)=e$ and $x\in D_{h}$.
Therefore, using \eqref{eq:uniqueness} we obtain 
\[
\hat{f}_{e}(x)=\hat{f}_{\delta(h)}|_{D_{h}\cap W_{e}}(x)=f_{h}|_{W_{e}}(x)=f|_{W_{e}}(x),
\]
i.e.~$\hat{f}_{e}=f|_{W_{e}}$ for all $e\in E$. Finally, \eqref{eq:uniqueness}
also gives $\hat{\alpha}=\left(\delta\right)_{\text{l}}^{-1}\cdot\alpha$. 
\end{proof}

\subsection{\label{subsec:Concrete-sites}Concrete sites and generalized diffeological
spaces}

In this final section, we want to sketch one of the many possibilities
that we can start to explore using the Grothendieck topos $\Tgsf$.
The idea is to show that the site $(\gluable,\Gamma)$ is a concrete
site. In this way, considering the space of concrete sheaves over
this site, we get a category of spaces that extends usual smooth manifolds
but it is closed with respect to operations such as: arbitrary subspaces,
products, sums, function spaces, etc.~(see \cite{Igl,Koc,Ba-Ho11,Gio09,Gio10c}
and references therein for similar approaches). As above, all the
necessary categorical notions will be introduced). 
\begin{defn}
\label{def:reprSubcan}Let $\mathbb{D}$ be a category and $F:\mathbb{D}^{\text{op}}\ra\Set$
be a functor. Then, we say that $F$ \emph{is representable} if $F\simeq\mathbb{D}(-,D)$
for some $D\in\mathbb{D}$. Moreover, if $(\mathbb{D},\Gamma)$ is
a site, then we say that $(\mathbb{D},\Gamma)$ \emph{is a subcanonical
site} if every representable functor $F$ is a sheaf $F\in\text{Sh}(\mathbb{D},\Gamma)$.

\noindent Since in Thm.~\ref{thm:sheafGlueable} we proved that $\gluable(-,\mathcal{Y})$
is a sheaf, directly from Def.~\ref{def:reprSubcan} it follows that
$\left(\gluable,\Gamma\right)$ is a subcanonical site. 
\end{defn}

\noindent A concrete site is a site whose objects can be thought of
as an underlying set with a structure. The idea is that if $\mathbb{1}\in\mathbb{D}$
is a terminal object, then $|D|:=\mathbb{D}(\mathbb{1},D)\in\Set$
is the underlying set of $D\in\mathbb{D}$ and if $f:C\ra D$ in $\mathbb{D}$,
then $|f|:=\mathbb{D}(\mathbb{1},f):x\in|C|\mapsto x\cdot f=f\circ x\in|D|$
is the set-theoretical map corresponding to the arrow $f$. These
maps have a natural relation with covering families of $\Gamma$,
as stated in the following 
\begin{defn}
\label{def:concreteSite}We say that $(\mathbb{D},\Gamma,\mathbb{1})$
\emph{is a concrete site} if: 
\begin{enumerate}
\item $(\mathbb{D},\Gamma)$ is a subcanonical site. 
\item $\mathbb{1}\in\mathbb{D}$ is a terminal object, i.e.~$\mathbb{1}\in\mathbb{D}$
and $\forall D\in\mathbb{D}\,\exists!t\in\mathbb{D}(D,\mathbb{1})$.
\end{enumerate}
$\mathbb{D}(\mathbb{1},D)=:|D|$ is called the \emph{underlying set
of $D\in\mathbb{D}$}. For $f\in\mathbb{D}(C,D)$, the map $|f|:=\mathbb{D}(\mathbb{1},f):|C|\ra|D|$
is called the \emph{function associated to the morphism $f$.} 
\begin{enumerate}[resume]
\item \label{enu:concSiteAssFnct}The functor $\mathbb{D}(\mathbb{1},-):\mathbb{D}\ra\Set$
is faithful, i.e.~for all $f$, $g\in\mathbb{D}(C,D)$, if $|f|=|g|$,
then $f=g$. 
\item \label{enu:concSiteCov}If $\left(D_{j}\xra{i_{j}}D\right)_{j\in J}\in\Gamma(D)$,
then the associated maps trivially cover $|D|$, i.e.: 
\begin{equation}
\bigcup_{j\in J}|i_{j}|\left(|D_{i}|\right)=|D|.\label{eq:trivCov}
\end{equation}
\end{enumerate}
\end{defn}

\noindent For example, let us define a terminal object in the category
$\gluable$ of glueable spaces as: 
\[
\mathbb{1}:=\left(\{0\}\right)_{\bar{\mathbb{1}}}\in\gluable
\]
where $\bar{\mathbb{1}}=\{*\}$. Note that, if we view $\R^{n}=\Set\left(\left\{ 1,\ldots,n\right\} ,\R\right)$,
then $\text{Card}\left(\R^{0}\right)=1$ and hence $\text{Card}\left(\rti^{0}\right)=\text{Card}\left(\left(\R^{0}\right)^{I}/\sim_{\rho}\right)=1$.
Therefore, $\rti^{0}=\{0\}$ is the trivial ring. It is also a strongly
internal set because $B_{1}(0)=\left\{ x\in\rti^{0}\mid|x-0|<1\right\} =\{0\}$.

\noindent What is $\phi\in\gluable(\mathbb{1},\mathcal{X})=|\mathcal{X}|$
in this case? Set $\phi=\left(\left(f\right)_{\bar{\mathbb{1}}},\alpha\right)$
and $\mathcal{X}=\left(U_{j}\right)_{j\in J}$, then $\alpha:\bar{\mathbb{1}}\ra J$
and $f:\{0\}\ra U_{\alpha(*)}$, which can be identified with the
pair $(f(0),\alpha(*))\in U_{\alpha(*)}\times\{\alpha(*)\}$. Therefore,
\[
|\mathcal{X}|=\gluable(\mathbb{1},\mathcal{X})\simeq\sum_{j\in J}U_{j}=\bigcup_{j\in J}U_{j}\times\{j\}.
\]
Similarly, $\psi=\left(\left(f_{j}\right)_{j\in J},\alpha\right)\in\gluable\left(\left(U_{j}\right)_{j\in J},\left(V_{l}\right)_{l\in L}\right)$
can be identified with the map: 
\[
|\psi|:(x,j)\in\sum_{j\in J}U_{j}\mapsto\left(f_{j}(x),\alpha(j)\right)\in\sum_{l\in L}V_{l},
\]
i.e.~with the map $(x,j)\mapsto\left(f(x),\alpha(j)\right)$, where
$f\in\gsf\left(\bigcup_{j\in J}U_{j},\bigcup_{l\in L}V_{l}\right)$
is obtained by gluing $\left(f_{j}\right)_{j\in J}$. This implies
condition Def.~\ref{def:concreteSite}.\ref{enu:concSiteAssFnct},
whereas Def.~\ref{def:concreteSite}.\ref{enu:concSiteCov} follows
from Rem.~\ref{rem:covGlu}.\ref{enu:covGluCov}: 
\begin{thm}
\label{thm:GluConcSite}$\left(\gluable,\Gamma\right)$ is a concrete
site. 
\end{thm}

\noindent It is well known that a sheaf $F\in\text{Sh}(\mathbb{D},\Gamma)$
can be thought of as a generalized space defined by the information
$F(D)\in\Set$ associated to each test space $D\in\mathbb{D}$. The
idea of \emph{concrete sheaf} is that it is this kind of generalized
space defined by an underlying set of points $F(\mathbb{1})$. For
example, any $y\in\gluable(\mathbb{1},\mathcal{Y})$ can be identified
with the map $\sum_{0}\{0\}\times\{0\}=\{(0,0)\}\ra\sum_{l\in L}V_{l}$,
and hence $\gluable(\mathbb{1},\mathcal{Y})\simeq\sum_{l\in L}V_{l}$. 
\begin{defn}
\label{def:concSheaf}Let $(\mathbb{D},\Gamma,\mathbb{1})$ be a concrete
site. Then we say that $F$ \emph{is a concrete sheaf} (and we write
$F\in\text{CSh}(\mathbb{D},\Gamma,\mathbb{1})$) if: 
\begin{enumerate}
\item $F\in\text{Sh}(\mathbb{D},\Gamma)$. 
\item For all $s\in F(D)$, let $\text{\ensuremath{\underline{s}}}:p\in|D|\mapsto F(p)(s)\in F(\ensuremath{\mathbb{1})}$,
then we have 
\[
\forall D\in\mathbb{D}\,\forall s,t\in F(D):\ \underline{s}=\underline{t}\ \Rightarrow\ s=t.
\]
\end{enumerate}
\end{defn}

\noindent Similarly to what we did above, we can prove that $\gluable(-,\mathcal{Y})$
is a concrete sheaf.

\section{\label{sec:FuturePerspectives}Conclusions and future perspectives}

Sobolev and Schwartz solved the problem ``how to derive continuous
functions?''. Also Sebastiao e Silva (see \cite{Seb54}) solved the
same problem without relying on functional analysis at all, but instead
using only a formal approach and arriving at an isomorphic solution.
We solved the problem: ``how to derive continuous functions obtaining
set-theoretical functions, unrestrictedly composable, extending the
usual classical theorems of calculus and allowing for inifinitesimal
and infinite values?''. 
This second problem doesn't appear to have a trivial formal solution.

We have shown that GSF theory has features that closely resemble classical
smooth functions. In contrast, some differences have to be carefully
considered, such as the fact that the new ring of scalars $\rti$
is not a field, it is not totally ordered, it is not order complete,
so that its theory of supremum and infimum is more involved (see \cite{MTAG20}),
and its intervals are not connected in the sharp topology because
the set of all the infinitesimals is a clopen set. Almost all these
properties are necessarily shared by other non-Archimedean rings because
their opposites are incompatible with the existence of infinitesimal
numbers.

Conversely, the ring of Robinson-Colombeau generalized numbers $\rti$
is a framework where the use of infinitesimal and infinite quantities
is available, it is defined using elementary mathematics, and with
a strong connection with infinitesimal and infinite functions of classical
analysis. As proved in \cite{FGBL}, this leads to a better understanding
and opens the possibility to define new models of physical systems.
We can hence state that GSF theory is potentially a good framework
for mathematical physics.

As we started to see in Sec.~\ref{subsec:Concrete-sites}, the category
of concrete sheaves over the concrete site of gluable families contains
the category of strongly open sets and GSF and hence, also the category
of ordinary smooth functions on open sets. In future works, we will
build on this and show that it also contains the category of smooth
manifolds (more generally all diffeological spaces). This opens the
possibility to study singular differential geometry using non-Archimedean
methods and, as is typical of topos theory, interesting connections
with logic.

Finally, as we will see in the next two papers of this series (\cite{Lu-Gi16,GiLu}),
GSF theory is also an interesting non-Archimedian framework for the
mathematical analysis of singular non-linear ordinary and partial
differential equations.

\newpage{}

\end{document}